\documentclass[12pt]{article}
\usepackage{algorithm}
\usepackage{algorithmic}

\usepackage[utf8]{inputenc}

\usepackage{graphicx}
\usepackage{bm}
\usepackage{comment}
\usepackage{psfrag}
\usepackage{latexsym,amsmath,amsfonts,amscd,amsthm}
\usepackage{changebar}
\usepackage{color}
\usepackage{bm}
\usepackage{tikz}
\usepackage{multirow}
\usepackage{xcolor}
\usepackage{hyperref}
\usepackage{setspace}
\usepackage{amssymb}
\usepackage{mathrsfs}
\usepackage{caption}
\usepackage{subcaption}
\usepackage[title]{appendix}
\usepackage{ulem}
\usepackage{tikz}
\usetikzlibrary{arrows,backgrounds}

\newlength{\spse}
\newtheorem{thm}{Theorem}[section]

\newtheorem{defn}[thm]{Definition}
\newtheorem{rem}[thm]{Remark}

\newcommand{\Dt}{\Delta t}
\newcommand{\vareps}{\varepsilon}

\newcommand{\bx}{\mathbf{x}}
\newcommand{\bv}{\boldsymbol{v}}
\newcommand{\bnu}{\boldsymbol{\mu}}

\newcommand{\lgl}{\langle}
\newcommand{\rgl}{\rangle}
\newcommand{\half}{\frac{1}{2}}
\newcommand{\brho}{\boldsymbol{\rho}}
\newcommand{\bg}{\boldsymbol{g}}
\newcommand{\bc}{\boldsymbol{c}}
\newcommand{\bG}{\boldsymbol{G}}
\newcommand{\bff}{\boldsymbol{f}}
\newcommand{\bbeta}{\boldsymbol{\beta}}

\newcommand{\mcE}{\mathcal{E}}
\newcommand{\mcR}{\mathcal{R}}
\newcommand{\mcV}{\mathcal{V}}
\newcommand{\mcT}{\mathcal{T}}

\title{A micro-macro decomposed reduced basis method for the time-dependent radiative transfer equation}

\author{
Zhichao Peng\thanks{Department of Mathematics, Michigan State University, East Lansing, MI 48824 U.S.A. Email: {\tt pengzhic@msu.edu}.}
\and
Yanlai Chen\thanks{Department of Mathematics, University of Massachusetts Dartmouth, 285 Old Westport Road, North Dartmouth, MA 02747, USA. Email: {\tt{yanlai.chen@umassd.edu}}. Research is partially supported by National Science Foundation grant DMS-2208277, and by the UMass Dartmouth Marine and UnderSea Technology (MUST) Research Program made possible via an Office of Naval Research grant N00014-20-1-2849.}
\and
Yingda Cheng \thanks{Department of Mathematics, Department of  Computational Mathematics, Science and Engineering, Michigan State University,
East Lansing, MI 48824 U.S.A.
Email: {\tt ycheng@msu.edu}. Research is supported by NSF grants    DMS-2011838 and AST-2008004.}
\and
Fengyan Li\thanks{Department of Mathematical Sciences, Rensselaer Polytechnic Institute, Troy, NY 12180, U.S.A. Email: 
{\tt lif@rpi.edu}. Research is supported by NSF grant DMS-1913072.}}

\date{}

\usetikzlibrary{shapes.geometric, arrows}
\usetikzlibrary[calc]
\usetikzlibrary{positioning}
\usetikzlibrary{fit}% used for the efficient working of the positioning system  
\usetikzlibrary{arrows.meta}
\tikzset{block_device_small/.style={draw, thick, text width=1.5cm, minimum height=1.0cm, align=left,fill=cyan},   
}

\tikzset{block_device/.style={draw, thick, text width=4.5cm, minimum height=1.5cm, fill={cyan},  align=center}}
\tikzset{block_laptop/.style={draw, thick, text width=4cm, minimum height=1.5cm, fill={pink},  align=center}}
\tikzstyle{arrow} = [thick,-Stealth]
\tikzstyle{line} = [thick,-]
\tikzstyle{square-node} = [rectangle, rounded corners, minimum width=3cm, minimum height=1cm,text centered, draw=black, , fill=red!30]

\tikzset{
    *|/.style={
        to path={
            (perpendicular cs: horizontal line through={(\tikztostart)},
                                 vertical line through={(\tikztotarget)})
            -- (\tikztotarget) \tikztonodes
        }
    }
}

\usepackage{sansmath}

\begin{document}

\maketitle
%\pzc{(1) Update online, (2) update numerical section, (3) flowchart}
\abstract{
Kinetic transport equations are notoriously difficult to simulate  because of their complex multiscale behaviors and the   need to numerically resolve a high dimensional probability density function. Past literature has focused on building reduced order models (ROM) by analytical methods. In recent years, there is a surge of interest in developing ROM using data-driven or computational tools that offer more applicability and flexibility. This paper is a work towards that direction.

 Motivated by our previous work of designing ROM for the stationary radiative transfer equation in \cite{peng2022reduced} by leveraging  the low-rank structure of the solution manifold induced by the angular variable,  we here further advance the methodology to the time-dependent model. 
Particularly, we take the celebrated reduced basis method (RBM) approach and  propose a novel micro-macro decomposed reduced basis method (MMD-RBM). The MMD-RBM is constructed by exploiting, in a  greedy fashion, the low-rank structures of both the micro- and macro-solution manifolds with respect to   the angular and temporal variables. Our reduced order surrogate consists of: reduced bases for reduced order subspaces and a   reduced quadrature rule in the angular space.
The proposed MMD-RBM features several structure-preserving components: 1)  an equilibrium-respecting strategy to construct reduced order subspaces which better utilize
the structure of the decomposed system, and  2) a recipe for preserving positivity of the quadrature weights thus to maintain the stability of   the underlying reduced solver. 
The resulting ROM can be used to achieve a fast online solve for the angular flux 
in angular directions outside the training set and  for arbitrary order moment of the angular flux.

We perform benchmark test problems in 2D2V, and the numerical tests show that the MMD-RBM can capture the low rank structure effectively when it exists. A careful study in the computational cost shows that the offline stage of the  MMD-RBM is more efficient than the proper orthogonal decomposition (POD) method, and in the low rank case, it even outperforms a standard full order solve. Therefore, the proposed MMD-RBM can be seen  both as a surrogate builder and a low-rank solver at the same time. Furthermore, it can be readily incorporated into multi-query scenarios to accelerate problems arising from uncertainty quantification, control, inverse problems and optimization. 
}

%%%%%%%%%%%%%%%%%%%%%%%%%%%%%%%%%%%%%%
% Introduction 
%%%%%%%%%%%%%%%%%%%%%%%%%%%%%%%%%%%%%%
\section{Introduction}
In this paper, we design a reduced order model (ROM) for a class of kinetic transport equation: the time-dependent radiative transfer equation (RTE), which provides prototype models for   optical tomography \cite{arridge2009optical}, radiative transfer \cite{pomraning1973equations}, remote sensing \cite{spurr2001linearized} and neutron transport \cite{lewis1984computational} etc.
The isotropic time-dependent RTE under the diffusive scaling is written as:
\begin{align}
    \vareps\partial_t f+\bv\cdot\nabla_{\bx} f = \frac{\sigma_s}{\vareps}(\lgl f\rgl-f)-\vareps\sigma_a f +  \vareps G.
    \label{eq:kinetic}
\end{align}
It  features three independent variables, $t \in {\mathcal R}^+, \bx \in \Omega_\bx, \bv\in\Omega_v$, denoting the time, spatial location, and angular direction. For the full %$3D$ 
model considered in this paper, $\Omega_v=\mathbb{S}^2$ is the unit sphere.  The equation models the transport and the interaction of the particles (e.g. photons) with the background media (e.g. through the scattering and absorption).
  The unknown $f(\bx,\bv,t)$ is the angular flux (also called the distribution of particles).   $\mathcal{L}_{\textrm{collision}}f=\sigma_s(\lgl f\rgl-f)$ is the scattering operator,  
 % The collision operator $\mathcal{L}_{\textrm{collision}}f=\sigma_s(\lgl f\rgl-f)$ models the interaction between particles and the background media,
  where $\lgl f\rgl = \frac{1}{|\Omega_v|}\int_{\Omega_v}f(\bx,\bv,t)d\bv$ is the scalar flux (also the density) which is the average of $f$ in the angular space. $G(\bx)$ is an isotropic source term. In \eqref{eq:kinetic}, $\sigma_s(\bx)\geq 0$ and $\sigma_a(\bx)\geq 0$ are, respectively, the scattering and absorption cross sections. The Knudsen number
$\vareps$ is the non-dimensional mean free path of the particles. 
The main challenges for numerically solving this equation come from its high dimensional and multiscale nature. First, the angular flux $f$ depends on  the phase variable $(\bx, \bv)$ and the time. 
%$ a 6-dimensional vector $(\bx,\bv,t)$. 
Therefore, any standard grid-based method will  suffer from the curse of dimensionality. Second, the solution crosses different regimes thanks to its dependence on the non-dimensionalized mean free path $\vareps$. When $\vareps$ is $O(1)$, the problem is transport dominant. When $\vareps\rightarrow 0$ and $\sigma_s>0$, equation \eqref{eq:kinetic} converges to its diffusion limit:
\begin{align}
    \partial_t\rho - \nabla_\bx\cdot(\sigma_s^{-1} D\nabla_\bx\rho) = -\sigma_a \rho+G,
    \label{eq:diffusion_limit}
\end{align}
where $\rho(\bx,t)=\lgl f\rgl$ 
%is the scalar flux 
and $D=\textrm{diag}(\lgl v_x^2\rgl, \lgl v_y^2\rgl, \lgl v_z^2\rgl)$. 
This trans-regime behavior presents itself as both a challenge and an opportunity.

 To leverage the opportunity presented by the inherent structure of the equation in the diffusive regime and address the challenge especially of high dimensionality,    projection based ROMs and tensor decomposition based low rank algorithms have been designed for the stationary and time-dependent RTE.
Along the line of low rank algorithms based on tensor decomposition, dynamical low rank algorithm (DLRA)  \cite{peng2020low,einkemmer2021asymptotic,peng2021high}%kusch2021robust}  
and the proper generalized decomposition (PGD)
\cite{alberti2020reduced,prince2020space,dominesey2022reduced} %\cite{dominesey2018reduced,alberti2020reduced,prince2020space,dominesey2021reduced,dominesey2022reduced}
 have been designed. Projection based ROMs have also been actively developed in the recent few years, for example the proper orthogonal decomposition (POD) and its variations \cite{buchan2015pod,coale2019reduced,coale2019reduced2,tano2021affine,behne2021model,choi2021space,hughes2022adaptive}, the  dynamical mode decomposition  (DMD)  \cite{mcclarren2018acceleration,mcclarren2022data}. %low rank Schwarz solver based on random SVD \cite{chen2020random,chen2021low} and the generalized multiscale finite element method \cite{chung2020generalized}.
Among those work, the POD methods in \cite{buchan2015pod, tencer2016reduced, hughes2022adaptive} and 
our previous work in reduced basis method (RBM) for the steady state problem \cite{peng2022reduced} make explicit use of the low rank structure of  the solution manifold induced by  the angular variable, namely, the ROM built is based on treating the angular variable as the ``parameter" of the model. Once such ROM surrogate is constructed, it can be used to achieve a fast online calculation of the angular flux in an angular direction outside the training set. We will also show in this paper that a fast calculation of high order moments of the angular flux can be obtained by using the ROM surrogate. Moreover, the ROM can be further incorporated to  multi-query scenarios to accelerate calculations in inverse problems and uncertainty quantification.

In this paper, we continue our effort in \cite{peng2022reduced} and take the RBM approach \cite{patera2007reduced,rozza2008reduced,haasdonk2017reduced}, which  is a projection-based model order reduction strategy for  parametric problems  and consists of  Offline and Online stages.
In the Offline stage, it constructs a low-dimensional reduced order subspace to approximate the underlying solution manifold of the parametric problem. In the Online stage, the reduced order solution for unseen parameter values
is sought through a  (Petrov-)Galerkin projection into the low-dimensional surrogate subspace constructed offline. RBM utilizes a greedy algorithm for constructing the surrogate subspace offline. It iteratively augments the reduced order subspace by greedily identifying the snapshot, via an error estimator or an error / importance indicator, corresponding to the most under-resolved parameter (were the current reduced space to be adopted) in the training set until the stopping criteria is satisfied.  
%When the Kolmogorov $n$-width of the underlying parametric problem decays fast, the dimension of the constructed reduced order subspace upon the termination of the greedy procedure is much smaller than the number of degrees of freedom in the full order solver. As a result, computational cost saving is achieved when predicting solutions for parameter \fli{values} outside the training set.
%Our previous work developed a novel RBM strategy for the steady state RTE \cite{peng2022reduced}. 
%To mitigate the curse of dimensionality, the designed RBM aims to construct a surrogate solver which reduces the degrees of freedom (DOFs) needed. 
%The constructed ROM can be utilized to compress the solution manifolds and predict $f$ along angular directions outside of the training set. \yic{what has been done in previous paper}

 While the angular variable is treated as the parameter of the model in our previous work in \cite{peng2022reduced} for the stationary RTE, here for the time-dependent RTE, we regard both the angular $\bv$ and temporal $t$ variables as parameters and build a RBM by leveraging the low-rank structure of the $(\bv, t)$-induced solution manifold.  As observed in \cite{peng2022reduced} for the stationary case, the solution  of the time-dependent RTE corresponding to different  angular directions $\bv$ are not decoupled, due to the integral operator for the scattering. This makes our problem very different from the standard parametric problems the vanilla RBM is applied to. Compared to \cite{peng2022reduced}, the present work presents several significant algorithmic advances.
%compared to our previous work in \cite{peng2022reduced}. 
Our full order and reduced order models are based on the 
%We apply 
micro-macro decomposition of the RTE  \cite{liu2004boltzmann} 
instead of the original form in \eqref{eq:kinetic} for directly solving $f$.  To improve the performance in the diffusive and intermediate regime, we design an  equilibrium-respecting strategy to construct reduced order subspaces which better utilize the structure of the decomposed system. We call the proposed method micro-macro decomposed reduced basis method (MMD-RBM). %A byproduct of the proposed method is a reduced quadrature rule which has fewer angular samples compared with a standard one. \yic{is this necessary?} 
Furthermore, sampled angular variables are typically unstructured, and a direct robust and accurate quadrature rule to compute angular integrals %$\lgl \bv g\rgl$ and $(I-\Pi)(\bv\cdot\nabla g)$ 
is lacking. This is in particular crucial for time-dependent problems because it relates to the stability of the ROM.
%lacked. We find s
A recipe for constructing such quadrature rules preserving positivity of the weights is provided.

The rest of the paper is organized as follows. In Section \ref{sec:fom}, we present the micro-macro decomposition and the associated full order solver. In Section \ref{sec:rb}, we present Offline and Online stages of the MMD-RBM and estimate the  computational cost. In Section \ref{sec:numerical}, the performance of the proposed methods are demonstrated through a series of numerical examples. At last, we draw conclusions in Section \ref{sec:conclusions}.

%%%%%%%%%%%%%%%%%%%%%%%%%%%%%%%%%%%%%%
% Full order numerical method
%%%%%%%%%%%%%%%%%%%%%%%%%%%%%%%%%%%%%%
\section{Micro-macro decomposed RTE and its discretization\label{sec:fom}}

%%%%%%%%%
%\section{Micro-macro %decomposition %\label{sec:micro-macro}}
%\yic{since this section is so short, we can merge it with section 1} 
The radiative transfer equation (RTE) in \eqref{eq:kinetic} is multiscale in nature. When $\vareps=O(1)$, it  is transport dominant. On the other hand when $\vareps\rightarrow 0$, the model converges to its diffusion limit, and this   
%The diffusion limit
can be illustrated %obtained
through the micro-macro decomposition \cite{liu2004boltzmann}. Define $\Pi$ as the orthogonal projection onto the null space of the collision operator $\textrm{Null}(\mathcal{L}_\textrm{collision})$ in $L^2(\Omega_v)$. With the isotropic scattering being considered here, $\Pi f=\lgl f\rgl$. We decompose $f$ as
%$f(\bx,\bv,t) 
$f=\Pi f+(I-\Pi)f=\rho(\bx,t)+\vareps g(\bx,\bv,t)$, 
with $\rho(\bx,t)=\lgl f\rgl$ as the scalar flux (or called density).
Equation \eqref{eq:kinetic} can then be rewritten as the micro-macro decomposed system:
\begin{subequations}
\label{eq:micro-macro}
\begin{align}
  &\partial_t\rho +\nabla_{\bx}\cdot\lgl \bv g\rgl = -\sigma_a\rho+G, \label{eq:macro}\\
  \vareps^2&\partial_t g+ \vareps (I-\Pi)(\bv\cdot\nabla_\bx g)+\bv\cdot\nabla_\bx\rho = -\sigma_s g-\vareps^2\sigma_a g.
    \label{eq:micro}
\end{align}
\end{subequations}
As $\vareps\rightarrow 0$ and with
%for 
$\sigma_s(\bx)>0$,    %equation 
\eqref{eq:micro} becomes the local equilibrium
\begin{align}
    g=-\frac{1}{\sigma_s}\bv\cdot\nabla_{\bx}\rho.\label{eq:local_equilibrium}
\end{align} 
Substitute \eqref{eq:local_equilibrium} to %the equation
\eqref{eq:macro}, we obtain the diffusion limit:
\[
    \partial_t\rho - \nabla_\bx\cdot(\sigma_s^{-1} D\nabla_\bx\rho) = -\sigma_a \rho+G,
%    \label{eq:diffusion_limit:a}
\]
where $D=\textrm{diag}(\lgl v_x^2\rgl, \lgl v_y^2\rgl, \lgl v_z^2\rgl)$.

\subsection{Fully discretized micro-macro decomposed system}
\label{sec:FOM}

 When standard numerical methods are applied to solve \eqref{eq:kinetic}, the computational cost can be prohibitive  when $\vareps\ll 1$, as the mesh sizes smaller than $\vareps$ are often needed for both accuracy and stability \cite{caflisch1997uniformly,naldi1998numerical}. A numerical method for \eqref{eq:kinetic} is said to be asymptotic preserving (AP) \cite{jin2010asymptotic} if it preserves the asymptotic limit as $\vareps\rightarrow 0$ at the discrete level, namely,  as $\vareps\rightarrow 0$ the method  becomes a consistent and stable discretization for the limiting model. AP methods can work uniformly well for the model with a broad range of $\vareps$, particularly with $\vareps\ll1$ on under-resolved  meshes. This type of methods will be our choice as full order methods. 
In particular, in this work we adapt the IMEX-DG-S method \cite{peng2021asymptotic} to multiple dimensions. 
The  method is AP, with desirable time step conditions for stability, specifically, it is  unconditionally stable in the diffusive regime ($\vareps\ll 1$) and conditionally stable with a hyperbolic-type %standard 
CFL condition in the transport regime
($\vareps=O(1)$). %{($\vareps\approx 1$).
Alternatively, one can use other AP schemes based on  the micro-macro decomposition as the full order model, such as \cite{lemou2008new,jang2015high,peng2020stability}, which can have different stability property in the  diffusive regime ($\vareps\ll 1$). 

In this work, we assume all unknowns are independent of the $z$ variable, namely, $\partial_z \rho=\partial_z f=\partial_z g=0$. With this, we consider  $\Omega_\bx=[x_L,x_R]\times[y_L,y_R]$ in two space dimensions (with $d=2$) and $\Omega_v=\mathbb{S}^2$ as the angular space. The methodology developed here can be extended to $\Omega_\bx$ in three dimensions straightforwardly.
Next, we will present our full order method, starting from the time discretization.

\bigskip 
\noindent {\bf Time discretization:} To achieve unconditional stability in the diffusion  dominant regime as well as the AP property, the time discretization is defined as follows. Given the solutions $\rho^n$ and $g^n$ at $t^n=n\Dt$, we seek $\rho^{n+1}$ and $g^{n+1}$ such that
\begin{subequations}
\label{eq:semi-discretization}
    \begin{align}
        &\frac{\rho^{n+1}-\rho^n}{\Dt}+\nabla_{\bx}\cdot\lgl\bv g^{n+1}\rgl = -\sigma_a \rho^{n+1}+G^{n+1},
        \label{eq:macro_semi}\\
    \vareps^2&\frac{g^{n+1}-g^n}{\Dt} + \vareps(I-\Pi)(\bv\cdot\nabla_{\bx} g^n)+\bv\cdot\nabla_{\bx}\rho^{n+1} = -\sigma_s g^{n+1}-\vareps^2\sigma_a g^{n+1}. \label{eq:micro_semi}   
    \end{align}
\end{subequations}
As $\vareps\rightarrow0$ and with $\sigma_s>0$, \eqref{eq:micro_semi} becomes 
\begin{align}
    g^{n+1}=-\frac{1}{\sigma_s}\bv\cdot\nabla_{\bx}\rho^{n+1}.\label{eq:local_equ}
\end{align}
Substituting \eqref{eq:local_equ} into \eqref{eq:macro_semi}, we obtain the limit of scheme \eqref{eq:semi-discretization}  as $\vareps\rightarrow0$, 
\[
    \frac{\rho^{n+1}-\rho^n}{\Dt}-\nabla_\bx\cdot(\sigma_s^{-1} D\nabla_\bx\rho^{n+1})=-\sigma_a\rho^{n+1}+G^{n+1}.
\]
This is nothing but the backward Euler method for the diffusion limit in  \eqref{eq:diffusion_limit}.
%\yac{of the equation,} \eqref{eq:diffusion_limit}. 
Hence, this time discretization is AP.

\bigskip
\noindent {\bf Angular discretization:} In the angular space, we apply the discrete ordinates ($S_N$) method \cite{pomraning1973equations}. Let  $\{\bv_j\}_{j=1}^{N_v}$ be a set of quadrature points in $\Omega_v$  and $\{\omega_j\}_{j=1}^{N_v}$ be the corresponding quadrature weights, satisfying $\sum_{j=1}^{N_v}\omega_j=1$. The semi-discrete system  \eqref{eq:semi-discretization}
%\eqref{eq:micro}
is further  discretized in the angular variable,  following  a collocation approach, by being evaluated at $\{\bv_j\}_{j=1}^{N_v}$,  with the integral operator $\lgl\cdot\rgl$ 
approximated by its discrete analogue: 
\begin{equation}
    \lgl f\rgl \approx \lgl f \rgl_h = \sum_{j=1}^{N_v}\omega_j f(\cdot,\bv_j,\cdot). 
    \label{eq:d:int}
\end{equation}
We require the quadrature rule to  satisfy 
\begin{equation}
\lgl v_\xi v_\eta\rgl_h = \lgl v_\xi v_\eta\rgl =\frac{1}{3}\delta_{\xi\eta},\;\xi,\eta \;\in \{x,y,z\},\;\delta_{\xi\eta}=\begin{cases}
1,\;\xi=\eta\\
0,\;\xi\neq\eta
\end{cases}, 
\label{eq:quad_requirement}
\end{equation} 
so the coefficient matrix $D=\textrm{diag}(\lgl v_x^2\rgl, \lgl v_y^2\rgl, \lgl v_z^2\rgl)$ will be exact, and the correct diffusion limit will be obtained for the full order model without cross-derivative terms (see Section \ref{sec:mat-vec-schur}).
Particularly, with $\Omega_v=\mathbb{S}^2$, we use the Lebedev quadrature rule \cite{lebedev1976quadratures} in our fully-discrete method unless otherwise specified. 
\bigskip

\noindent {\bf Spatial discretization:}
In the physical space, we apply a discontinuous Galerkin (DG) discretization. 
Letting 
\[
    \mathcal{I}_h=\left\{\mathcal{I}_{kl}=[x_{k-\half},x_{k+\half}]\times[y_{l-\half},y_{l+\half}], 1\leq k
    \leq N_x,1\leq l
    \leq N_y\right\}
    %_{k=1,l=1}^{k=N_x,l=N_y}
\]
be a partition of the physical domain $\Omega_\bx$, we define  the discrete space as 
\[
    U_h^K(\Omega_\bx):=\{u(\bx): u(\bx)|_{\mathcal{I}_{kl}}\in Q^K(\mathcal{I}_{kl}),1\leq k
    %\in\mathbb{N}
    \leq N_x,1\leq l
    %\in\mathbb{N}
    \leq N_y\},
\]
where  $Q^K(\mathcal{I}_{kl})$ is the bi-variate polynomial space with the degree in each direction
%on each variate 
at most $K$ on the element $\mathcal{I}_{kl}$. 
We also write  $\phi(x_0^{\pm},y)=\lim_{x\rightarrow x_0^{\pm}}\phi(x,y)$ and $\phi(x,y_0^\pm)=\lim_{y\rightarrow y_0^{\pm}}\phi(x,y)$.

Let the numerical solution at $t^n$ be $\rho^n_h(\cdot)\approx \rho(\cdot,t^n)$ and $g_{h,j}^n(\cdot)\approx g(\cdot,\bv_j,t^n), \forall j=1,\dots, N_v$.
With a DG discretization applied in space, we reach our fully-discrete scheme: 
given $\rho_h^n \in U_h^K, \{g_{h,j}^n\}_{j=1}^{N_v}\subset U_h^K$, we seek  $\rho_h ^{n+1}\in U_h^K, \{g_{h,j}^{n+1}\}_{j=1}^{N_v}\subset U_h^K$, satisfying the following equations $\forall k=1,\dots,N_x, l=1,\dots,N_y$, 
\begin{subequations}
\label{eq:dg}
    \begin{align}
    &\int_{\mathcal{I}_{kl}} \frac{\rho_h^{n+1}-\rho_h^n}{\Dt}\phi_h d\bx
    +{\sum_{\gamma=1}^{N_v}\omega_\gamma\int_{\mathcal{I}_{kl}}\left(\mathcal{D}^{g}_x( v_{\gamma,x} g_{h,\gamma}^{n+1};\rho_h^{n+1})+\mathcal{D}^{g}_y( v_{\gamma,y} g_{h,\gamma}^{n+1};\rho_h^{n+1})
    \right)\phi_h d\bx}\notag\\
    &=\int_{\mathcal{I}_{kl}}(-\sigma_a \rho_h^{n+1}+G^{n+1}
     %+G^{n+1}
     )\phi_h d\bx,\quad \forall \phi_h\in U_h^K,  \label{eq:dg_macro} \\
    &\vareps^2\int_{\mathcal{I}_{kl}} \frac{g_{h,j}^{n+1}-g_{h,j}^n}{\Dt} \psi_h d\bx+\int_{\mathcal{I}_{kl}}\left(v_{j,x}\mathcal{D}^-_x+v_{j,y}\mathcal{D}^-_y\right) \rho_{h}^{n+1}\;\psi_h d\bx\notag\\
       &+\vareps\sum_{\gamma=1}^{N_v}(\delta_{j\gamma}-\omega_\gamma)\int_{\mathcal{I}_{kl}}\left(\mathcal{D}_x^{\textrm{up}} (v_{\gamma,x}, g_{h,\gamma}^{n})+\mathcal{D}^{\textrm{up}}_y (v_{\gamma,y}, g_{h,\gamma}^{n})\right)\psi_hd\bx
    \notag\\
    &= 
    -\int_{\mathcal{I}_{kl}}(\sigma_s+\vareps^2\sigma_a)g_{h,j}^{n+1}\psi_h d\bx,\qquad \forall \psi_h\in U_h^K, \forall j=1,\dots N_v. 
\label{eq:dg_micro}
    \end{align}
\end{subequations}
Here $\delta_{j\gamma}$ is the Kronecker delta, $\mathcal{D}_x^-(\cdot), \mathcal{D}_y^-(\cdot), \mathcal{D}_x^g(\cdot;\cdot),\mathcal{D}_y^g(\cdot;\cdot), \mathcal{D}_x^{\textrm{up}}(\cdot,\cdot), \mathcal{D}_y^\textrm{up}(\cdot,\cdot) \in U_h^K$ are all discrete (partial) derivatives, and they can be  expressed in terms of  $\mathcal{D}_x^\pm(\cdot), \mathcal{D}_y^\pm(\cdot)\in U_h^K$ that are defined as follows
\begin{subequations}
    \label{eq:discrte_alt_derivatives}
    \begin{align}
    \int_{\mathcal{I}_{kl}}\mathcal{D}^\pm_x\phi_h\psi_h d\bx
    = &-\int_{\mathcal{I}_{kl}}\phi_h\partial_x\psi_h d\bx
      +\int_{y_{l-\half}}^{y_{l+\half}}\phi_h(x_{k+\half}^\pm,y)\psi_h(x_{k+\half}^-,y)dy\notag\\
      &-\int_{y_{l-\half}}^{y_{l+\half}}\phi_h(x_{k-\half}^\pm,y)\psi_h(x_{k-\half}^+,y)dy, \qquad  \forall \psi_h\in U_h^K,\\
    \int_{\mathcal{I}_{kl}}\mathcal{D}^\pm_y\phi_h\psi_h d\bx
    = &-\int_{\mathcal{I}_{kl}}\phi_h\partial_y\psi_h d\bx
      +\int_{x_{k-\half}}^{x_{k+\half}}\phi_h(x,y_{l+\half}^\pm)\psi_h(x,y_{l+\half}^-)dx\notag\\
      &-\int_{x_{k-\half}}^{x_{k+\half}}\phi_h(x,y_{l-\half}^\pm)\psi_h(x,y_{l-\half}^+)dx, \qquad  \forall \psi_h\in U_h^K.
    \end{align}
\end{subequations}
With $\bv\cdot\nabla_\bx g^n$ in \eqref{eq:micro_semi}  discretized following an upwind mechanism, we set 
%\begin{subequations}
%    \label{eq:discrte_up_derivatives}
    \begin{align*}
    &\mathcal{D}^{\textrm{up}}_x(v_{x},\phi_h)=v_x \mathcal{D}^{\star}_x(\phi_h), \qquad  
   \text{with}\;\star=\begin{cases}
                       -,\;v_x\geq 0,\\
                        +,\;v_x<0,
                        \end{cases}\\
&\mathcal{D}^{\textrm{up}}_y(v_{y},\phi_h)=v_y \mathcal{D}^{\star}_y(\phi_h), \qquad  
  \text{with}\;\star=\begin{cases}
                        -,\;v_y\geq 0,\\
                      +,\;v_y<0.
                      \end{cases}
    \end{align*}
%\end{subequations}
Moreover, we take
\begin{equation}
\label{eq:discrte_g_derivatives}
    \mathcal{D}^{g}_\xi( v_{\gamma,\xi} g_{h,\gamma};\rho_h)= v_{\gamma,\xi} \mathcal{D}^{+}_\xi g_{h,\gamma} + \alpha_\xi\mathcal{D}^{\textrm{jump}}_\xi \rho_h,\quad\text{with }\xi=x,y.
\end{equation}

Here, $\mathcal{D}^{\textrm{jump}}_x(\cdot) \in U_h^K$, given locally on the element $\mathcal{I}_{kl}$ by
\begin{align}
\int_{\mathcal{I}_{kl}}\mathcal{D}^{\textrm{jump}}_x(\rho_h)\psi_h d\bx
&= \int_{y_{l-\half}}^{y_{l+\half}}\left(\rho_h(x_{k+\half}^-,y)-\rho_h(x_{k+\half}^+,y)\right)\psi_h(x_{k+\half}^-,y)dy\notag\\
&-\int_{y_{l-\half}}^{y_{l+\half}}\left(\rho_h(x_{k-\half}^-,y)-\rho_h(x_{k-\half}^+,y)\right)\psi_h(x_{k-\half}^+,y)dy,\quad  \forall \psi_h\in U_h^K,\notag
\end{align}
and equivalently,  
\[
\mathcal{D}^{\textrm{jump}}_x(\rho_h)=\mathcal{D}^{-}_x(\rho_h)-\mathcal{D}^{+}_x(\rho_h).
\]
Similarly
\begin{equation}
 \mathcal{D}^{\textrm{jump}}_y(\rho_h)=\mathcal{D}^{-}_y(\rho_h)-\mathcal{D}^{+}_y(\rho_h).
    \label{eq:jump:y}
\end{equation}
The jump operators are added in  \eqref{eq:discrte_g_derivatives} to   maintain accuracy in the case of the Dirichlet boundary conditions \cite{castillo2000priori}. As shown in  \cite{castillo2000priori}, the constants $\alpha_x$, $\alpha_y$ in \eqref{eq:discrte_g_derivatives} need to be $O(1)$ and positive.  In this paper, we consider the vacuum boundary condition. In all the discrete derivatives, when the data from the outside of the domain is needed for the solution, we directly set it as $0$. 
%\yac{question on the jumps}  

From here on, we refer to the fully-discrete method \eqref{eq:dg} along with \eqref{eq:discrte_alt_derivatives}-\eqref{eq:jump:y} as the full order model denoted as FOM. Given that our plan is to treat the angular variable $\bv$ as a parameter to formulate reduced order models, when we want to emphasize the set of the angular values $\mcV$ (and its ``associated'' quadrature weights in \eqref{eq:d:int}) used to define \eqref{eq:dg}, we also write it as FOM($\mcV$). As an example,  we have $\mcV=\{\bv_j\}_{j=1}^{N_v}$  for \eqref{eq:dg}.

%\subsubsection{Matrix-vector form and Schur complement}
\subsection{Matrix-vector form and Schur complement}
\label{sec:mat-vec-schur}
Though $\lgl vg\rgl$ is treated implicitly in \eqref{eq:dg_macro},
%{eq:dg}, 
we only need to invert a discrete heat operator for $\rho$ with the help of  the Schur complement, and this will be demonstrated next via the matrix-vector form of the scheme.  Let $\{e_l(\bx)\}_{l=1}^{N_{\bx}}$ be a basis of the DG space $U_h^K$, then $\rho_h^n$ and $g_{h,j}^n$ can be expanded as
$\rho_h^n(\bx)=\sum_{l=1}^{N_{\bx}}\rho_l^n e_l(\bx)$ and    $g_{h,j}^n(\bx)=\sum_{l=1}^{N_{\bx}}g_{l,j}^ne_l(\bx).$
Defining $\brho^n=(\rho_1^n,\dots,\rho_{N_{\bx}}^n)^T$ and $\bg_j^n=(g_{1,j}^n,\dots,g_{N_{\bx},j}^n)^T$, we are ready to rewrite \eqref{eq:dg} into its matrix-vector formulation:
\begin{subequations}
\label{eq:dg-matrix-vector}
\begin{align}
&\mathcal{A}\left({\boldsymbol{\rho} }^{n+1},{\bg}_1^{n+1},{\bg}_2^{n+1},\dots{\bg}_{N_v}^{n+1}\right)^T
=\left({\bf{b}}_\rho^n,{\bf{b}}_{g_1}^n,{\bf{b}}_{g_2}^n,\dots,{\bf{b}}_{g_{N_v}}^n\right)^T,\label{eq:matrix-vector1}
\\[0.15\baselineskip]
&\mathcal{A}=
\left(
\begin{matrix}
M + \Dt\Sigma_a+\Dt D^\textrm{jump}	  	&\Dt \omega_1 (v_{1,x}D_x^++v_{1,y}D^+_y) 		&\dots 	& \Dt \omega_{N_v}(v_{N_v,x}D_x^++v_{N_y,y}D^+_y) \\
\Dt(v_{1,x}D_x^-+v_{1,y}D^-_y) 	& \Theta &\dots 	& 0 \\
\vdots	  			& \vdots 	&\ddots	&\vdots\\
\Dt(v_{N_v,x}D_x^-+v_{N_v,y}D^-_y)	& 0	&\dots  	&\Theta
\end{matrix}
\right),\\
& {\bf{b}}_\rho^n = M\brho^n+\Dt\bG^{n+1},\\
&{\bf{b}}_{g_j}^n = \vareps^2 M\bg_j^n-\vareps\Dt\sum_{\gamma=1}^{N_v}(\delta_{j\gamma}-\omega_\gamma)(D_{x,v_{\gamma,x}}^\textrm{up}+D_{y,v_{\gamma,y}}^\textrm{up})\bg_\gamma^n, \quad j=1,\dots, N_v
\end{align}
\end{subequations}
Here $M$ is the mass matrix, $\Sigma_s$ (resp. $ \Sigma_a$) is  the scattering (resp. absorption) matrix, $D^{\textrm{jump}}$ is the jump matrix, $D_\xi^\pm$, {$D^{\textrm{up}}_{\xi,v_{\gamma,\xi}}$} ($\xi=x,y, \gamma=1,\dots, N_v$) are  discrete derivatives matrices, all being of the size $N_{\bx}\times N_{\bx}$ ($N_{\bx}$ is the number of degrees of freedom resulting from the spatial discretization), with their $(kl)$-th entry given as:
%\begin{subequations}
\begin{align*}
 &M_{kl}=\int_{\Omega_x}e_l e_k d\bx, \quad (\Sigma_s)_{kl}=\int_{\Omega_x}\sigma_se_l e_k d\bx,\quad (\Sigma_a)_{kl}=\int_{\Omega_x}\sigma_ae_l e_k d\bx,\\
& (D_{\xi}^\pm)_{kl}=\int_{\Omega_x}\mathcal{D}_{\xi}^\pm e_le_kd\bx, \quad (D^{\textrm{up}}_{\xi,v_{\gamma,\xi}})_{kl}=\int_{\Omega_x}\mathcal{D}_{\xi}^{\textrm{up}}(v_{\gamma,\xi},e_l)e_k d\bx, \quad (\text{with}\; \xi=x,y),\\
& D^\textrm{jump}=\alpha_x(D_x^--D_x^+)+\alpha_y (D_y^--D_y^+).
\end{align*}
%\end{subequations}
In addition, $\bG^{n+1}$ is the source vector, with its $k$-th entry $\int_{\Omega_x}G^{n+1}e_kd\bx$, and $\Theta = \vareps^2 (M +\Dt\Sigma_a)+ \Dt\Sigma_s$. Using the standard choices of the basis of $U_h^K$ (e.g with the support of each basis function being one mesh element), the matrices $M$, $\Sigma_s$, $\Sigma_a$ and $\Theta$ are block-diagonal. When the boundary conditions are periodic  or vacuum in space, one can easily show  $D^+_\xi=-(D^-_\xi)^T$ with $\xi=x,y$ (see \cite{peng2021asymptotic} for details).

To avoid inverting the big matrix $\mathcal{A}$ directly, we apply the Schur complement. Noticing that 
\begin{equation}
\bg_j^{n+1} = \Theta^{-1}\left( {\bf{b}}_{g_j}^n - \Dt(v_{j,x}D_x^-+v_{j,y}D^-_y)\brho^{n+1}\right), \quad\forall j=1,\dots, N_v,
\label{eq:schur_step1}
\end{equation}
we eliminate $\bg_j^{n+1}$ terms in the equation determined by the first line of $\mathcal{A}$ and obtain 
\begin{equation}
    \mathcal{H}\brho^{n+1}= \widetilde{\mathbf{b}}_\rho^{n},  \label{eq:schur}
\end{equation}
where
\begin{align}
  \mathcal{H}&=M+\Dt\Sigma_a+\Dt D^\textrm{jump} - \Dt^2\sum_j \omega_j (v_{j,x} D_x^++v_{j,y}D_y^+)\Theta^{-1}(v_{j,x} D_x^-+v_{j,y}D_y^-) \notag \\
  &=  M+\Dt\Sigma_a+\Dt D^\textrm{jump} - \Dt^2(\lgl v_x^2\rgl_h  D_x^+\Theta^{-1}D_x^-+\lgl v_y^2\rgl_h D_y^+\Theta^{-1}D_y^-).\notag
\end{align}
\begin{comment}
\begin{align}
  \mathcal{H}\brho^{n+1}&=\Big (M+\Dt\Sigma_a+\Dt D^\textrm{jump} - \Dt^2\sum_j \omega_j (v_{j,x} D_x^++v_{j,y}D_y^+)\Theta^{-1}(v_{j,x} D_x^-+v_{j,y}D_y^-) \Big) \brho^{n+1}\notag \\
  &=  \Big(M+\Dt\Sigma_a+\Dt D^\textrm{jump} - \Dt^2(\lgl v_x^2\rgl_h  D_x^+\Theta^{-1}D_x^-+\lgl v_y^2\rgl_h D_y^+\Theta^{-1}D_y^-)\Big)\brho^{n+1} \notag\\
  &= \widetilde{\mathbf{b}}_\rho^{n}.
  \label{eq:schur}
\end{align}
\end{comment}
 The second line above is a direct result of  $\lgl v_{x}v_{y}\rgl_h=\lgl v_{x}v_{y}\rgl=0$ in \eqref{eq:quad_requirement}.
With \eqref{eq:schur}, we only need to invert a linear system $\eqref{eq:schur}$ of a much smaller size for $\rho$. Moreover, $\mathcal{H}$ is a discrete heat operator, and it  is symmetric positive definite due to $D^+_\xi=-(D^-_\xi)^T$ with  $\xi=x,y$,  and hence can be   efficiently inverted, e.g. by the conjugate  gradient (CG)  method with  algebraic multigrid (AMG) preconditioners. Once  $\brho^{n+1}$ is available,  $\bg_j^{n+1}$ can be obtained from \eqref{eq:schur_step1}, and this can be carried out in a parallel fashion, given that $\Theta$ is block-diagonal and the equations \eqref{eq:schur_step1} in $j$ are decoupled.

%%%%%%%%%%%%%%%%%%%%%%%%%%%%%%%%%%%%%%
% Stability
%%%%%%%%%%%%%%%%%%%%%%%%%%%%%%%%%%%%%%
\subsection{Stability}
When $U_h^K$ with $K=0$ is used (as numerically tested in Section \ref{sec:numerical}), our FOM method is first order accurate, and its stability can be established by following similar techniques in \cite{peng2021asymptotic}, and this result  will play an important role in the design of the ROM. The key to prove the stability in \cite{peng2021asymptotic} is to introduce the following discrete energy:
\begin{equation}
\label{eq:mu_energy}
E_{h}^n=||\rho_h^n||^2+\vareps^2 \sum_{j=1}^{N_v}\omega_j||g_{h,j}^{n}||^2+
\Dt \sum_{j=1}^{N_v}\omega_j \int_{\Omega_x}\sigma_s(g_{h,j}^n)^2d\bx,
\end{equation}
where $||\cdot||$ is the standard $L^2$ norm in $L^2(\Omega_x)$.
\begin{comment}
\begin{defn}
\label{def:stab}
Given $\mu\in[0,1]$, we define a discrete energy: 
\begin{align}
\label{eq:mu_energy}
E_{h}^n=||\rho_h^n||^2+\vareps^2 \sum_{j=1}^{N_v}\omega_j||g_{h,j}^{n}||^2+
\Dt \sum_{j=1}^{N_v}\omega_j \int_{\Omega_x}\sigma_s(g_{h,j}^n)^2d\bx,
\end{align}
where $||\cdot||$ is the standard $L^2$ norm in $L^2(\Omega_x)$.
\end{defn}
\end{comment}
With $\sigma_s\geq 0$, the term $E_{h}^n$ is non-negative and gives a well-defined energy. Using similar techniques in \cite{peng2021stability,peng2021asymptotic}, we can extend the Theorem 5.4 in \cite{peng2021asymptotic} from 1D to 2D. We next state this result, presented in in the context of the current work.
\begin{thm}\label{thm:stability}
\textbf{(Stability condition)}\footnote{This theorem can be established by following the proofs of Theorem 5.3 and Theorem 5.4 in \cite{peng2021asymptotic} for the one spatial dimension case. The only difference is that, due to the extra dimension in space, there will be two extra terms similar to equations (5.7) and (5.8) of \cite{peng2021asymptotic} in an equality similar to equation (5.5) of \cite{peng2021stability}. }
Suppose $\omega_j\geq0, \forall 1\leq j\leq N_v$, and $\sigma_s\geq \sigma_m > 0$. 
Let $h=\min(\min_{1\leq i\leq N_x}(x_{i+\half}-x_{i-\half}), \min_{1\leq i\leq N_y}(y_{i+\half}-y_{i-\half}))$, we have that
\begin{itemize}
    \item[(1)] when $\frac{\vareps}{\sigma_m h}\leq \frac{1}{4\max_{1\leq j\leq N_v}|\bv_j|_\infty}$, $E_{h}^{n+1}\leq E_{h}^n$  $\; \forall \Dt>0$;
    \item[(2)]  when $\frac{\vareps}{\sigma_m h}\leq \frac{1}{4\max_{1\leq j< N_v}|\bv_j|_\infty}$,  $E_{h}^{n+1}\leq E_{h}^n$ under the time step condition
    $$\Dt\leq\frac{ \vareps h}{4\max_{1\leq j\leq N_v}|\bv_j|_\infty-\sigma_m h/\vareps}.$$
\end{itemize}
\end{thm}
%Since $E_{h}^n$ is a well-defined energy, this
%\fliq{/dependence of C/}
The theorem implies that the scheme is unconditionally stable in the diffusive regime (i.e. when $\vareps/(\sigma_m h)$ is small enough), and the stability condition in the transport regime (i.e. $\vareps=O(1)$) is on the same level as the standard CFL condition (i.e. $\Dt\leq O(\vareps h)$).

%%%%%%%%%%%%%%%%%%%%%%%%%%%%%%%%%%%%%%
% Reduced order model
%%%%%%%%%%%%%%%%%%%%%%%%%%%%%%%%%%%%%%
\section{The micro-macro decomposed reduced basis method}
%s}
\label{sec:rb}

Our proposed MMD-RBM algorithm consists of an Offline stage, which constructs the low dimensional subspaces and a reduced quadrature rule, and an Online stage which features a surrogate solver capable of efficiently computing moments {of $f$ and predicting  the angular flux $f$} corresponding to angular directions 
%velocities 
unseen during the Offline stage. 
In this section, we outline the entire algorithm in Section \ref{sec:rb:alg}. In particular, we provide a high-level sketch in Figure \ref{fig:algorithm_flow} to assist reading. We then discuss each step of the Online and Offline stages in Sections \ref{sec:online:what2do} and \ref{sec:offline}, respectively. A computational complexity analysis is provided in Section \ref{sec:cost} relating the cost of MMD-RBM with those of vanilla POD and brute force FOM.
%%%%%%%%%%%%%%%%%%%%%%%%%%%%%%%%%%%%%%%%%%%%%%%%%%%%%%%%%%%%%%%%%%%%%%%%%%
% Roadmap
%%%%%%%%%%%%%%%%%%%%%%%%%%%%%%%%%%%%%%%%%%%%%%%%%%%%%%%%%%%%%%%%%%%%%%%%%%
%\subsection{\fli{Overview and} %roadmap of the proposed algorithm}  \label{sec:rb:roadmap}

%, \fliq{and will be constructed as a main ingredient of the proposed algorithm}.
%\pzc{Compared with the steady state problem, extra constraints of the quadrature rule are needed to maintain the stability of the time integration. We follow a similar but slightly simpler approach as in \cite{peng2022reduced} to construct such quadrature rules which maintain the stability. }

%%%%%%%%%%%%%%%%%%%%%%%%%%%%%%%%%%%%%%%%%%%%%%%%%%%%%%%%%%%%%%%%%%%%%%%%%%
% Roadmap
%%%%%%%%%%%%%%%%%%%%%%%%%%%%%%%%%%%%%%%%%%%%%%%%%%%%%%%%%%%%%%%%%%%%%%%
\subsection{Outline of the MMD-RBM algorithm}
\label{sec:rb:alg}

The flowchart of the entire algorithm  is summarized  in Figure \ref{fig:algorithm_flow}. 
%The offline algorithm is presented in Algorithm \ref{alg:offline}. The online stage will be discussed in details in Section \ref{sec:online:what2do}. 
Other than the clear distinction of Offline and Online stages, 
another feature of this algorithm is that
\[
\mbox{ROM}(\cdot;\cdot,\cdot),
\]
representing our reduced order (thus online) solver, appears offline too, albeit with a pair of dynamically expanding surrogate spaces  as the second and third input. 
Being a critical step in the greedy algorithm, this solver helps to recursively build the reduced parameter sets and augment the surrogate spaces in a greedy fashion.
For this reason, before we dive into the detailed description of the Offline stage in Section \ref{sec:offline}, we first introduce in Section \ref{sec:online:what2do} this reduced formulation which corresponds to the full-order scheme \eqref{eq:dg}. 

Specifically, in Section \ref{sec:online:what2do}, we introduce our projection-based reduced formulation ROM($\mcV; U^\rho_{h,r}, U^g_{h,r}$). Here  $U^\rho_{h,r}$ is the reduced order space  for $\rho$, $U^g_{h,r}$  is the reduced order space  for $g$, and $\mcV$ is the angular set used in the angular discretization. We assume that there are quadrature weights $\{\omega_{\bv}\}_{\bv\in \mcV}$ associated with $\mcV$, and the discrete analogue $\lgl \cdot\rgl_{h,\mcV}$ for the integral operator $\lgl \cdot\rgl$. In the  online surrogate solver, we solve ROM($\mcV_{\textrm{rq}};U_{h,r}^\rho,U_{h,r}^g$) with the terminal $U_{h,r}^\rho$ and $U_{h,r}^g$; and in the greedy sampling offline, we solve ROM($\mcV_{\textrm{train}};U_{h,r}^\rho,U_{h,r}^g$) with the current (and to-be-updated) $U_{h,r}^\rho$ and $U_{h,r}^g$. 
Here, $\mcV_{\textrm{rq}}$ is the (usually unstructured) set of angular values identified by the Offline algorithm while $\mcV_{\textrm{train}}$ denotes the (usually structured) training set of the angular directions specified at the beginning of the Offline algorithm. 
\begin{figure}
\begin{center}
\scalebox{0.88}{
\begin{tikzpicture}
\tikzset{every node}=[font=\sffamily\sansmath]
%%%%%%%%%%%%%%%%%%%%%
% Offline algorithm
%%%%%%%%%%%%%%%%%%%%%
\node (Stopping) [diamond,  line width = 2pt, aspect = 3.0,text width=3.0cm,minimum height =1.0cm,draw, thick,fill=cyan] {Stopping criteria satisfied?};
%\node (Stopping) [trapezium,  line width = 2pt,trapezium right angle=60,trapezium left angle=-60,text width=4.0cm,draw, thick,fill=cyan] {Stopping criteria satisfied?};
\node(Initialization)[rectangle,draw,thick,fill=cyan,above=of Stopping,text width=11.0cm,yshift=-0.5cm]{Initialization:(1) sampled parameter sets, (2) reduced quadrature nodes $\mcV_{\textrm{rq}}$, (3) reduced spaces $U_{h,r}^\rho$ and $U_{h,r}^g$ via FOM($\mcV_{\textrm{rq}}$)};
\node(Input)[rectangle,draw,above=of Initialization,text width=11.0cm,rounded corners,yshift=-0.15cm,fill=lime]{Input:  temporal mesh $\mcT_{\textrm{train}}$ and angular training set $\mcV_{\textrm{train}}$};
\node (Indicator) [block_device_small,text width=4.0cm,draw, thick,fill=cyan,below=of Stopping,xshift=3.0cm] {Compute error indicators};
%$\Delta_\rho^n$ and $\Delta_{g_{\bv}}^n$};
\node (ReducedSolve) [block_device_small,text width=5.0cm,draw, thick,fill=cyan,left=of Indicator] {Solve ROM($\mcV_{\textrm{train}};U_{h,r}^\rho,U_{h,r}^g$)};
\node (Sampling) [block_device_small,text width=3.5cm,draw, thick,fill=cyan,below=of Indicator,yshift=0.5cm] {Greedy selection  of  angular and time sample};
%$t^{n_{\rho}}=\arg\max_n \Delta_\rho^n$,\\
%$(t^{n_{g}},{\bv}_g)=\arg\max_{(n,\bv)}\Delta_{g_{\bv}}^n$};
\node (SymmetryEnhancing) [block_device_small,text width=5.5cm,draw, thick,fill=cyan,left=of Sampling,xshift=0.25cm] {
$\star$ Symmetry-enhancing update\\
$\star$ Update reduced quadrature rule $\lgl \cdot\rgl_{h,\mcV_{\textrm{rq}}}$
};%: sample $(t^{n_g},-{\bv}_g)$};
%\node(SamplingSummary)[above=of Indicator,xshift=-2.5cm,yshift=-1.0cm,xshift=1.0cm,text width=4.0cm]{Greedy procedure};
%and augmentation};%: line $8$ of Algorithm \ref{alg:offline}};
%\node(UpdateQuad)[block_device,draw,below=of GreedySampling,yshift=0.5cm,text width=5cm]{};
\node(UpdateBasis)[block_device,draw,below=of SymmetryEnhancing,text width=5cm,xshift=2.5cm]{Solve FOM($\mcV_{\textrm{rq}}$) and update spaces $U_{h,r}^\rho,U_{h,r}^g$};
\node[draw,densely dotted,inner xsep=2mm,inner ysep=2.0mm,fit=(ReducedSolve)(Indicator)(Sampling)(SymmetryEnhancing)(UpdateBasis),label={[xshift=2.25cm,yshift=-0.1cm]Greedy procedure}](GreedySampling){};
\node[draw,dashed,inner xsep=4mm,inner ysep=2.0mm,xshift=-2.5mm,fit=(ReducedSolve)(Indicator)(Sampling)(SymmetryEnhancing)(UpdateBasis)(Initialization)(Stopping)(GreedySampling),label={[xshift=2.25cm,yshift=-0.1cm]\textbf{Offline stage}}](Offline){};
\node (Output) [block_laptop,text width=4.0cm,draw, thick,fill=lime,right=of Stopping,rounded corners,xshift=1.95cm] {Output:  $\mcV_{\textrm{rq}}$, $\lgl\cdot\rgl_{r,\mcV_{\textrm{rq}}}$ and $U_{h,r}^\rho$, $U_{h,r}^g$};
\draw[arrow](Stopping.south) --node[midway,left,xshift=-0.25cm]{No} (ReducedSolve);
\draw[arrow](Input)--(Initialization);
\draw[arrow](ReducedSolve)--(Indicator);
\draw[arrow](Indicator)--(Sampling);
\draw[arrow](Sampling)--(SymmetryEnhancing);
\draw[arrow](-0.75,-5.4)--(-0.75,-6.5);%(UpdateBasis);
%\draw[arrow](UpdateQuad)--(UpdateBasis);
\draw [arrow](UpdateBasis) -- ++(-5.95,0.0)|-(Stopping);
\draw [arrow](Stopping) --node[midway,above]{Yes} (Output);
\draw [arrow](Initialization)--(Stopping);
%%%%%%%%%%%%%%%%%%%%%%%%%%%%%%
\node(OnlineMoments)[block_laptop,fill=pink,below=of Output,text width=4.0cm,yshift=-2.25cm]{Solve ROM($\mcV_{\textrm{rq}};U_{h,r}^\rho,U_{h,r}^g$) to compute moments};
\node(OnlinePrediction)[rectangle,draw,thick,fill=pink,below=of OnlineMoments,text width=4.0cm,yshift=0.5cm]{Predict $f$ for unseen $\bv$};
\node[draw,dashed,inner xsep=1.5mm,inner ysep=2.0mm,fit=(OnlineMoments)(OnlinePrediction),label={[xshift=-0.1cm]:\textbf{Online stage}}](OnlineStage){};
\draw [arrow] (Output) -- (OnlineStage);

\end{tikzpicture}}
\end{center}
\caption{The flowchart of the proposed MMD-RBM algorithm.}
\label{fig:algorithm_flow}
\end{figure}
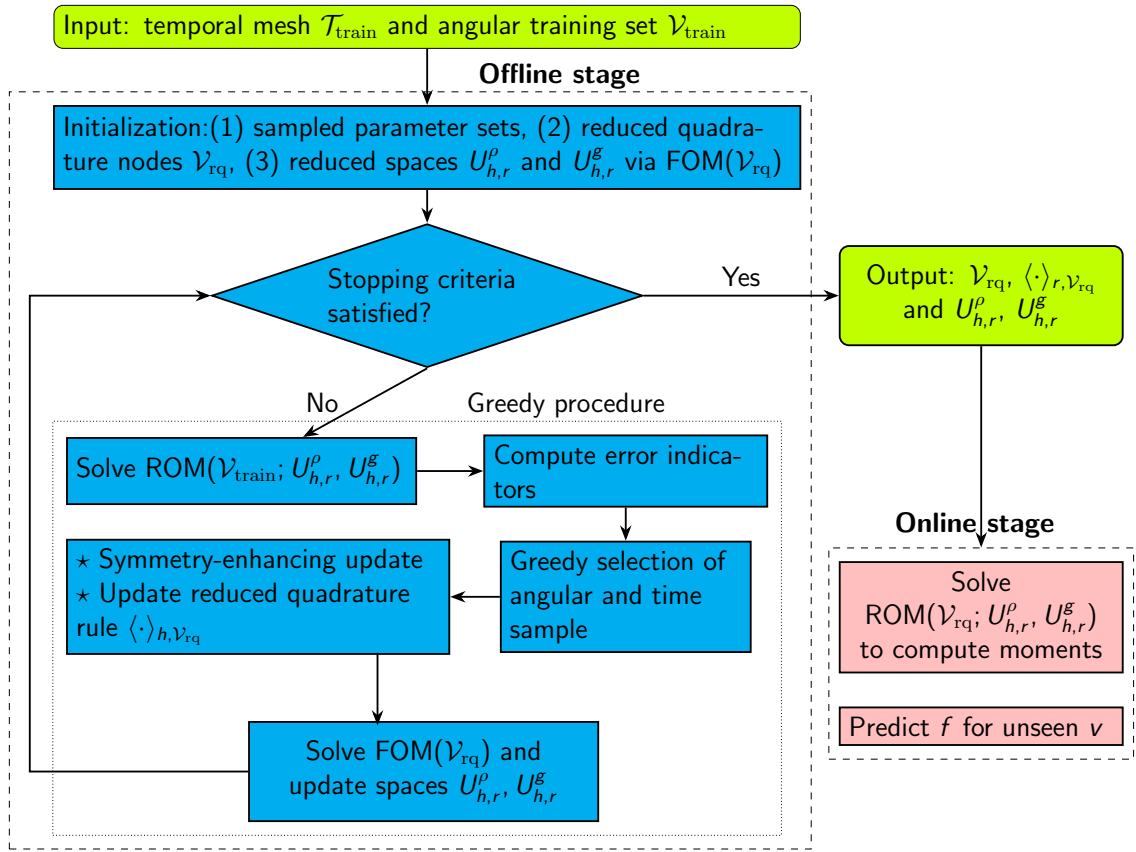

In the Online stage (the pink block of the flowchart, to be described in Section \ref{sec:online:what2do}), our ROM can be utilized to predict $f$ at angular directions outside the training set as well as some moments of $f$ with significantly fewer 
%smaller 
degrees of freedom.  In the Offline stage (the blue block of the flowchart, to be described in Section \ref{sec:offline}), after initializing the quadrature nodes of the reduced quadrature rule  $\mcV_{\textrm{rq}}$ and the set of sampled parameters $\mathcal{T}_{\textrm{rb}}^\rho$ and $\mathcal{TV}_{\textrm{rb}}^g$, we use a greedy algorithm to iteratively construct the subspace $U_{h,r}^\rho$ and $U_{h,r}^g$. The main steps are
%of the greedy iteration include (1) 
\begin{itemize}
\item described in Section \ref{sec:offline_l1}, solving ROM($\mcV_{\textrm{train}};U_{h,r}^\rho,U_{h,r}^g)$ to identify the most under-resolved angular and temporal samples, $t^{\textrm{new}}_\rho$ for $\rho$ and $(t^{\textrm{new}}_g,\bv^{\textrm{new}}_g)$ pair for $g$, based on an importance indicator. Updating the set of sampled parameters $\mathcal{T}_{\textrm{rb}}^\rho$ with $t^{\textrm{new}}_\rho$ and $\mathcal{TV}_{\textrm{rb}}^g$, in a symmetry-enhancing fashion, with $(t^{\textrm{new}}_g,\pm\bv^{\textrm{new}}_g)$. 
\item described in Section \ref{sec:offline_redquad}, updating the corresponding reduced quadrature rule $\lgl\cdot\rgl_{h,\mcV_{\textrm{rq}}}$ preserving weight positivity via a novel least squares strategy. 
\item described in Section \ref{sec:offline_basis_update}, updating the RB spaces $(U_{h,r}^\rho,U_{h,r}^g)$.
%with $(\rho(t^{\textrm{new}}_\rho), g(t^{\textrm{new}}_g,\pm\bv^{\textrm{new}}_g))$.
\end{itemize}

\subsection{Reduced MMD formulation and online
%Online 
functionalities}
\label{sec:online:what2do}

{\bf Reduced MMD formulation ROM($\mcV; U^\rho_{h,r}, U^g_{h,r}$).} We present the reduced MMD  formulation in its matrix-vector form. Toward this end, we assume that $B_\rho\in\mathbb{R}^{N_{\bx}\times r_\rho}$ and $B_g\in\mathbb{R}^{N_{\bx}\times r_g}$ contain the orthonormal basis of $U^{\rho}_{h,r}$ and $U^{g}_{h,r}$, respectively, as their columns, and look for the reduced solution $\brho_r=B_\rho\bc_{\rho}$ for $\rho$, and $ \bg_{\bv,r}=B_g\bc_{g_{\bv}}$ for $g$ at $\bv\in\mcV$.  More specifically: given $\bc_\rho^n\in\mathbb{R}^{r_\rho}$ and $\bc_{g_{\bv}}^n\in\mathbb{R}^{r_g}\;\forall \bv\in \mcV$,  we seek  $\bc_\rho^{n+1}\in\mathbb{R}^{r_\rho}$ and $\bc_{g_{\bv}}^{n+1}\in\mathbb{R}^{r_g}\;\forall \bv\in \mcV$, satisfying
\begin{subequations}
    \label{eq:ROM-mm}
    \begin{align}
    &B_\rho^TMB_\rho \frac{\bc_\rho^{n+1}-\bc_\rho^n}{\Dt}+\sum_{\bv=(v_x,v_y)\in \mcV}\omega_{\bv} B_\rho^T(v_{x}D_x^++v_{y}D_y^+)B_g \bc_{g_\gamma}^{n+1}
\notag\\
    & +B_\rho^T D^\textrm{jump}B_\rho \bc_\rho^{n+1}= -B_\rho^T\Sigma_aB_\rho \bc_\rho^{n+1}+B_\rho^T\bG^{n+1},\\
    \vareps^2&B_g^TMB_g\frac{\bc_{g_{\bv}}^{n+1}-\bc_{g_{\bv}}^n}{\Dt}+\vareps\sum_{\bnu=(\bnu_x, \bnu_y)\in\mcV}(\delta_{\bv \bnu}-\omega_{\bnu})B_g^T(D_{x,{\bnu}_x}^\textrm{up}+D_{y,{\bnu}_y}^\textrm{up})B_g \bc_{g_\gamma}^n
   \notag\\
    &+B_g^T(v_{x}D_x^-+v_{y}D_y^-)B_\rho \bc_\rho^{n+1}
    =-B_g^T(\Sigma_s+\vareps^2\Sigma_a)B_g\bc_{g_{\bv}}^{n+1}.
    \end{align}
\end{subequations}
Similar to the FOM, the Schur complement can again be applied when solving the linear
system \eqref{eq:ROM-mm}, and the resulting $r_\rho\times r_\rho$ problem is 
in the form:  $\mathcal{H}_r^\rho\bc_\rho^{n+1}=\textrm{RHS}_{r,\rho}^n.$ Here
\begin{align}
\label{eq:schur:r:MMD}  \mathcal{H}_r^\rho=&B_\rho^T(M+\Dt\Sigma_a+\Dt D^\textrm{jump})B_\rho \notag\\
  &- \Dt^2(\lgl v_x^2\rgl_{h,\mcV}  D_{r,\rho g,x}^{+}(\Theta_{r,g})^{-1}D_{r,\rho g,x}^{-}
  +\lgl v_y^2\rgl_{h,\mcV} D_{r,\rho g,y}^{+}(\Theta_{r,g})^{-1}D_{r,\rho g,y}^{-}),
\end{align}
where $D_{r,\rho g,\xi}^+=B_\rho^TD_\xi^+ B_g$ and $D_{r,\rho g,\xi}^-=B_g^TD_\xi^- B_\rho$ with $\xi=x,y$ and $\Theta_{r,g}=B_g^T(\vareps^2M+\Dt\Sigma_s+\vareps^2\Dt\Sigma_a)B_g$, therefore  $\mathcal{H}_r^\rho$ is symmetric positive definite, just like its FOM counterpart.

\medskip
\noindent {\bf Online functionalities.} 
This reduced MMD formulation is iteratively called in the Offline training stage, as to be seen in 
%topic of 
Section \ref{sec:offline}. At each iteration, the spaces $U^{\rho}_{h,r}$ and $U^{g}_{h,r}$ are augmented and the reduced quadrature rule $\lgl\cdot\rgl_{h,\mcV_{\textrm{rq}}}$ is updated. At the end of this process with the terminal surrogate spaces $U^{\rho}_{h,r}$ and $U^{g}_{h,r}$,  ROM($\mcV_{\textrm{rq}}; U^\rho_{h,r}, U^g_{h,r}$) can be utilized as a surrogate solver for two purposes. First, we can reconstruct the scalar flux $\rho$ and high order moments of $f$; and second, we can predict solutions $f$ for  $\bv$ unseen in the offline process. We next detail these two functionalities.

%\bc_{g_{\bv}}^{n+1}

To reconstruct $\rho$ and compute the high order moments, we solve ROM($\mcV_{\textrm{rq}};U_{h,r}^\rho,U_{h,r}^g$) to compute $\bc_\rho^n$ and $\bc_g^n$. The scalar flux, the first and the second order moments are approximated as: 
\begin{subequations}
\label{eq:ROM-reconstruction}
\begin{align}
&\rho^n\approx B_\rho {\bc}_\rho^n,\\
&{\lgl f^n v_\xi\rgl = \lgl (\rho^n+\vareps g^n)v_\xi\rgl=\vareps\lgl g^n v_\xi\rgl\approx \vareps B_g\lgl v_\xi {\bc}_{g_{\bv}}^n \rgl_{h,\mcV_{\textrm{rq}}},\quad \xi=x,y,z},
\\%=\vareps B_g\sum_{\bv\in\mcV_{\textrm{rq}}}\omega_{\bv} \bv {\bc}^n_{g_{\bv}},\\
&{\lgl f^n v_\xi v_\eta\rgl = \lgl (\rho^n+\vareps g^n)v_\xi v_\eta\rgl=\lgl v_\xi v_\eta\rgl\rho^n+\vareps\lgl g^nv_\xi v_\eta\rgl}\notag\\
&\qquad\qquad\approx{ \lgl v_\xi v_\eta\rgl B_\rho {\bc}^n_{\rho} +\vareps B_g\lgl v_\xi v_\eta {\bc}_{g_{\bv}}^n \rgl_{h,\mcV_{\textrm{rq}}},\quad \xi,\eta=x,y,z.}
%\notag\\
%&\qquad\qquad=\lgl\bv\otimes\bv\rgl B_\rho{\bc}^n_\rho+\vareps B_g\sum_{\bv\in\mcV_{\textrm{rq}}}\omega_{\bv} \bv\otimes\bv {\bc}^n_{g_{\bv}}.
\end{align}
\end{subequations}
%\fliq{/Q: What is $\bv\otimes\bv$? and does it have same dimension as   $B_g\lgl\bv\otimes\bv {\bc}_{g_{\bv}}^n \rgl_{h,\mcV_{\textrm{rq}}}$?/}
Moreover, higher order moments can be computed similarly by integrating, {\it using the reduced quadrature rule $\lgl\cdot\rgl_{h,\mcV_{rq}}$,} corresponding quantities involving the reduced order solutions. We note that the advantages to reconstruct $\rho$ and high order moments with ROM($\mcV_{\textrm{rq}};U_{h,r}^\rho,U_{h,r}^g$) include computation efficiency, resulting from the adoption of the reduced quadrature rule, and memory saving\footnote{For the FOM, the memory to save the time history of $\rho$ and the high order moments is 
%are 
 of $O(N_{\bx}N_{t})$. In the reduced order reconstruction, 
$O(N_{\bx}r_\rho)$ and $O(N_{\bx}r_g)$ are needed to save $B_\rho$ and $B_g$, while $O(N_tr_\rho)$ and $O(N_tr_g)$  are assigned for the time history of ${\bc}_\rho^n$ and moments of ${\bc}_{g_{\bv}}^n$ (e.g. $\lgl v_x{\bc}_{g_{\bv}}^n\rgl_{h,\mcV_{\textrm{rq}}}$). 
The total memory needed by the reduced order model to reconstruct the time history of $\rho$ is of $O(r_\rho(N_\bx+N_t))$, and that  for the $k^{\rm th}$ order moments following
\eqref{eq:ROM-reconstruction} is of $O(r_g(N_\bx+N_t))$ ($k$ odd) and $O((r_\rho+r_g)(N_\bx+N_t))$ ($k$ even) respectively. These are all significantly smaller than their FOM counterparts assuming $r_\rho, r_g \ll N_\bx \mbox{ or } N_t$.}.

When predicting $f$ for an unseen angular direction $\bv^{\textrm{un}}$,  we solve
\begin{align}
    \vareps^2&B_g^TMB_g\frac{\bc_{g_{{\bv}^{\textrm{un}}}}^{n+1}-\bc_{g_{{\bv}^{\textrm{un}}}}^n}{\Dt}+
    \vareps\left( B_g^T(D_{x,v^{\textrm{un}}_x}^\textrm{up}+D_{y,v^{\textrm{un}}_y}^\textrm{up})B_g \bc_{g_{{\bv}}^{\textrm{un}}}^n-{\bc}_{\lgl \bv\cdot\nabla_x g\rgl}^{n,\textrm{upwind}}\right)
   \notag\\
    &+B_g^T(v_{x}^{\textrm{un}}D_x^-+v_{y}^{\textrm{un}}D_y^-)B_\rho \bc_\rho^{n+1}
    =-B_g^T(\Sigma_s+\vareps^2\Sigma_a)B_g\bc_{g_{{\bv}^{\textrm{un}}}}^{n+1},
    \label{eq:predict_f}\\
\mbox{ with }& {\bc}_{\lgl \bv\cdot\nabla_x g\rgl}^{n,\textrm{upwind}}=\lgl B_g^T(D_{x,{\bnu}_x}^\textrm{up}+D_{y,{\bnu}_y}^\textrm{up})B_g \bc_{g_{\bv}}^n\rgl_{h,\mcV_{\textrm{rq}}}.\notag
\end{align}
In equation \eqref{eq:predict_f}, $c_\rho^{n+1}$ and $\bc^{n,\textrm{upwind}}_{\lgl v\cdot\nabla_x g\rgl}$ can be obtained through  pre- or on-the-fly computations by solving ROM$(\mcV_{\textrm{rq}};U_{h,r}^\rho,U_{h,r}^g)$. The angular flux $f$ for $\bv^{\textrm{un}}$ is approximated by
${\bff}_{\bv^{\textrm{un}}}^n\approx B_\rho {\bc}_\rho^n+\vareps B_g \bc_{g_{{\bv}^{\textrm{un}}}}^n$.

\subsection{Offline algorithm \label{sec:offline}}

%Now we are ready to present each component of the offline algorithm in details. 

%\subsubsection{\fli{(re-edited)} Inputs and initialization}

%\medskip
%\noindent{\bf Inputs and initialization.}
\begin{algorithm}[!ht]
\caption{Offline algorithm}
%: construct the reduced basis.
%The general framework of the proposed RBM.
\begin{algorithmic}[1]
%\STATE{\textbf{Offline algorithm: construct the reduced basis.}}\\
%\rule{0.9\textwidth}{0.4pt}
\STATE{\textbf{Input:} the training parameter sets $\mathcal{T}_\textrm{train}$ and $\mathcal{V}_\textrm{train}$} 
\STATE{\textbf{Step 1 (initialization):} Initialize sampled parameter sets $\mathcal{T}_{\textrm{rb}}^\rho=\emptyset$ and $\mathcal{TV}_{\textrm{rb}}^g=\emptyset$, the reduced quadrature nodes set $\mcV_{\textrm{rq}}$, and  the reduced spaces $U_{h,r}^\rho$ and  $U_{h,r}^g$.}
\STATE{\textbf{Step 2 (greedy iteration):}\FOR{$i=1:\textrm{max number of iterations}$}
    \IF{the stopping criteria are satisfied}
        \STATE{Stop.}
    \ELSE
        \STATE{
        \begin{itemize}
            \item[](i) solve the reduced order problem %determined by the current basis
    ROM($\mcV_{\textrm{train}};U_{h,r}^\rho,U_{h,r}^g$);
    \item[](ii) compute the values of the  $L^1$ importance indicators for $\rho$ and $g$, and greedily pick the most under-resolved time $t^{\textrm{new}}_\rho$ for $\rho$ and the most under-resolved $(t^{\textrm{new}}_g,\bv^{\textrm{new}}_g)$ pair for $g$;
    \item[](iii) update the  parameter sets  $\mathcal{T}_{\textrm{rb}}^\rho$ and $\mathcal{TV}_{\textrm{rb}}^g$ with symmetry-enhancing strategy;
    \item[](iv) update the reduced  quadrature set ${\mathcal V}_\textrm{rq}$ and the corresponding quadrature rule $\langle \cdot\rangle_{h,\mathcal{ V}_\textrm{rq}}$;
    \item[](v) perform the full order solve with the reduced quadrature rule FOM($\mcV_{\textrm{rq}}$) and update the reduced spaces $U_{h,r}^\rho$ and  $U_{h,r}^g$, and the corresponding basis matrices.
        \end{itemize}
        }
    \ENDIF
\ENDFOR}
\STATE{\textbf{Output:} a reduced order solver, determined by ${\mathcal V}_\textrm{rq}$, $\langle \cdot\rangle_{h,\mathcal{ V}_\textrm{rq}}$, and $U_{h,r}^\rho$, $U_{h,r}^g$.}
\end{algorithmic}
\label{alg:offline}
\end{algorithm}
Summarized in Algorithm \ref{alg:offline}, the Offline algorithm starts with the training sets for $t$ and  $\bv$, given as 
\begin{equation*}
%\label{eq:train:vt}
    \mathcal{T}_\textrm{train}=\{t^n,0\leq n\leq N_{t}\},\quad
    \mcV_\textrm{train}=\{\bv_j: 1\leq j\leq N_v\},
\end{equation*}
with some prescribed cardinalities $N_t$ and $N_v$. In preparing for the greedy iteration, we initialize the sampled parameter sets,  ${\mathcal T}_{\textrm{rb}}^\rho\subset\mathcal{T}_\textrm{train}$ and ${\mathcal{TV}}_{\textrm{rb}}^g\subset \mathcal{T}_\textrm{train}\otimes \mcV_\textrm{train}$, as empty.  We use a low order Lebdev quadrature rule (i.e. nodes and weights) to initialize the set of reduced quadrature nodes ${\mathcal V}_{\textrm{rq}}$ and the associated quadrature rule $\langle \cdot\rangle_{h,\mathcal{ V}_\textrm{rq}}$.  Given ${\mathcal V}_{\textrm{rq}}$, we call the full order solver FOM(${\mathcal V}_{\textrm{rq}}$) with the integral replaced by $\langle \cdot\rangle_{h,\mathcal{ V}_\textrm{rq}}$, and obtain the numerical solution $\{\brho^n, \bg^n_{\bv}: 1\leq n\leq N_t,  \forall \bv\in {\mathcal V}_{\textrm{rq}}\}$ which allows us to initiate the reduced spaces and the corresponding snapshot matrices
\begin{equation*}
U_{h,r}^{\rho}=\textrm{span}\{\brho^{N_t}\},\quad U_{h,r}^g= \textrm{span}\{\bg^{N_t}_{\bv}, \; \bv\in {\mathcal V}_{\textrm{rq}}\},
\end{equation*}
\begin{equation*}
    S_\rho = [{\brho}^{N_t}]\in\mathbb{R}^{N_{\bx}\times 1},\quad S_g=[\bg_{\bv}^{N_t}]_{\bv\in\mcV_{\textrm{rq}}}\in\mathbb{R}^{N_\bx\times |\mcV_{\textrm{rq}}|}.
\end{equation*}
{The initial basis matrix $B_\eta$ is obtained by orthonormalizing the columns of $S_\eta$ for $\eta=\rho,g$.}
We are now ready for details of the greedy iteration, with its main components presented below according to the order summarized at the end of Section \ref{sec:rb:alg}. 
%Note that from the second greedy iteration, $U_{h,r}^g$ is determined by the sampled parameter set $\mathcal{T}_{\textrm{rb}}^\rho$ and $\mathcal{TV}_{\textrm{rb}}^g$. In the second greedy iteration, $\max\{|\mathcal{T}_{\textrm{rb}}^\rho,|\mathcal{TV}_{\textrm{rb}}^g|\}<|\mcV_{\textrm{rq}}|$, and as a result, the dimension of $U_{h,r}^g$ in the second iteration is smaller than the dimension of the initial $U_{h,r}^g$. After the second greedy iteration,  the dimension of $U_{h,r}^g$ will grow as iterations continue. %By initializing $\mathcal{T}_{\textrm{rb}}^\rho$ and $\mathcal{TV}_{\textrm{rb}}^g$ as empty sets instead of $\{t^{N_t}\}$ and $\{t^{N_t}\}\times \mcV_{\textrm{rq}}$, we achieve lower dimensional reduced spaces after onffline training.

%%%%%%%%%%%%%%%%%%%%%%%%%%%%%%%%%%%%%%%%%%%%%%%%%%%%%%%%%%%%%%%%%%%%%%%%%%%%%%%%%%%%%%%%%%%%
\subsubsection{L1 importance indicator and symmetry-enhancing parameter selection}
\label{sec:offline_l1}
At every greedy step, the most under-resolved 
parameter values for $\rho$ and $g$ (were the current reduced spaces to be adopted) will be  determined by the $L^1$ importance indicator \cite{chen2021l1,chen2021eim}. Indeed, given the reduced order space $U_{h,r}^\eta$ ($\eta=\rho,g$), its snapshot and orthonormal matrices $S_\eta$ and $B_\eta$, together with the sampled parameter set $\mathcal{T}_{\textrm{rb}}^\rho\subset\mathcal{T}_\textrm{train}$ and $\mathcal{TV}_{\textrm{rb}}^g\subset\mathcal{T}_{\textrm{train}}\otimes\mathcal{V}_{\textrm{train}}$, we invoke ROM($\mcV_\textrm{train}; U^\rho_{h,r}, U^g_{h,r}$) to obtain the reduced order solution $\{(\brho^n_r, \bg_{\bv,r}^n):\;  \forall n=1,\dots N_t,
\forall\bv\in\mcV_\textrm{train}\}.$
They are expanded under the two basis systems as
\begin{equation}
\left\{\left(\brho^n_r=B_\rho \bc_\rho^n=S_\rho \tilde{\bc}_\rho^n, \; \bg_{\bv,r}^n=B_g \bc_{g_{\bv}}^n=S_g \tilde{\bc}_{g_{\bv}}^n\right): \forall n=1,\dots N_t,
\forall\bv\in\mcV_\textrm{train}\right\}.
\label{eq:l1_representation}
\end{equation}
The $L^1$ importance indicator is defined as:
\[
    \Delta_\rho^n =||\widetilde{\bc}_\rho^n||_{1}, \quad 
    \Delta_{g_{\bv}}^n =||\widetilde{\bc}_{g_{\bv}}^n||_{1}.
\]
Here $||\cdot||_1$ represents the $\ell^1$-norm. As shown in \cite{chen2021eim}, $\widetilde{\bc}_\rho^n$ (resp. $\widetilde{\bc}_{g_{\bv}}^n$) represents a Lagrange interpolation basis in the parameter induced solution space $\{\brho_{r}^n: 1\leq n\leq N_t\}$ 
%$\{\brho_{r}^n: 1\leq n\leq N_{t,\textrm{train}}\}$ 
(resp.  $\{\bg_{\bv,r}^n: 1\leq n\leq N_t, \bv\in \mcV_{\textrm{train}}\}$), implying that  the indicator $\Delta_\rho^n$ (resp. $\Delta_{g_{\bv}}^n$) represents the corresponding the Lebesgue constant. The following strategy to select the parameter sample then amounts to controlling the growth of the Lebesgue constants and hence is key toward accurate interpolation.
\begin{subequations}
%    \label{eq:greedy_micro_macro}
    \begin{align}
    &t_\rho^{\textrm{new}}=\textrm{argmax}_{t^n\in\mathcal{T}_{\textrm{train}}\setminus\mathcal{T}_{\textrm{rb}}^\rho}\Delta_{\rho}^n,\label{eq:greedy_rho}\notag\\
    &(t_g^\textrm{new},\bv_g^{\textrm{new}})=\textrm{argmax}_{(t^n,\bv)\in\mathcal{T}_\textrm{train}\otimes\mcV_{\textrm{train}}\setminus\mathcal{TV}_{\textrm{rb}}^g} \Delta_{g_{\bv}}^n.\notag
%    \label{eq:greedy_g}
\end{align}
\end{subequations}

Once these greedy picks are determined, the parameter sample sets will be updated
\begin{equation*}
\mathcal{T}_{\textrm{rb}}^\rho \leftarrow\{t^{\textrm{new}}_\rho\}\bigcup\mathcal{T}_{\textrm{rb}}^\rho, \quad \mathcal{TV}_{\textrm{rb}}^g \leftarrow \left\{(t_g^\textrm{new},\bv_g^\textrm{new}),(t_g^\textrm{new},-\bv_g^\textrm{new}) \right\}\bigcup\mathcal{TV}_{\textrm{rb}}^g.
%\label{eq:update_parameter_set}
\end{equation*}
Similar to the steady state problem \cite{peng2022reduced}, a symmetry enhancing strategy is applied when updating $\mathcal{TV}_{\textrm{rb}}^g$ by adding both $\bv_g^\textrm{new}$ and its opposite angular direction $-\bv_g^\textrm{new}$. This strategy improves the robustness and accuracy of the reduced quadrature rule, especially in the early stage of the greedy algorithm. 

\begin{rem}
The main advantage of the $L^1$ importance indicator is that it is residual free and can be computed fast (also see  \eqref{eq:compute_l1_indicator}). One can alternatively use the residual as an error estimator. However, the RTE is a multiscale transport system and the residual of its numerical method
is not a sharp error estimator. Sharper error estimators can be constructed for transport problems by solving the adjoint problems \cite{hartmann2003adaptive}, and this requires extra cost and will not be pursued in this paper.
\end{rem}

%%%%%%%%%%%%%%%%%%%%%%%%%%%%%%%%%%%%%%%%%%%%%%%%%%%%%%%%%%%%%%%%%%%%%%%%%%
\subsubsection{Reduced quadrature rule construction}

\label{sec:offline_redquad}

When $\bv_g^{\textrm{new}}\notin \mcV_{\textrm{rq}}$, we update the set of the reduced quadrature nodes as
\begin{equation*}
    \mcV_{\textrm{rq}}\leftarrow \{\bv_g^{\textrm{new}},-\bv_g^{\textrm{new}}\}\cup\mcV_{\textrm{rq}}.
    %\label{eq:update_angular_set}
\end{equation*} 
%\sout{Note that if $\bv_g^{\textrm{new}}\in \mcV_{\textrm{rq}}$, there is no update needed.} 
Though with some symmetry built-in at each step, the angular samples in $\mcV_{\textrm{rq}}$ that are greedily picked offline are in general unstructured. 
A stable and accurate numerical quadrature rule associated with these samples, although important to the robustness and accuracy of the proposed reduced order solver,  may not naturally exist.  
To fill this void, we design a least squares strategy to construct a reduced quadrature rule, similar to that for mesh-free numerical methods \cite{fornberg2014spherical} and further propose an algorithm capable of preserving weight positivity.

\begin{thm}
Given an integrable function $f(\bv): \mathbb{S}^2 \rightarrow {\mathbb R}$ and a positive integer $M$, let  $Y_{m,l}$ be the real-valued spherical harmonic function of degree $m$ and order $l$ with $0\le m \le M$ and $-m \le l \le m$. On a (possibly unstructured) grid $\mcV_{\textrm{rq}}$ with cardinality $N_v^{\textrm{rq}}$ and nodes having spherical coordinates $\{(\theta_k,\phi_k)\}_{k=1}^{N_v^{\textrm{rq}}}$, the following {\it reduced quadrature rule}  
\begin{equation}
\langle f\rangle_{h,\mcV_{\textrm{rq}}}=\sum_{k=1}^{N_v^{\textrm{rq}}}\omega_k f(\bv(\theta_k,\phi_k)), \mbox{ with } \omega_{k}=\frac{1}{\sqrt{4\pi}}\mathbb{I}^\dagger_{1,k}
\label{eq:quad-weights-formula}
\end{equation} 
has a degree of exactness $M$. Here $\mathbb{I}$ is a matrix of size ${N_v^{\textrm{rq}}\times (M+1)^2}$ with  
%\begin{subequations}
%\begin{equation}
%\label{eq:interpolation_matrix}
$\mathbb{I}_{ij}= Y_{ml}(\theta_i,\phi_i)$ and $j=m^2+l+m+1.$ It is assumed  $(M+1)^2\leq N_v^{\textrm{rq}}$. 
%\end{equation}
\end{thm}
\begin{proof}
We note that $\mathbb{S}^2=\{\bv=\bv(\theta,\phi)=(\sin(\theta)\cos(\phi),\sin(\theta)\sin(\phi),\cos(\theta)),\theta\in[0,\pi],\phi\in[0,2\pi]\}$ and the real-valued spherical harmonics form an orthogonal basis of $L^2(\mathbb{S}^2)$. We define the following ansatz of order $M$,
\begin{equation}
\label{eq:l:ansatz}
f_{\bbeta}(\bv(\theta,\phi))=\sum_{m=0}^{M}\sum_{l=-m}^m \beta_{m,l}Y_{m,l}(\theta,\phi),
\end{equation}
%Here $Y_{m,l}$ are real-valued spherical harmonics,  with the first few as 
%\begin{align}
%&Y_{0,0}=\sqrt{\frac{1}{4\pi}}, \; Y_{1,-1}=\sqrt{\frac{3}{4\pi}}\sin\theta\sin\phi=v_y, \notag\\ &Y_{1,0}=\sqrt{\frac{3}{4\pi}}\cos\theta=v_z, \; Y_{1,1}=\sqrt{\frac{3}{4\pi}}\sin\theta\cos\phi=v_x.
%\end{align}
%The coefficient $\bbeta_{\textrm{LS}}$ is determined by 
and seek a particular such function with coefficient being the solution to the least squares problem:
%\begin{align} 
%    \label{eq:quad-lst-squares}
\[
    \bbeta_\textrm{LS}=\arg\min_{\bbeta} \sum_{i=1}^{N_v^{\textrm{rq}}}%\bv\in\mcV_{\textrm{rq}}}
    \left|f_{\bbeta}(\bv(\theta_i,\phi_i))-f(\bv(\theta_i,\phi_i))\right|^2 =\arg\min_{\bbeta}||\mathbb{I}\bbeta-\bff||,
    \]
%\end{align}
where $\mathbb{I}\in\mathbb{R}^{N_v^{\textrm{rq}}\times (M+1)^2}$ and $\bff\in\mathbb{R}^{N_v^{\textrm{rq}}}$ satisfy
%\begin{subequations}
%\begin{equation}
%\label{eq:interpolation_matrix}
$\mathbb{I}_{ij}= Y_{ml}(\theta_i,\phi_i),\; \text{with}\;j=m^2+l+m+1,\quad\text{and}\quad
\bff_i = f(\bv(\theta_i,\phi_i)).$
%\end{equation}
%\end{subequations}
One can easily see that 
%\begin{align}
    $\bbeta_{\textrm{LS}}=\mathbb{I}^\dagger \bff,\;$
%\end{align}
where $\mathbb{I}^\dagger$ is the pseudo inverse of $\mathbb{I}$.
The integral $\lgl f\rgl$ is now approximated by the reduced quadrature rule $\langle f\rangle_{h,\mcV_{\textrm{rq}}}$ 
which is nothing but the exact integration of the least squares approximation
\begin{align}
\langle f\rangle_{h,\mcV_{\textrm{rq}}}=
%lgl f\rgl \approx
&\frac{1}{4\pi}\int_0^\pi\int_{0}^{2\pi}f_{\bbeta_{\textrm{LS}}}(\bv(\theta,\phi))d\theta d\phi= \frac{1}{4\pi}\big(\int_{0}^\pi\int_{0}^{2\pi}\beta_{\textrm{LS},00}Y_{0,0}(\theta,\phi)d\theta d\phi\notag\\
&\quad+\sum_{m=1}^M\sum_{l=-m}^m\int_{0}^\pi\int_{0}^{2\pi}\beta_{\textrm{LS},ml}Y_{m,l}(\theta,\phi)d\theta d\phi\big)\notag\\
=&\frac{1}{\sqrt{4\pi}}\beta_{\textrm{LS},00}=\sum_{k=1}^{N_v^{\textrm{rq}}}\frac{1}{\sqrt{4\pi}}\mathbb{I}^\dagger_{1k}f(\bv(\theta_k,\phi_k)).\notag
%\label{eq:quadrature-weight-computation}
\end{align}
From the construction above, one can see that the reduced quadrature rule is exact for polynomials (in $\bv$) up to degree $M$, hence of accuracy order $M$.
\end{proof}

We emphasize that, just like any numerical integration of interpolatory type, the weights are independent of the integrand $f$. In this work, we always assume  $M\geq 3$. As a result, $\lgl v_\xi^2\rgl$ with $\xi=x,y,z$  are computed exactly and they will appear in the diffusion limit.  Additionally $\lgl v_xv_y\rgl=\lgl v_xv_z\rgl=\lgl v_yv_z\rgl=0$ is also exactly computed, and this will ensure the absence of the  cross-derivatives of second order in the reduced order problems \eqref{eq:schur:r:MMD}, as illustrated in  \eqref{eq:schur}. 
We also note that the proposed algorithm can be easily generalized to the 1D slab geometry $\Omega_v=[-1,1]$ and the unit circle $\Omega_v=\mathbb{S}^1$ by replacing the spherical harmonic expansion in \eqref{eq:l:ansatz} with expansions of Legendre polynomials and  trigonometric  functions, respectively. 
%For the 1D slab geometry with $\Omega_v=[-1,1]$, one can choose  the Legendre polynomials as the underlying orthogonal  basis. For the unit circle $\Omega_v=\mathbb{S}^1$, the orthogonal basis can be chosen as the trigonometric  functions.

While the construction of the reduced quadrature has spectral accuracy, it does not guarantee the associated quadrature weights to be non-negative. It is observed numerically that the reduced and full order solvers could blow up when some of quadrature weights are negative. The root of this instability is that the discrete energy $\mathcal{E}_h^n$ defined in \eqref{eq:mu_energy} can be negative in the presence of negative quadrature weights. To preserve stability, we propose a strategy, described in Algorithm \ref{alg:quadrature-nn-weights}, to generate the reduced quadrature rule with non-negative weights. The basic idea is to decrease  the order $M$, when  negative weights are present,  until either all the weights are non-negative for the first time or $M$ reaches a prescribed minimal value $M_{\min}\geq 3$. 
If taking $M=M_{\min}$ still results in negative weights, we simply use the same quadrature rule as the previous greedy iteration and set the weights associated with 
the newly added angular samples to be $0$. Recall that the initial quadrature rule is chosen as a low order Lebedev  quadrature rule with positive quadrature weights. Therefore, the proposed strategy always results in non-negative reduced quadrature weights during the  greedy iterations.

\begin{algorithm}[!ht]
\caption{Iterative procedure to construct reduced quadrature rule with non-negative weights.\label{alg:quadrature-nn-weights}}
\begin{algorithmic}[1]
\STATE{\textbf{Input}:} Current sampled angular points {$\mcV_{\textrm{rq}}=\{\bv_{k_j}\}_{j=1}^{N_{v}^\textrm{rq}}$} and the sampled angular points for the previous iteration $\mcV_{\textrm{rq}}^\textrm{old}$.
Let the reduced quadrature rule for $\mcV_{\textrm{rq}}^\textrm{old}$ be
{$\{\bv_{k_j}^{\textrm{old}},\omega^\textrm{old}_j\}_{j=1}^{N_v^{\textrm{rq,old}}}$}
%{N_{v,\textrm{rq}}^\textrm{old}}$ 
%for $\mcV_{\textrm{rq}}^\textrm{old}$ 
with $\omega_j^\textrm{old}\geq 0,\;\forall j$, the order $M_{\min}$ and $M_{\max}$.
\STATE{Initialize the bool variable $\textrm{Failure}=true$.}
\FOR{$M=M_{\max}:-1:M_{\min}$}
\STATE{Use equation \eqref{eq:quad-weights-formula} to construct an order $M$ reduced quadrature rule $\lgl\cdot\rgl_{h,\mcV_{\textrm{rq}}}$.}
    \IF{All the quadrature weights are non-negative,} 
    \STATE{set $\textrm{Failure}=\textrm{false}$, and break.}
    \ENDIF
\ENDFOR
\IF{$\textrm{Failure}$}
\STATE{set the quadrature weight $\omega_j^\textrm{new}$ for $\bv_{k_j}\in\mcV_{\textrm{rq}}$ as}
\begin{align*}
    \omega_j^\textrm{new}=\begin{cases}
                           0, \quad &\text{if}\quad \bv_{k_j}\not\in \mcV_{\textrm{rq}}^\textrm{old},\\
                           \omega_j^\textrm{old},\quad&\text{if}\quad \bv_{k_j}\in \mcV_{\textrm{rq}}^\textrm{old}.
                          \end{cases}
\end{align*}
\ENDIF
\STATE{\textbf{Output}:} {the quadrature rule  $\{\bv_{k_j}, \omega_{j}^{\textrm{new}}\}_{j=1}^{N_{v}^{\textrm{rq}}}$ for $\mcV_{\textrm{rq}}$ with non-negative weights. }
%for the quadrature nodes in $\mcV_{\textrm{rq}}$.}
\end{algorithmic}
\end{algorithm}

%\fli{/TODO: Output for algorithm 2??/}

%%%%%%%%%%%%%%%%%%%%%%%%%%%%%%%%%%%%%%%%%%%%%%%%%%%%%%%%%%%%%%%%%%%%%%%%%%%%%%%%%%%%%%%%%%%%
\subsubsection{Update of the reduced order spaces} %$U^\rho_{h,r}, U^g_{h,r}$ and the bases
\label{sec:offline_basis_update}

%{\bf Update of the reduced order spaces $U^\rho_{h,r}, U^g_{h,r}$ and the bases: }
Given the sampled parameter  set $\{\mathcal{T}_{\textrm{rb}}^\rho, \mathcal{TV}_{\textrm{rb}}^g\}$, reduced quadrature nodes $\mcV_{\textrm{rq}}$ containing the $\bv-$components of $\mathcal{TV}_{\textrm{rb}}^g$, and the associated quadrature rule $\lgl \cdot\rgl_{h,\mcV_{\textrm{rq}}}$,
we augment the reduced order space $U_{h,r}^\eta$ $(\eta=\rho,g)$ and its corresponding matrices $S_\eta$ and $B_\eta$. 
Indeed, we perform FOM($\mcV_{\textrm{rq}}$) which is affordable thanks to the small size of $\mcV_{\textrm{rq}}$ to obtain the solution snapshots $\brho^n, \bg^n_{\bv}, \forall n=1,\dots, N_t, \forall \bv\in \mcV_{\textrm{rq}}$. We are then ready for the updates.

\medskip
\noindent {\bf Update $U_{h,r}^\rho$ and $B_\rho$.} This will be  done in a straightforward manner, namely
%\begin{equation}
    $U^\rho_{h,r}=\textrm{span}\{\brho^m:  \;t^m \in \mathcal{T}_{\textrm{rb}}^\rho\}$. 
%    \label{eq:rb_space_rho}
%\end{equation}
Correspondingly, the snapshot matrix $S_\rho$ is assembled. 
%Though not being proved, $S_\rho$ is full rank in our numerical experiments.
We then orthonormalize $S_\rho$ through the (reduced) singular value decomposition (SVD):
\begin{equation}
    S_\rho = B_\rho\Lambda_\rho V_\rho^T \in\mathbb{R}^{N_{\bx}\times r_\rho},\label{eq:svd_rho}
\end{equation}
where $B_\rho\in\mathbb{R}^{N_{\bx}\times r_\rho}$, $V_\rho\in\mathbb{R}^{r_\rho\times r_\rho}$,   satisfying $B_\rho^TB_\rho=V_\rho^TV_\rho= I_{r_\rho}$, and $\Lambda_\rho\in\mathbb{R}^{r_\rho\times r_\rho}$ is a diagonal matrix.  The columns of $B_\rho$ form an orthonormal basis of $U_{h,r}^\rho$. As one will see,  the singular values in $\Lambda_\rho$ can be further utilized in the stopping criteria.

%\item 
\medskip
\noindent {\bf Update $U_{h,r}^g$ and $B_g$ via an equilibrium respecting strategy.}
%Instead of a straightforward strategy, we propose an \textbf{} to update 
The update of the reduced order space for $g$ is more subtle.  Particularly, we set
\begin{equation*}
    U^g_{h,r}=\textrm{span}\left\{\{\Dt\Theta^{-1}D_x^-\brho^m,\Dt\Theta^{-1}D_y^-\brho^m: \; t^m\in \mathcal{T}_{\textrm{rb}}^\rho\}\cup\{\bg_{\bv}^m:  (t^m,\bv)\in \mathcal{TV}_{\textrm{rb}}^g\}\right\}.
%    \label{eq:rb_space_g}
\end{equation*}
That is, the reduced order space for $g$  includes not only the sampled $g$-snapshots but also the scaled discrete derivatives of the sampled $\rho$-snapshots. Correspondingly, the snapshot matrix $S_\rho$ is assembled which is further orthonormalized through its own SVD 
\begin{align}
    S_g = B_g\Lambda_gV_g^T\in\mathbb{R}^{N_{\bx}\times r_g},
    \label{eq:svd_g}
\end{align}
where $B_g\in\mathbb{R}^{N_{\bx}\times r_g}$, $V_g\in\mathbb{R}^{r_g\times r_g}$, satisfying $B_g^TB_g=V_g^TV_g=I_{r_g}$. The columns of $B_g$ form an orthogonal basis of $U_{h,r}^g$. 

\medskip
\noindent \textbf{Fast computation of $L^1$ error indicator.}
Using the SVD in \eqref{eq:svd_rho} and \eqref{eq:svd_g}, one can show that $\widetilde{\bc}_\rho^n$ and $\widetilde{\bc}_{g_{\bv}}^n$ in \eqref{eq:l1_representation} satisfy
\begin{equation*}
    \widetilde{\bc}_\rho^n = V_\rho\Lambda_\rho^{-1}c_\rho^n, \quad
    \widetilde{\bc}_{g_{\bv}}^n = V_g\Lambda_g^{-1}c_{g_{\bv}}^n,
\end{equation*}
and as a result $\Delta_\rho^n$ and $\Delta_{g_{\bv}}^n$ can be computed efficiently as
\begin{equation}
\Delta_{\rho}^n =||V_\rho\Lambda_\rho^{-1}\bc_\rho^n||_1\quad\text{and}\quad
\Delta_{g_{\bv}}^n =||V_g\Lambda_g^{-1}\bc_{g_{\bv}}^n||_1.
\label{eq:compute_l1_indicator}
\end{equation}

\begin{rem}
The equilibrium respecting strategy is designed to improve the performance of our method especially in  the diffusive regime.  To see the motivation,  
note that as  
$\vareps\rightarrow0$ and with $\sigma_s>0$, we have
\begin{align*}
    \bg_{\bv}^m \rightarrow -\Sigma_s^{-1}(v_{x}D_x^-+v_{y}D_y^-)\brho^m.
\end{align*}
That is, in the diffusion limit, $\bg ^m$ is a linear combination of the scaled derivatives of $\brho^n$. In  general,  $\vareps$ is small in the diffusive regime yet nonzero, and one would want to consider the relation  in \eqref{eq:schur_step1} instead. Hence 
$\Dt\Theta^{-1}D_x^-\brho^m$ and $\Dt\Theta^{-1}D_y^-\brho^m$ are included to enrich the reduced order space for $g$. Another benefit of such enrichment over including  $\Sigma_s^{-1} D_x^-\brho^m$ and $\Sigma_s^{-1}D_y^-\brho^m$ is to be able to handle the case when $\sigma_s$ is zero in some subregion(s) and  the associated $\Sigma_s$ is singular. It is easy to see that $\lim_{\vareps\rightarrow 0}\Theta=\Dt\Sigma_s$.
\end{rem}

%We  now assemble the snapshot matrix  $S_g$, with its columns being 
%\begin{equation}
%    \Dt\Theta^{-1}D_x^-\brho^m,\Dt\Theta^{-1}D_y^-\brho^m,\; \textrm{with}\;  t^m\in \mathcal{T}_{\textrm{rb}}^\rho, \quad \bg_{\bv}^m \; \textrm{with}\;  (t^m,\bv)\in \mathcal{TV}_{\textrm{rb}}^g.
%\end{equation}
%And $S_g$ 

\begin{rem}
We orthornormalize $S_\rho$ and $S_g$ with SVD, and one can alternatively orthornormalize them with the QR decomposition.  The SVD decomposition provides singular values which can be utilized in the stopping criteria and furnishes a mechanism for efficiently computing the error indicators.
\end{rem}

\begin{rem}
We note that the dimension of $U_{h,r}^g$ resulting from the first greedy iteration will be smaller than its initial dimension. After the first greedy iteration, $U_{h,r}^g$ is determined by the sampled parameter set $\mathcal{T}_{\textrm{rb}}^\rho$ and $\mathcal{TV}_{\textrm{rb}}^g$, while the initial $U_{h,r}^g$ is not and its initial dimension is $|\mcV_{\textrm{rq}}|$. In the first greedy iteration,  $\max\{|\mathcal{T}_{\textrm{rb}}^\rho|,|\mathcal{TV}_{\textrm{rb}}^g|\}<|\mcV_{\textrm{rq}}|$ and this leads to the reduction of dimension of $U_{h,r}^g$ compared with its initialization.
\end{rem}

%\begin{rem}
%When updating $B_g$, we add more than one basis. This could potentially leads to redundancy, and one can remove the basis corresponding to singular values smaller than a tolerance. If this tolerance is chosen too tight, the error of the ROM will get stuck before reaching the desired error tolerance. In all our numerical tests, with a reasonable tolerance, we find only one test drops one basis function. Hence, we decide not to include this truncation step.
%\end{rem}

%%%%%%%%%%%%%%%%%%%%%%%%%%%%%%%%%%%%%%%%%%%%%%%%%%%%%%%%%%%%%%%%%%%%%%%%%%

\subsubsection{Stopping criteria}

The $L^1$ importance indicator identifies the most under-resolved parameter sample(s), but it does not inform us the magnitude of the error. To effectively stop the Offline greedy algorithm, we design the following two-fold stopping criteria. The first criterion, based on the spectral ratio, measures how much new information is added in each greedy iteration. The second criterion, an approximate relative error at the final time, can be  computed efficiently. The Offline greedy algorithm stops when both criteria are satisfied.
\begin{enumerate}
    \item 
{\bf Spectral ratio stopping criterion:} Similar to \cite{peng2022reduced}, we use the spectral ratio as one stopping criterion measuring how much new information is gained by expanding the reduced subspaces. Suppose we are in the $m$-th greedy iteration, with all notation now having  a  superscript $m$. Let $\Lambda_\rho^m$ and  $\Lambda_g^m$ be the diagonal matrix from the SVD in \eqref{eq:svd_rho} and \eqref{eq:svd_g}, with the last diagonal entry as $\sigma^{\rho,m}_{r^m_\rho}$ and $\sigma^{g,m}_{r^m_g}$, respectively.
We define two spectral ratios:
\[
\textrm{ratio}^m_\rho = \frac{\sigma^{\rho,m}_{r^m_\rho}}{Tr(\Lambda_\rho^m)},\qquad    \textrm{ratio}^m_g = \frac{\sigma^{g,m}_{r_g^m}}{Tr(\Lambda_g^m)},
\]
and check whether $\max\{\textrm{ratio}^m_\rho,\textrm{ratio}^m_g\}<\textrm{tol}_\textrm{ratio}$ is satisfied.
%The computational cost to compute $\textrm{ratio}_\rho^m$ and $\textrm{ratio}_g^m$ are $O(r_\rho)$ and $O(r_g)$, respectively.

The spectral ratio criterion itself does not directly estimate the error in the reduced order approximations. For that, we propose the second criterion. 

\item 
{\bf Approximate relative error at the final time with a coarse mesh in $\Omega_v$:}
Recall that in each greedy iteration, we have two sets of approximations for $\rho$ and $g(\cdot, \bv, \cdot)\; \forall\bv \in \mcV_{\textrm{rq}}$. One set, denoted as $\rho^n_{h,r}, g^n_{h,\bv, r}\;\forall\bv\in\mcV_\textrm{train}$, is obtained by calling the reduced order solve ROM($\mcV_\textrm{train}; U^\rho_{h,r}, U^g_{h,r}$) in the greedy sampling. The other set,   denoted as {$\rho^{n,\textrm{FOM}}_{h,\mcV_{\textrm{rq}}}, g^{n,\textrm{FOM}}_{h,\bv, \mcV_{\textrm{rq}}}\;\forall\bv\in\mcV_{\textrm{rq}}$}, is obtained when updating  the reduced order spaces by calling the full order solve FOM($\mcV_{\textrm{rq}}$), with a reduced quadrature rule associated with  $\mcV_{\textrm{rq}}$. Based on these approximations, we define the following to measure the relative errors at the final time $t^{N_t}$:
\begin{subequations}
\label{eq:error_approximation}
    \begin{align}
        &\textrm{Estimator}_{\rho} = \frac{||\rho_{h,r}^{N_t}-\rho_{h,\mcV_{\textrm{rq}}}^{N_t,\textrm{FOM}}||}{||\rho_{h,\mcV_{\textrm{rq}}}^{N_t,\textrm{FOM}}||}, \\
        %\notag\\
        &\textrm{Estimator}_{f} = \max_{\bv\in\mcV_{\textrm{rq}}\cap\mcV_{\textrm{train}}}\frac{||\rho_{h,r}^{N_t}+\vareps g_{h,\bv,r}^{N_t}-\rho_{h,\mcV_{\textrm{rq}}}^{N_t,\textrm{FOM}}-\vareps g_{h,\bv,\mcV_{\textrm{rq}}}^{N_t,\textrm{FOM}}||}{||\rho_{h,\mcV_{\textrm{rq}}}^{N_t,\textrm{FOM}}+\vareps g_{h,\bv,\mcV_{\textrm{rq}}}^{N_t,\textrm{FOM}}||},
        %\notag
    \end{align}
\end{subequations}
and  check whether $\textrm{Estimator}_\rho<\textrm{tol}_{\textrm{error},\rho}\quad\text{and}\quad\textrm{Estimator}_f<\textrm{tol}_{\textrm{error},f}$
are satisfied. 
%The computational costs of computing these two estimators  are $O(N_{\bx})$ and $O(N_{\bx}|\mcV_{\textrm{rq}}|)$, respectively.

\end{enumerate}

The reason why we still need the spectral ratio criterion is that $\mcV_{\textrm{rq}}$ is a coarse mesh in $\Omega_v$, and in the early stage of the greedy algorithm, the full order solution associated with this mesh may not be accurate enough to approximate the full order solution corresponding to the training set which has high resolution in $\Omega_v$. We also want to point out that this error approximation strategy can not be used in the greedy sampling step, as we need an error indicator for all the $v\in\mathcal{V}_{\textrm{train}}$ while the full order solution is only available for $\bv\in\mcV_{\textrm{rq}}$ which have already been
sampled. 

%Note that one can alternatively initialize the parameter sample set with $\mathcal{T}^\rho_{\textrm{rb}}=\{t^{N_t}\}$ and 
%$\mathcal{TV}_{\textrm{rb}}^g=\{(t^{N_t},\bv),\bv\in\mcV_\textrm{pre}\}$
%and use $U_{h,r}^{\rho,\textrm{pre}},U_{h,r}^{g,\textrm{pre}}$ as the initial reduced order space $U_{h,r}^{\rho},U_{h,r}^{g}$. However, this choice would lead to a higher dimensional reduced order spaces ($\textrm{dim}(U_{h,r}^g)\geq |\mcV_{pre}|)$
%subspace 
%especially in the diffusive regime, though it may reduce the number of iterations needed offline. 

%%%%%%%%%%%%%%%%%%%%%%%%%%%%%%%%%%%%%%%%%%%%%%%%%%%%%%%%%%%%%%%%%%%%%%%%%%
\subsection{Computational cost\label{sec:cost}}
Now, we summarize the computational cost of the Online and Offline stages. We will start with the computational cost of the reduced order problem ROM($\mcV; U^\rho_{h,r}, U^g_{h,r}$), which will be used both online and offline. This cost consists of two parts. 
Firstly, before time marching begins, one needs to  assemble the reduced order discrete operators such as $B_\rho^TMB_\rho$, $D_{r,\rho g,x}^\pm$ etc, and the leading order of the cost is $O(\max\{r_\rho,r_g\}^2N_{\bx})$.  Additionally, one needs to invert $\Theta_{r,g}$ and  $\mathcal{H}_r^\rho$. With  Cholesky factorization, the associated cost will be  $O(r_g^3)$ and  $O(r_\rho^3)$, respectively. Secondly, in each time step, with the precomputed Cholesky factor, the cost to solve \eqref{eq:ROM-mm} 
for  $\bc_\rho^{n+1}$ is  $O(r_\rho^2)$, and  the cost to  update $\bc_{g_{\bv}}^{n+1}$ for all $\bv\in \mcV$ based on the known $\bc_\rho^{n+1}$ is $O(\max(r_\rho, r_g)r_g |\mcV|)$.
Hence the total cost over $N_t$ time steps is
$O\big((r_\rho^2+\max(r_\rho, r_g)r_g|\mcV|)N_t\big)$ once the reduced order operators are computed prior to the time marching.
%+O(\max\{r_\rho,r_g\}^2N_{\bx})$ .
%\fliq{/original text/} and this needs to be done once at the beginning of the time marching,   \pzc{In each time step, one needs to invert $\mathcal{H}_r^\rho$ defined in \eqref{eq:schur:r:MMD} and $B_g^TMB_g$, and their LU decompoistion can be computed with $O(r_\rho^3)$ and $O(r_g^3)$ cost, respectively.}
%Since the underlying problem is linear, \pzc{the assembly of reduced order discrete operators and the computation of the LU decomposition can be done once at the beginning of the time marching during the Offline stage.}
%Secondly, in each time step, \pzc{with precomputed LU decomposition}, the cost to solve \eqref{eq:schur:r:MMD} for  $\bc_\rho^{n+1}$ is  $O(r_\rho^2)$, and  the cost to  update $\bc_{g_{\bv}}^{n+1}$ for all $\bv\in \mcV$ based on the known $\bc_\rho^{n+1}$ is $O(r_g^2 |\mcV|)$. Hence the total cost over $N_t$ time steps is $O(N_t(r_\rho^2+r_g^2|\mcV|))$.

\medskip
\noindent {\bf Online Cost.} The computational cost of the Online stage comes from solving the ROM$(\mcV_{\textrm{rq}};U_{h,r}^\rho,U_{h,r}^g)$ in \eqref{eq:ROM-mm} from $t=0$ to $N_t\Dt$, and it is {$O((r_\rho^2+\max(r_\rho, r_g)r_gN_{v}^{\textrm{rq}})N_t)$}
with $N_{v}^{\textrm{rq}}=|\mcV_{\textrm{rq}}|$. 
The computational cost to predict $f$ for an unseen angular direction from $n=0$ to $N_t$ by solving  \eqref{eq:predict_f} is {$O(\max(r_\rho, r_g)r_gN_t)$.
%r_g^2)$.
Here we assume that the reduced order operators are available.}

\medskip
\noindent {\bf Offline Cost.}  We denote the reduced orders for $\rho$ and $g$ in the $m$-th greedy iteration as  $r^m_{\rho}$ and $r^m_{g}$, %\fliq{/unify notation here and section 3.3.4, about m/}
and the number of reduced quadrature nodes by $N_{v,m}^{\textrm{rq}}$. We let $r_m=\max(r^m_{\rho},r^m_{g})$ and $N_v^{\textrm{train}}=|\mcV_\textrm{train}|$. The cost of the $m$-th iteration of the offline
%Offline
greedy procedure 
%algorithm 
in Algorithm \ref{alg:offline} is summarized in Table \ref{tab:offline_cost}, in particular the total computational cost of the Offline stage of the $m$-th iteration is
%is then the cost shown on the last row added over the greedy iterations
\[
    \sum_{m=1}^{N_\textrm{iter}}  \left(O(r_m^2(N_{\bx}+N_v^{\textrm{train}}N_t))+O({
 N_{v,m}^{\textrm{rq}}}N_{\bx}N_t)\right). 
\]
\begin{table}[htbp]
  \centering
 \medskip
    \begin{tabular}{|l|c|c|c|c|c|c|c|c|c|c|c|}
    \hline
 		    & Leading order of the cost  \\ \hline
 {\bf Greedy sampling}:  & \\ 
 Assemble reduced order operators & $O(r_m^2N_{\bx})$\\ 
 {Compute Cholesky factorization of $\mathcal{H}_r^\rho$ and $\Theta_{r,g}$} & {$O(r_m^3)$} \\ {Compute ROM($\mcV_{\textrm{train}};U_{h,r}^\rho,U_{h,r}^g$)} and  error indicators & $O(r_m^2N_v^{\textrm{train}}N_t)$\\ \hline
 {\bf Update $\mcV_{\textrm{rq}}$ and $\lgl\cdot \rgl_{h,\mcV_{\textrm{rq}}}$ if necessary} & $O(N_{v,m}^{\textrm{rq}})$\\ \hline
 {\bf Update reduced order spaces and basis}: & \\
 Solve FOM($\mcV_{\textrm{rq}}$) with AMG- preconditioned CG & $O(N_{v,m}^{\textrm{rq}}N_{\bx}N_t)$ \\
 Update basis with SVD & $O(r_m^2N_{\bx})$ \\ \hline
 {\bf Check stopping criteria} & $O(r_m+N_\bx)$\\ \hline
 {\bf Total cost for the $m$-th iteration} & $O(r_m^2 (N_{\bx}+N_v^{\textrm{train}}N_t))+O(
 %(r_m+
  N_{v,m}^{\textrm{rq}} N_{\bx}N_t) $\\ \hline
% {\bf Total cost if rank is $r$} & $O(r^3 (N_{\bx}+N_v^{\textrm{train}}N_t))+O(r^2 N_{\bx}N_t) )$\\
% \hline
\end{tabular}
\caption{The computational cost of  the $m$-th greedy iteration of the Offline algorithm.  \label{tab:offline_cost}}

\medskip
\begin{tabular}{|l|c|c|c|c|c|c|c|c|c|c|c|}
    \hline
 		    & Leading order of the cost  \\ \hline
Solving ROM($\mcV_{\textrm{rq}};U_{h,r}^\rho,U_{h,r}^g$) & $O(r^2N_v^{\textrm{rq}}N_t)$\\ \hline
 \end{tabular}
 \caption{The computational cost of the Online algorithm. \label{tab:online_cost}}
\end{table}

To estimate the overall offline cost, we assume that the final reduced orders are $r_\rho$ and $r_g$,  and let $r=\max(r_\rho,r_g)$. Given that the total number of greedy iterations $N_\textrm{iter}$ scales linearly with $r$, that $r_m$ scales linearly with $m$, and that in the worst scenario $N_{v,m}^{\textrm{rq}} (\le N_v^{\textrm{train}}) $ scales linearly with $m$, %\pzc{utilizing 
%=\frac{(N_{\textrm{iter}}+1)N_{\textrm{iter}}}{2}=O(r^2)\quad\text{and}\quad\sum_{m%=1}^{N_{\textrm{iter}}} %m^2=\frac{(2N_{\textrm{iter}}+1)(N_{\textrm{iter}}+1)N_{\textrm{iter}}}{6}=O(r^3),$%$} 
we conclude that 
\begin{equation}
\textrm{Offline time of MMD-RBM} = O(r^3 (N_{\bx}+N_v^{\textrm{train}}N_t))+O(r^2 {N_{\bx}}N_t) ).\label{eq:offline-cost}
\end{equation}
%\pzc{The $O(r^2N_v^{\textrm{train}}N_t)$ term in \eqref{eq:offline-cost} comes from the fact that
%$\sum_{k=1}^r k^2 = \frac{(2r+1)(r+1)r}{6}=\frac{r^3}{3}+\frac{r^2}{2}+\frac{r}{6}.$} \fliq{/shouldn't ($O(r^3 N_v^{\textrm{train}}N_t)$ absorb   ($O(r^2 N_v^{\textrm{train}}N_t)$?/}

To put this estimate into context, we compare it with the costs of the POD and the full order model. The offline cost of the vanilla POD is dominated by computing the SVD of the snapshot matrix which is of size $N_{\bx}\times (N_tN_v^{\textrm{train}})$. That cost (of obtaining $U$ and $\Sigma$ in $U\Sigma V^T$) is $O(\max(N_{\bx},N_v^{\textrm{train}}N_t)\times(\min(N_{\bx},N_v^{\textrm{train}}N_t))^2)$ 
\cite{golub2013matrix}. Therefore, the relative offline computational time of the MMD-RBM and the vanilla POD is 
{
\begin{equation*}
     \frac{\textrm{Offline time of MMD-RBM}}{\textrm{Offline time of vanilla POD}} 
    = O\left(\frac{r^2}{N_v^{\textrm{train}}N_\star}+\frac{r^3}{N_\star^2} 
    \right),
\end{equation*}
where   $N_\star=\min(N_{\bx},N_v^{\textrm{train}}N_t)$.} Moreover, we have 
%\fliq{/old text/
%\begin{equation*}
%     \frac{\textrm{Offline time of MMD-RBM}}{\textrm{Offline time of vanilla POD}} 
%    = O\left(\frac{r^2}{N_v^{\textrm{train}}N_m}+\frac{r^2}{N_{\bx}N_m}+\frac{r^3}{N_v^{\textrm{train}}N_tN_m}\right),
%\end{equation*}
%where 
%$N_m=\min(N_{\bx},N_v^{\textrm{train}}N_t)$.}                                       On the other hand, the ratio between the offline cost of the MMD-RBM to the cost \fli{$O(N_v^{\textrm{train}}N_{\bx}N_t)$} for solving FOM($\mcV_{\textrm{train}})$ by brute force is
\begin{equation*}
    \frac{\textrm{Offline time of MMD-RBM}}{\textrm{Time of solving FOM}(\mcV_{\textrm{train}})} =
    O\left(\frac{r^2}{N_v^{\textrm{train}}}+\frac{{r^3}}{N_{\bx}}+\frac{r^3}{N_v^{\textrm{train}}N_t}\right).
\end{equation*}

\begin{rem}
SVD can be computed incrementally \cite{brand2002incremental}, and hence the POD can be more efficient. If the low rank of the snapshot matrix, which is determined by the tolerance in the incremental SVD, is $r$, the associated  cost will be  $O(N_{\bx}N_v^{\textrm{\normalfont train}}N_t r)$. With the same $r$, the relative offline computation between our method and the POD with the incremental SVD is
\begin{equation*}
    \frac{\textrm{\normalfont Offline time of MMD-RBM}}{\textrm{\normalfont Offline time of POD with incremental SVD}} = O\left(\frac{r}{N_v^{\textrm{\normalfont train}}}+\frac{r^2}{N_v^{\textrm{\normalfont train}}N_t}+\frac{{r^2}}{N_{\bx}}\right).
\end{equation*}
One can see that as long as 
%Therefore, if 
$r\ll \min(\sqrt{N_{\bx}},N_v^{\textrm{\normalfont train}},\sqrt{N_v^{\textrm{\normalfont train}}N_t})$, the Offline stage of our method is faster than the POD method with the incremental SVD.
\end{rem}

%%%%%%%%%%%%%%%%%%%%%%%%%%%%%%%%%%%%%%%%%%%%%%%%%%%%%%%%%%%%%%%%%%%%%%%%%%
\section{Numerical examples \label{sec:numerical}}
We demonstrate the performance of the proposed MMD-RBM through a series of numerical examples. 
Throughout this section, the angular training set $\mcV_{\textrm{train}}$ is the set of $N_v=590$ Lebedev quadrature points. We use piece-wise constant polynomials, i.e. $K=0$ in space.
When $\sigma_s$ is constant, we use the following time step to guarantee stability,
%\cite{peng2021asymptotic}
\[
    \Dt=\begin{cases}
        h, &\text{if}\quad \vareps< {0.25\sigma_sh}\\
        0.25\min( \frac{h}{\sqrt{2}},\frac{\vareps h}{{\sqrt{2}\sigma_s}}),\quad&\text{otherwise},
        \end{cases}
%\label{eq:time_step_size}
\]
where $h=\min(\min_{1\leq i\leq N_x}(x_{i+\half}-x_{i-\half}), \min_{1\leq i\leq N_y}(y_{i+\half}-y_{i-\half}))$.  When $\sigma_s$ is spatially dependent, we use the smallest time step size allowed by all $\sigma_s$ values. Throughout this section, vacuum boundary conditions are considered. The constants in the numerical flux \eqref{eq:discrte_g_derivatives} are taken to be $\alpha_x=1/\lgl v_x^2\rgl_h$ and $\alpha_y=1/\lgl v_y^2\rgl_h$. We measure the absolute errors and the relative errors 
of the scalar flux $\rho$ and first order moment $\lgl \bv f\rgl$ as follows, by evaluating the difference between the reduced order solution and a reference solution which is computed by the full order solver with $N_v^\textrm{test}=2072$ Lebedev points denoted collectively as $\mcV_{\textrm{test}}$,
\begin{subequations}
\label{eq:errors}
    \begin{align}
    &\mathcal{E}_\rho=\sqrt{\Dt\sum_{n=1}^{N_t}||\rho^n_{h,\textrm{ROM}}-\rho^n_{h,\textrm{FOM}}||^2},  & \mathcal{R}_\rho & = \frac{\mathcal{E}_\rho}{\sqrt{{\Dt}\sum_{n=1}^{N_t}||\rho^n_{h,\textrm{FOM}}||^2}},\\
    %\notag \\
    &\mathcal{E}_{\lgl \bv f\rgl}=\sqrt{\Dt\sum_{n=1}^{N_t}||\lgl \bv f\rgl^n_{h,\textrm{ROM}}-\lgl \bv f\rgl^n_{h,\textrm{FOM}}||^2},     & \mathcal{R}_{\lgl \bv f\rgl} & = \frac{\mathcal{E}_{\lgl \bv f\rgl}}{\sqrt{{\Dt}\sum_{n=1}^{N_t}||\lgl \bv f\rgl^n_{h,\textrm{FOM}}||^2}}. %\notag 
    \end{align}    
\end{subequations}
Here $\lVert \cdot\rVert$ denotes the $L^2$ norm which is computed as $\lVert \rho \rVert=\sqrt{\int_{\Omega_x}\rho^2d\bx}$ for the scalar function $\rho$ and $\lVert \lgl \bv f\rgl\rVert=\sqrt{\int_{\Omega_x}\lgl \bv_x f\rgl^2+\lgl \bv_y f\rgl^2d\bx}$ 
for the vector function $\lgl \bv f\rgl = (\lgl \bv_x f\rgl,\lgl \bv_x f\rgl)^T$. 
Moreover, we have $|\mcV_{\textrm{test}} \backslash \mcV_{\textrm{train}}| = 2058$.
To demonstrate the ability of our method to predict the angular fluxes at angular directions outside the training set, we solve for $\{f(\bv): \bv \in \mcV_{\textrm{test}}\}$ %$\{f(\bv): \bv \in \mcV_{\textrm{test}} \backslash \mcV_{\textrm{train}}\}$ 
with our ROM and evaluate the worst case absolute and relative errors,
\begin{equation*}
\label{eq:errors_f}
\mathcal{E}_f=\max_{\bv}\sqrt{\Dt\sum_{n=1}^{N_t}||f^n_{h,\bv,\textrm{ROM}}-f^n_{h,\bv,\textrm{FOM}}||^2},\;
\mathcal{R}_f=\frac{\mathcal{E}_f}{\max_{\bv}\sqrt{\Dt\sum_{n=1}^{N_t}||f^n_{h,\bv,\textrm{FOM}}||^2}}.
\end{equation*}
We recall that $r_\rho$ and $r_g$ are the dimensions of the reduced order subspace for $\rho$ and $g$. $N_v^\textrm{rq}$ is the number of nodes in the reduced quadrature rule. 
Finally, we keep track of the data compression efficiency of our ROM via recording the compression ratio (C-R)
\[
\textrm{C-R} = \frac{\textrm{DOFs of ROM}(\mcV_{\textrm{rq}};U_{h,r}^\rho,U_{h,r}^g)}{\textrm{DOFs of  FOM}(\mcV_{\textrm{train}})}
=\frac{r_\rho+N_v^{\textrm{rq}}r_g}{(N_v^{\textrm{train}}+1)N_{\bx}}.
\]
%\fliq{/how is this defined? /}

All these quantities will appear in the tables of this section documenting the performance of the proposed MMD-RBM on various examples. We implement our solvers in the {\tt Julia} programming language. When comparing offline computational cost with the vanilla POD in Section \ref{sec:homo}, the code was run on Michigan State University's HPCC cluster. All the other tests were performed on a Macbook Air laptop with a M1 chip.
%%%%%%%%%%%%%%%%%%%%%%%%%%%%%%%%%%%%%%%%%%%%%%%%%%%%%%%%%%%%%%%%%%%%%%%%%%
\subsection{Homogeneous media\label{sec:homo}}
In the first example, we consider a homogeneous media with $\sigma_s=1$ and $\sigma_a=0$ on the computational domain $[0,2]^2$, uniformly partitioned into $80\times 80$ rectangular elements. We adopt an initial condition $f(\bx,\bv,0)=0$ and a Gaussian source  $G(x)=\exp\left(-100((x-1)^2+(y-1)^2)\right)$. Different values of the Knudsen number $\vareps=1.0$ (transport regime), $\vareps=0.1$ (intermediate regime) and $\vareps=0.005$ (diffusive regime) are considered to benchmark the performance of the proposed algorithm. 
The final time is $T=0.25$ for $\vareps=1.0$ and $0.1$, and it is $T=1.5$ for $\vareps=0.005$. The reduced quadrature rule and reduced spaces are initialized  with $26$ Lebedev points. For the stopping criteria, we set $\textrm{tol}_\textrm{ratio}$ as $1\mathrm{e-4}$, $\textrm{tol}_{\textrm{error},\rho}=1.0\%$, and  $\textrm{tol}_{\textrm{error},f}=2.0\%$.

\medskip
\noindent{\bf Performance of the MMD-RBM:}
The results of the MMD-RBM are presented in Table \ref{tab:example1_mm} and Figure \ref{fig:example1}. In the top row of Figure \ref{fig:example1}, we observe that the reduced order solutions match the full order solutions well. As shown in Table \ref{tab:example1_mm}, the MMD-RBM achieves small relative errors in the scalar flux, the first order moment, and $f$ (w.r.t $\bv\in\mcV_{\textrm{test}}$). The C-R in the ROM is consistently below $0.08\%$. The reduced dimensions $r_\rho$ and $r_g$ decrease as $\vareps$ decreases showcasing our method's capability of numerically capturing the fact that the problem approaches its diffusive limit.
%%%%%%%%%%%%%%%%%%%%%%%%%%%%%%%%%%%%%%%%%%%%%%%%%%%%%%%%%%%%%%%
\begin{table}[htbp]
  \centering
 \medskip
    \begin{tabular}{|l|c|c|c|c|c|c|c|c|c|c|c|}
    \hline
 		       &$r_\rho$ & $r_g$ & $N_v^{\textrm{rq}}$ & C-R & $\mcE_\rho$&   $\mcR_\rho$& 
 		        $\mcE_{\lgl \bv f\rgl}$ & $\mcR_{\lgl \bv f\rgl}$ &
 		        $\mcE_f$  & $\mcR_f$  \\ \hline
 %$\vareps=1$  &	$13$& $51$ $$ &3.49e-5& 0.20\% &3.30e-5&  0.19\% & 2.89e-4 & 1.29\% & 0.21\% & 4.11\%\\ \hline	
 %$\vareps=1$  &	$13$& $52$ & $48$ & 1.29e-5 & 0.22\% & 1.53e-5 & 1.40\% & 8.45e-5 & 1.30\% \\ \hline	
 %$\vareps=0.1$ &	$8$	& $32$ & $40$ & 1.44e-5  & 0.48\% & 4.88e-5 & 1.43\% & 9.78e-4 & 29.56\%\\ \hline
 %$\vareps=0.01$&	$3$	& $12$ & $32$ & 7.80e-5 & 0.48\% & 1.81e-4 & 1.43\% & 7.81e-5 & 0.48\%\\ \hline
  $\vareps=1$  &	$13$& $52$ & $48$ & 0.07\% & 1.29e-5 & 0.22\% & 1.99e-5 & 1.29\% & 1.21e-4 & 1.74\% \\ \hline	
 $\vareps=0.1$ &	$8$	& $32$ & $40$ & 0.03\% & 1.44e-5  & 0.48\% & 6.48e-6 & 1.34\% & 1.05e-4 & 3.16\%\\ \hline
 $\vareps=0.005$&	$3$	& $12$ & $32$ & 0.01\% & 7.86e-5  & 0.48\% & 1.29e-6 & 1.43\% & 7.90e-5 & 0.48\%\\ \hline
 \end{tabular}
 \caption{Dimensions of the reduced order subspaces, $r_\rho$, $r_g$, the number of reduced quadrature nodes $N_v^{\textrm{rq}}$, the testing error and the compression ratio for the homogeneous media example  with the MMD-RBM.}
      \label{tab:example1_mm}
\end{table}
%%%%%%%%%%%%%%%%%%%%%%%%%%%%%%%%%%%%%%%%%%%%%%%%%%%%%%%%%%%%%%%%%%%%%%%
\begin{figure}[]
  \begin{center} 
  \includegraphics[width=0.32\textwidth]{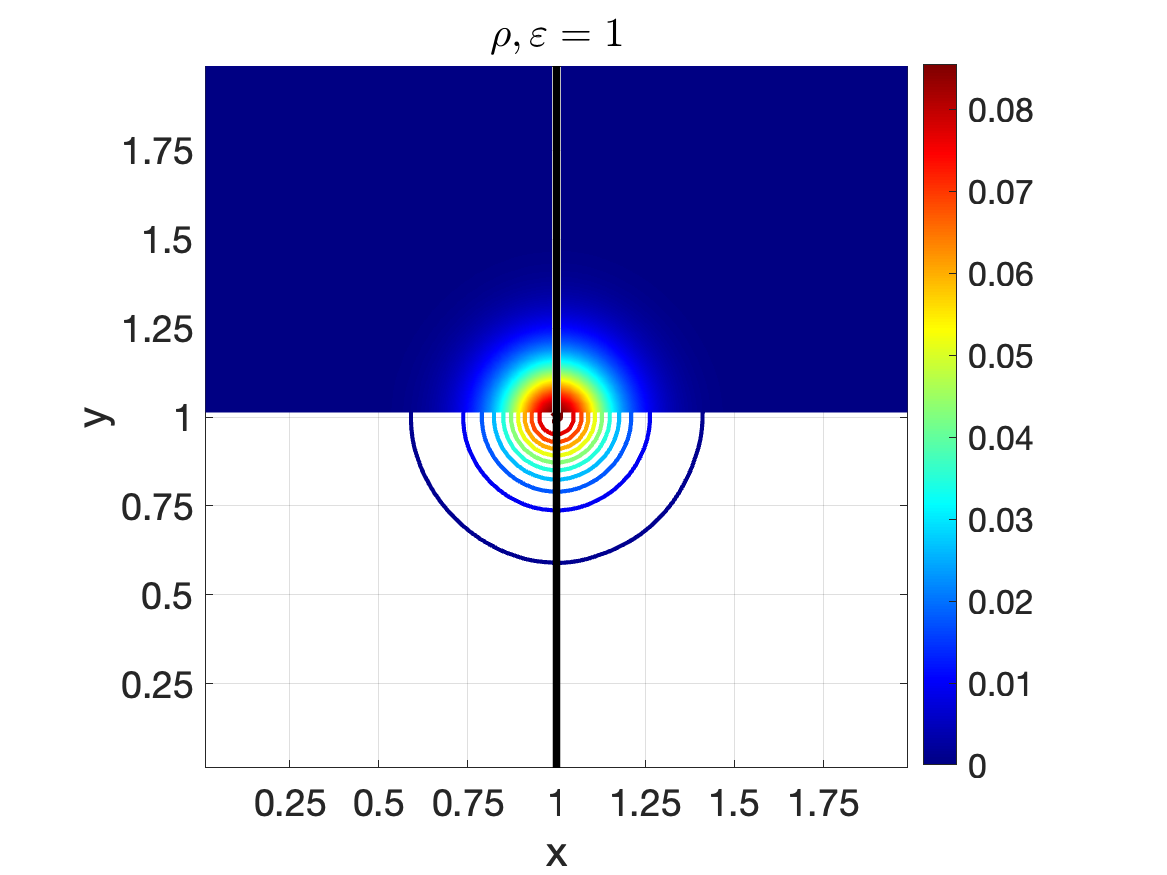}
  \includegraphics[width=0.32\textwidth]{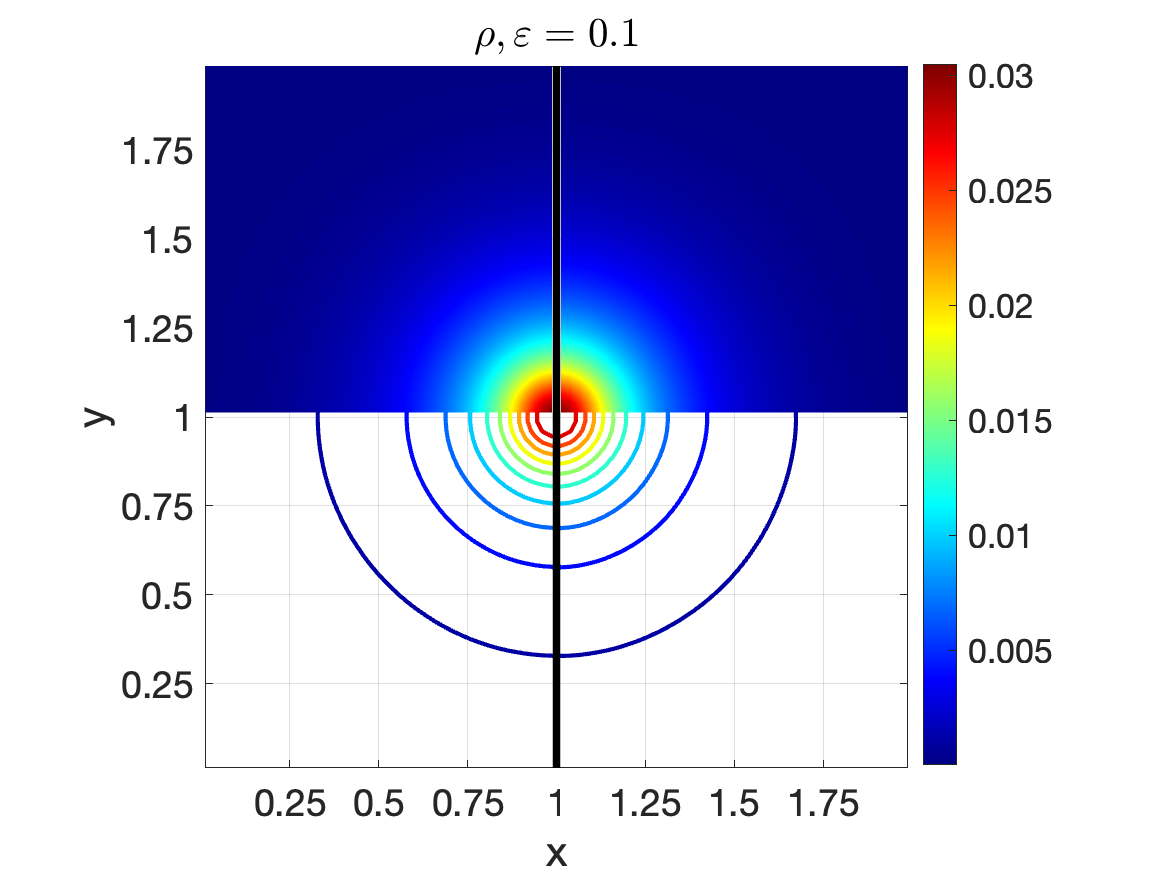}
  \includegraphics[width=0.32\textwidth]{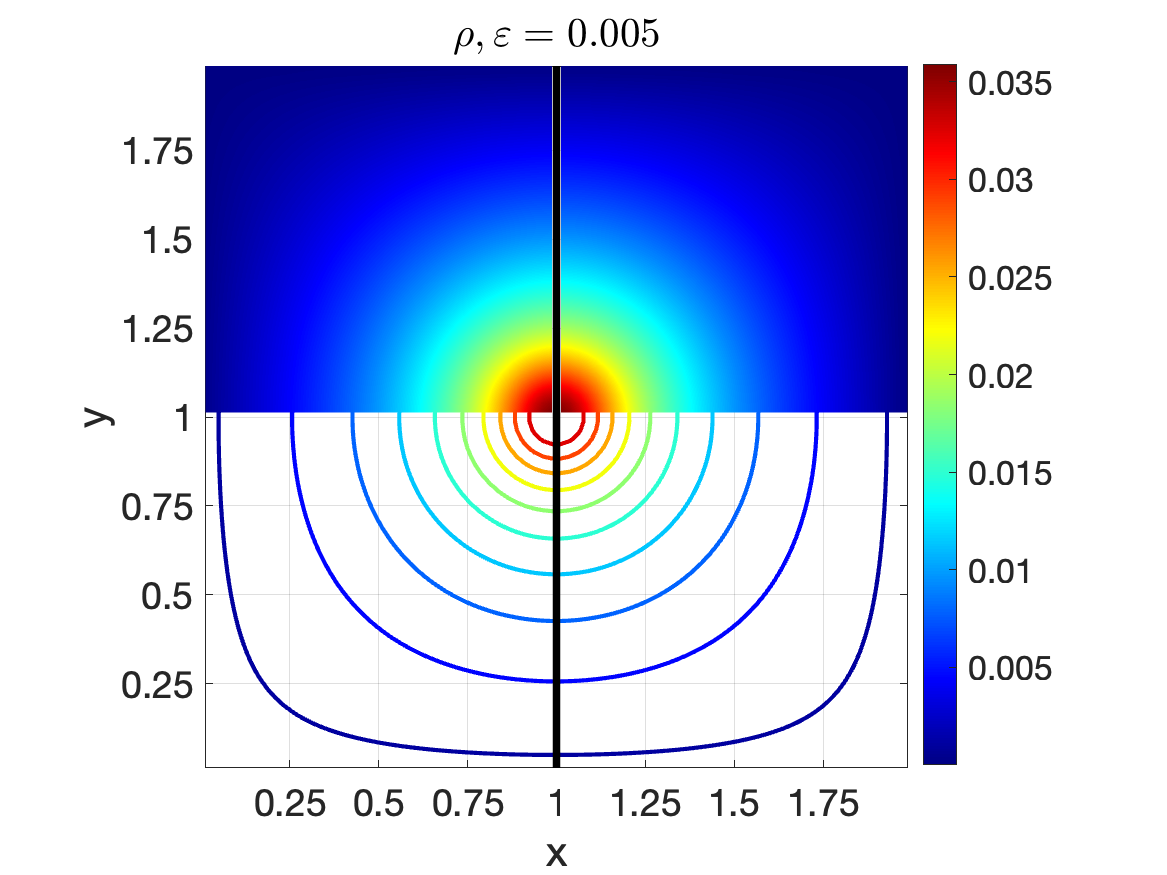}
 \includegraphics[width=0.32\textwidth]{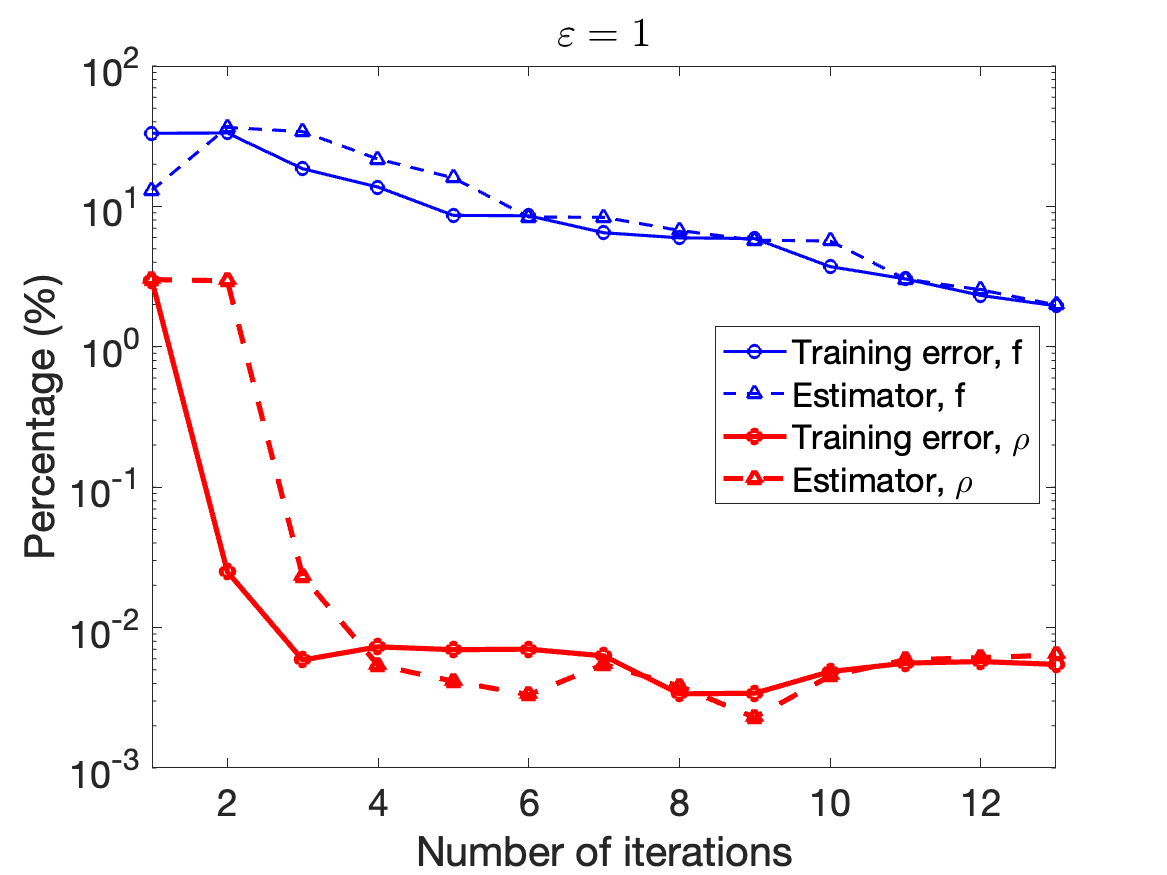}
  \includegraphics[width=0.32\textwidth]{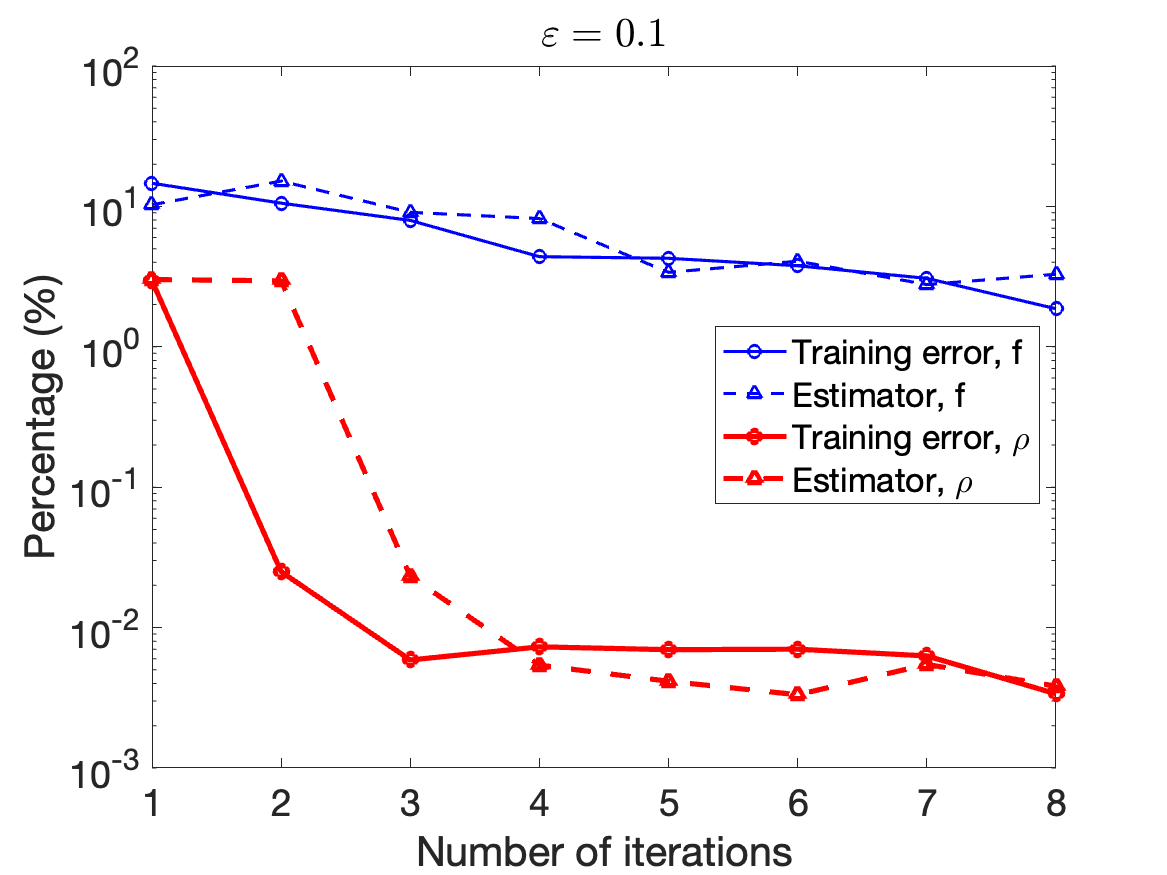}
  \includegraphics[width=0.32\textwidth]{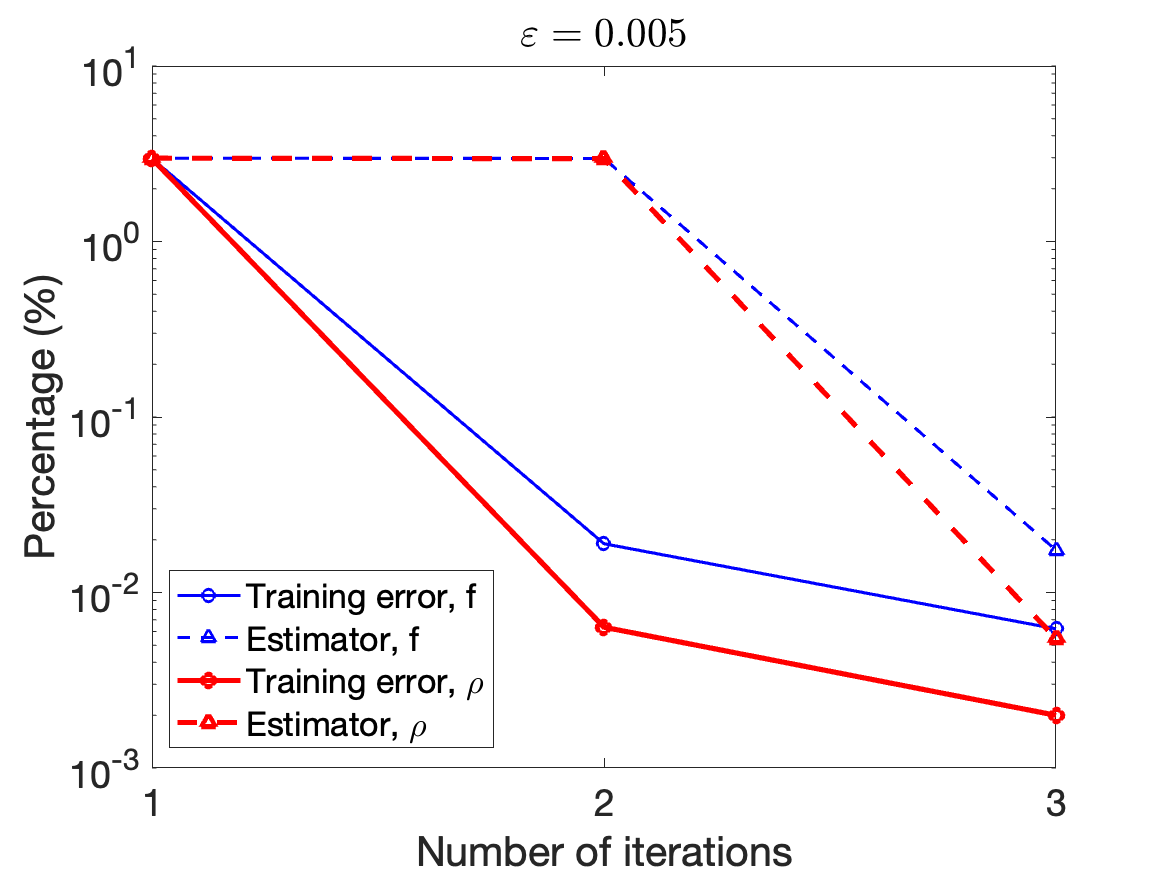}
  \includegraphics[width=0.32\textwidth]{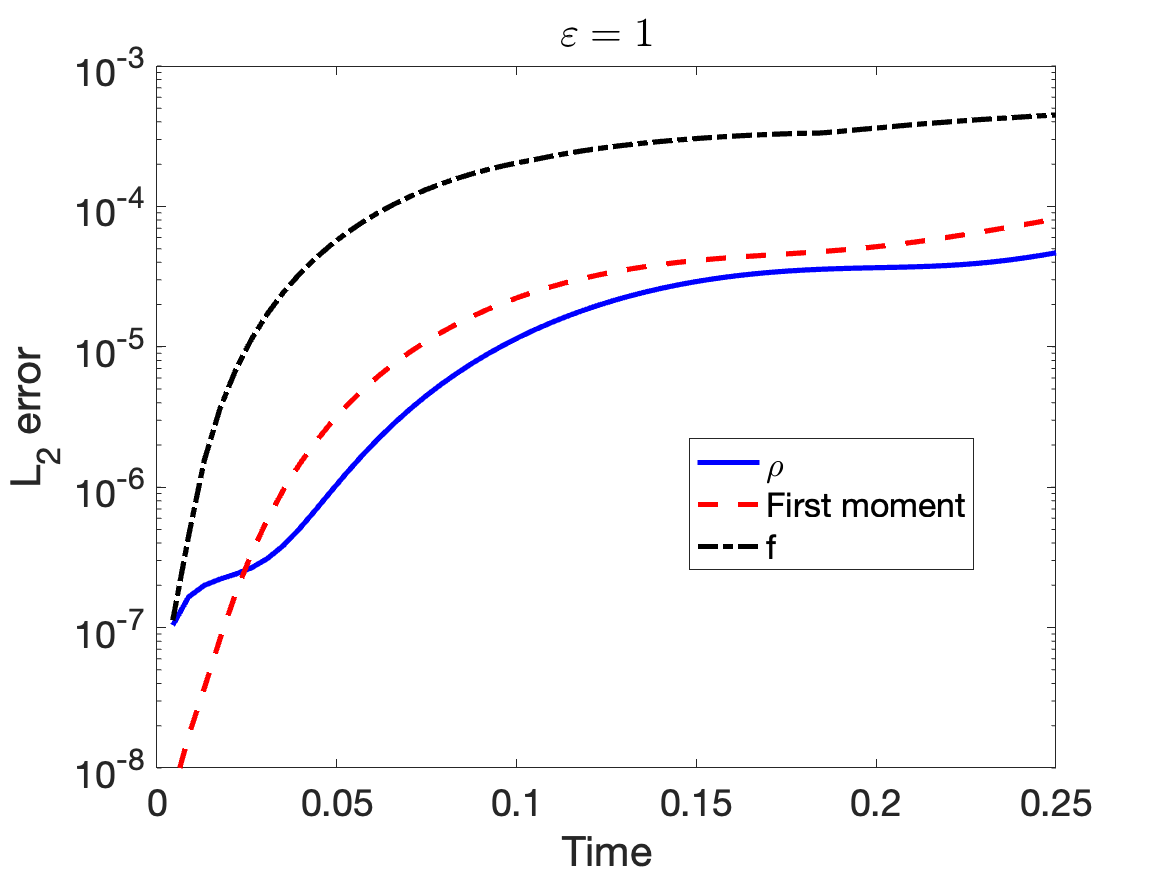}
  \includegraphics[width=0.32\textwidth]{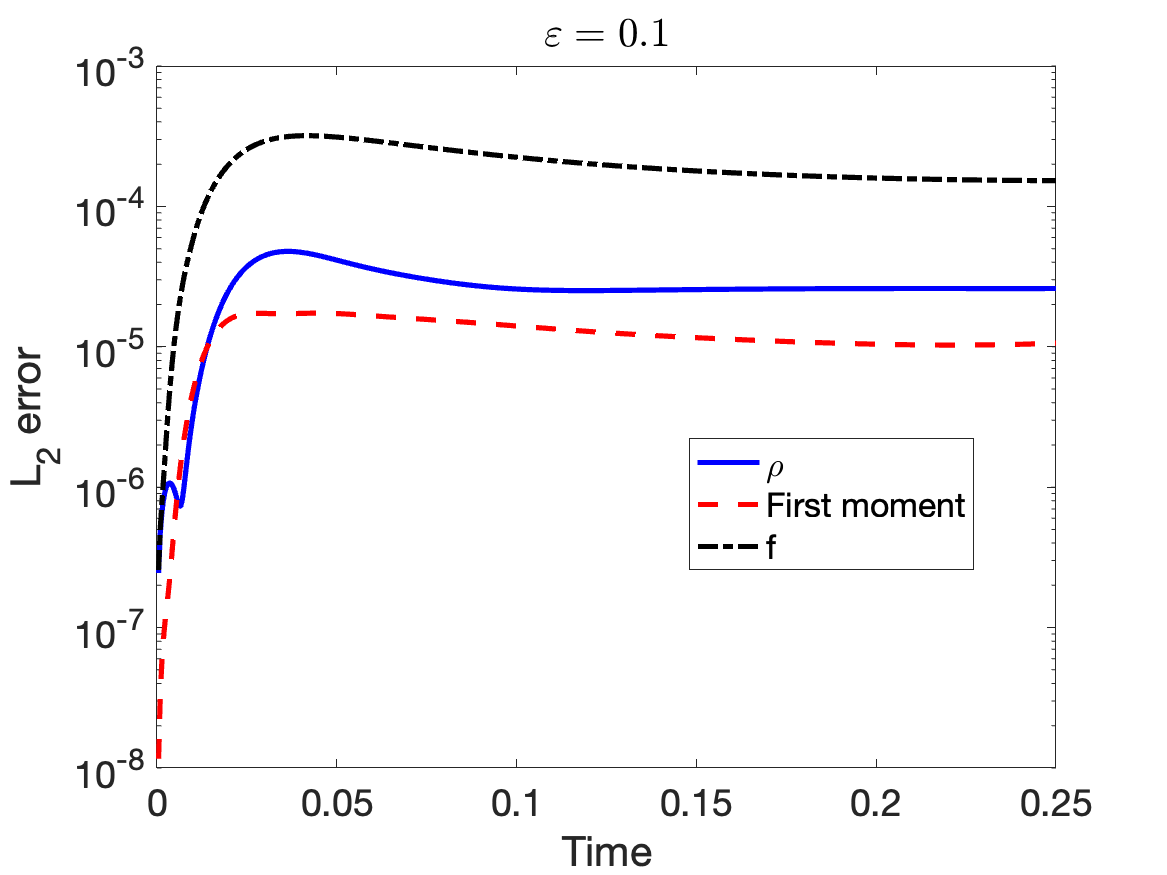}
  \includegraphics[width=0.32\textwidth]{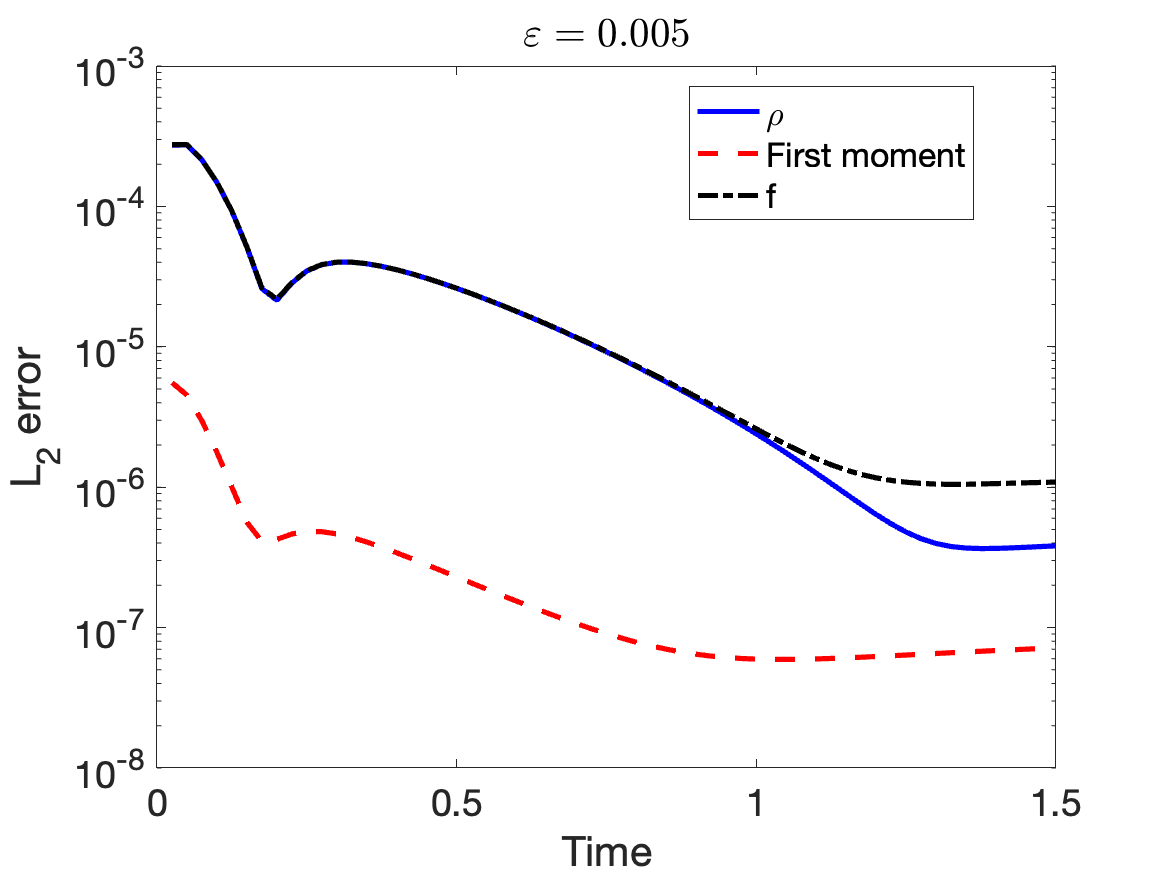}
  \caption{Results for the homogeneous media example. Shown on the top are the reduced order solutions (left) and the full order solutions (right). %On
   In the middle row are the relative training errors of $\rho$ and $f$ at the final time and values of our error estimators. Shown on the bottom are the error histories with respect to time, when we compute the scalar flux $\rho$, first order moment $\lgl \bv f\rgl$ and predict $f$ at unseen angular directions $\bv\in\mcV_{\textrm{test}}$. \label{fig:example1}}
  \end{center}
\end{figure}
%%%%%%%%%%%%%%%%%%%%%%%%%%%%%%%%%%%%%%%%%%%%%%%%%%%%%%%%%%%%%%%%%%%%%%%
%\begin{figure}[]
%  \begin{center} 
%  \includegraphics[width=0.32\textwidth]{pictures/example1_fom_log_1.0.png}
%  \includegraphics[width=0.32\textwidth]{pictures/example1_fom_log_0.1.png}
%  \includegraphics[width=0.32\textwidth]{pictures/example1_fom_log_0.01.png}
%    \includegraphics[width=0.32\textwidth]{pictures/example1_rom_log_1.0.png}
%  \includegraphics[width=0.32\textwidth]{pictures/example1_rom_log_0.1.png}
%  \includegraphics[width=0.32\textwidth]{pictures/example1_rom_log_0.01.png}
%  \caption{Comparison between the reduced order solution and the full order solution under log-scale for the example in Section \ref{sec:homo}. Top row: FOM. Bottom row: the MMD-RBM. \label{fig:example1_log}}
%  \end{center}
%\end{figure}

In the middle row of Figure \ref{fig:example1}, we present the training history of convergence. The relative training errors at the final time are defined as
\begin{equation}
   \mathcal{R}^{N_t}_\rho=||\rho^{N_t}_{h,\textrm{ROM}}-\rho^{N_t}_{h,\textrm{FOM}}||/||\rho^{N_t}_{h,\textrm{FOM}}||,\quad
    \mathcal{E}_{f}^{N_t}=\max_{v\in\mcV_{\textrm{train}}}||f_{h,\bv,\textrm{ROM}}^{N_t}-f_{h,\bv,\textrm{FOM}}^{N_t}||/||f_{h,\bv,\textrm{FOM}}^{N_t}||.
\end{equation}
%\pzc{(defined by replacing the summation in \eqref{eq:errors} with $n=N_t$)} \fliq{/Zhichao: please provide explicit definition here./} 
 The training errors at the final time and the error estimators  in \eqref{eq:error_approximation} are plotted with respect to the number of greedy iterations. We can see that as the number of greedy iterations grows, our estimators approximate the relative training errors at the final time well. Overall, the relative training errors for $\rho$ and $f$ decrease.   In the bottom row of Figure \ref{fig:example1},
%Bottom, 
we plot the error history, as time evolves, of $\rho$, $\lgl \bv f\rgl$ and $f$ (w.r.t $\bv\in\mcV_{\textrm{test}}$). 
%The errors at the time $t^n$ are defined as 
%\begin{align}
%    \mathcal{E}_\rho^n=||\rho^n_{h,\textrm{rom}}&-\rho^n_{h,\textrm{ref}}||, \quad
%    \mathcal{E}_{\lgl \bv f\rgl}^b= ||\lgl \bv f\rgl^n_{h,\textrm{rom}}-\lgl \bv f\rgl^n_{h,\textrm{ref}}||,\notag\\
%    &\mathcal{E}_f=\max_{\bv\in\mcV_{\textrm{test}}}||f^n_{h,\bv,\textrm{rom}}-f^n_{h,\bv,\textrm{ref}}||.
%\end{align} 
%For $\vareps=1.0$, the errors grow and the growth slows down as time evolves. For $\vareps=0.1$, the errors grow fast initially and then   plateau  as time evolves. For $\vareps=0.01$, errors decay over time. Initially, the errors for $\vareps=0.01$ are larger compared with $\vareps=1.0$ and $\vareps=0.1$, and we suspect that this is due to the larger time step size used for $\vareps=0.01$. 
It is clear that, across different regimes, the errors either grow and then plateaus at the level of the prescribed error threshold, or decrease from that level. 
%The ROM is highly effective in capturing the low rank solution manifold in the equilibrium state, and across all regime, the magnitude of the error is reasonably bounded by the error threshold used.

%%%%%%%%%%%%%%%%%%%%%%%%%%%%%%%%%%%%%%%%%%%%%%%%%%%%%%%%%%%%%%%%%%%%%%%
\begin{figure}[]
  \begin{center} 
\includegraphics[width=0.32\textwidth]{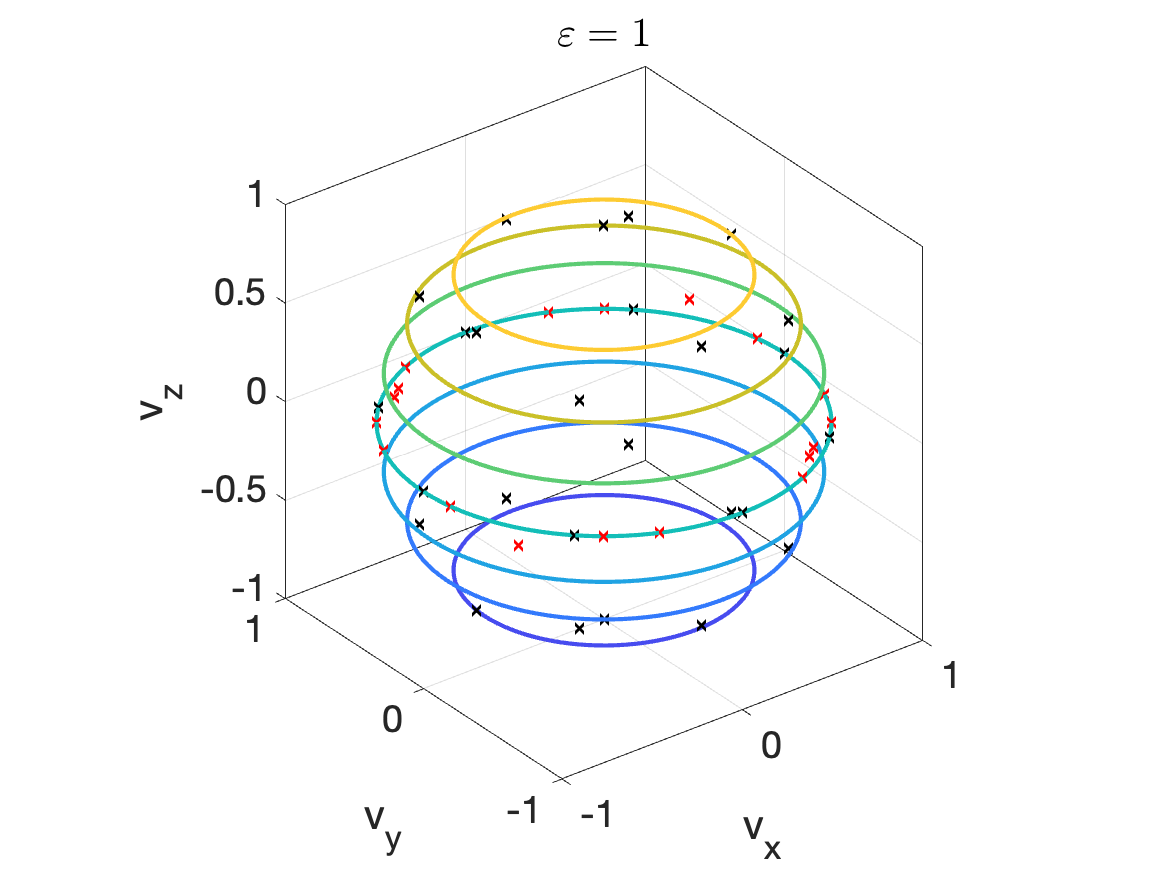}
  \includegraphics[width=0.32\textwidth]{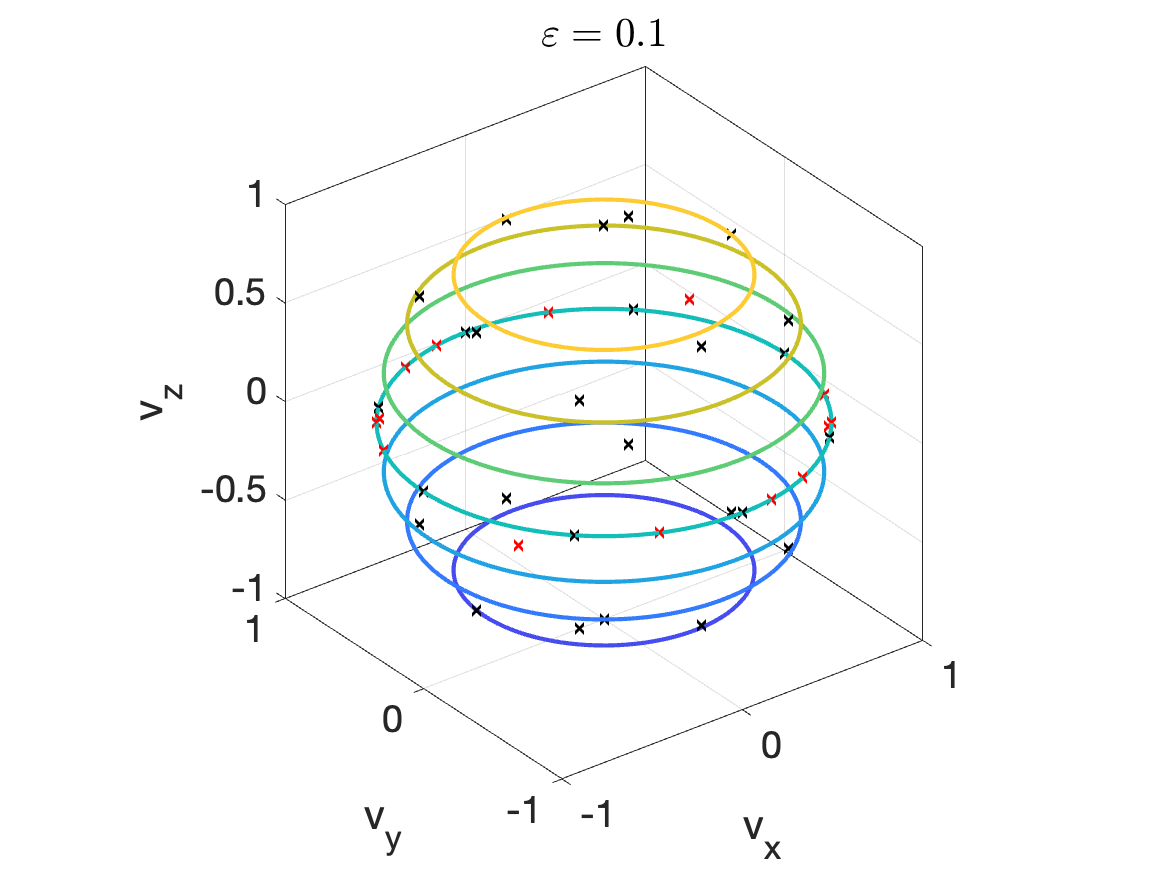}
  \includegraphics[width=0.32\textwidth]{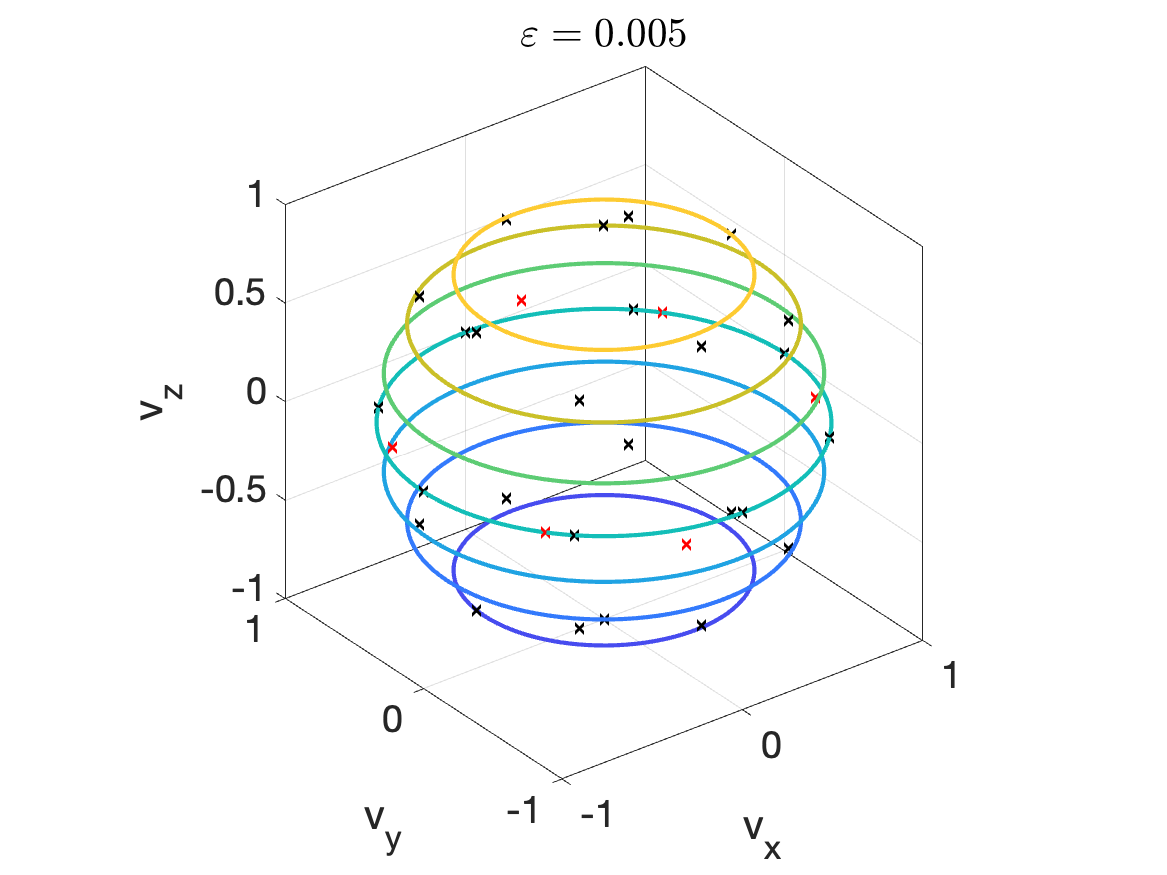}
  \includegraphics[width=0.32\textwidth]{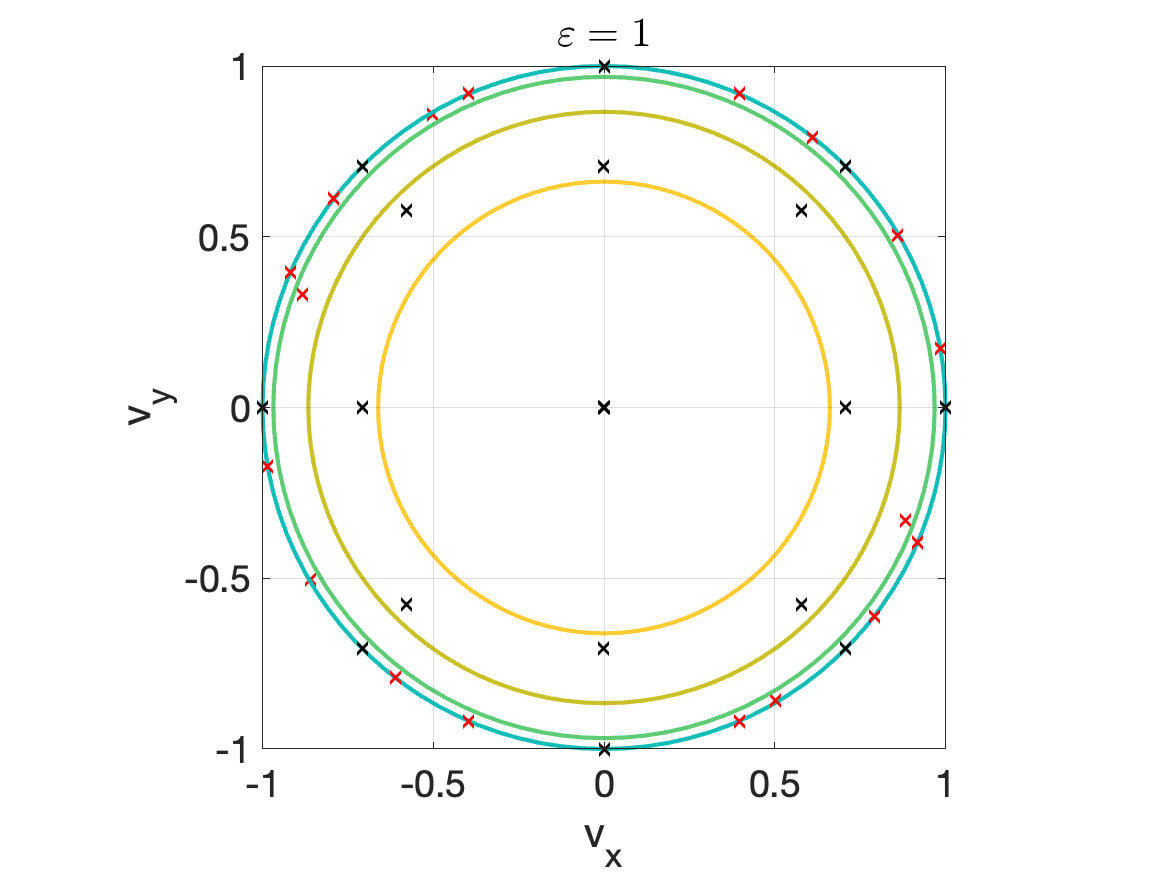}
  \includegraphics[width=0.32\textwidth]{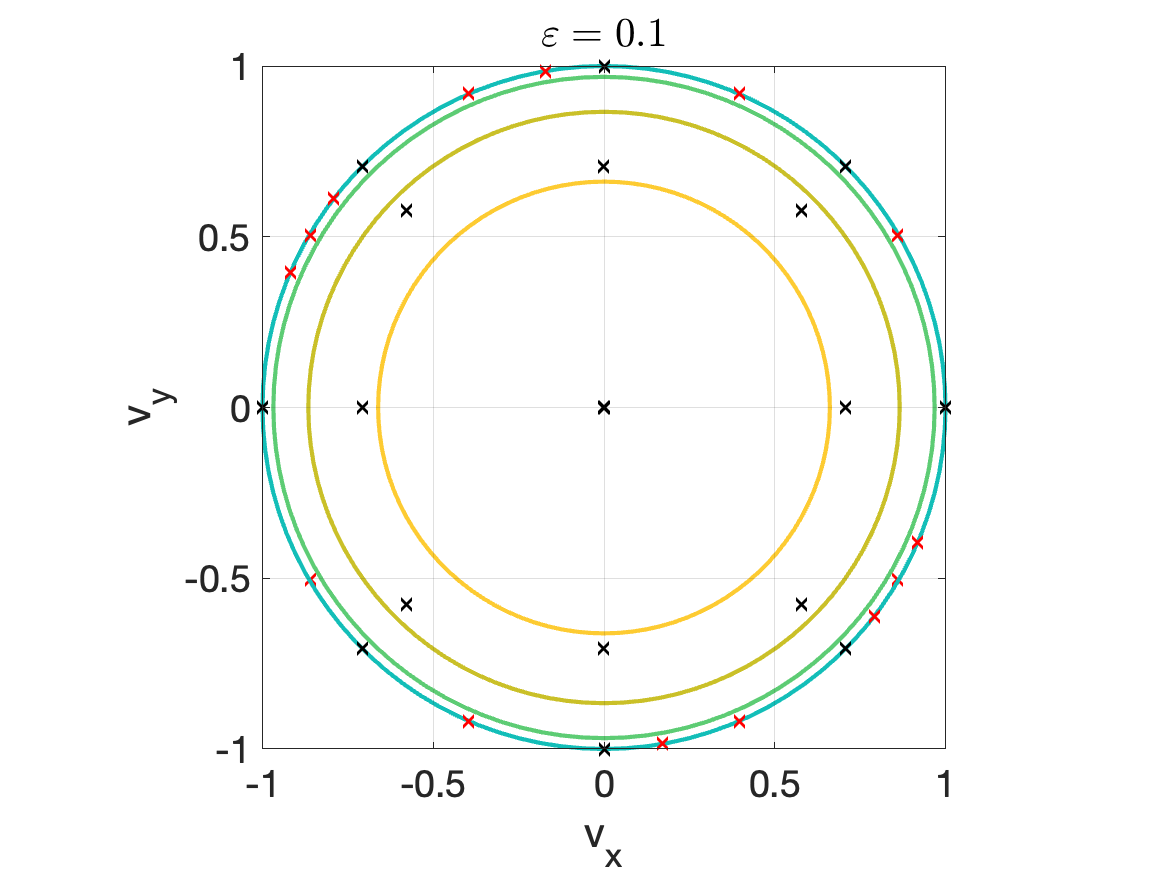}
  \includegraphics[width=0.32\textwidth]{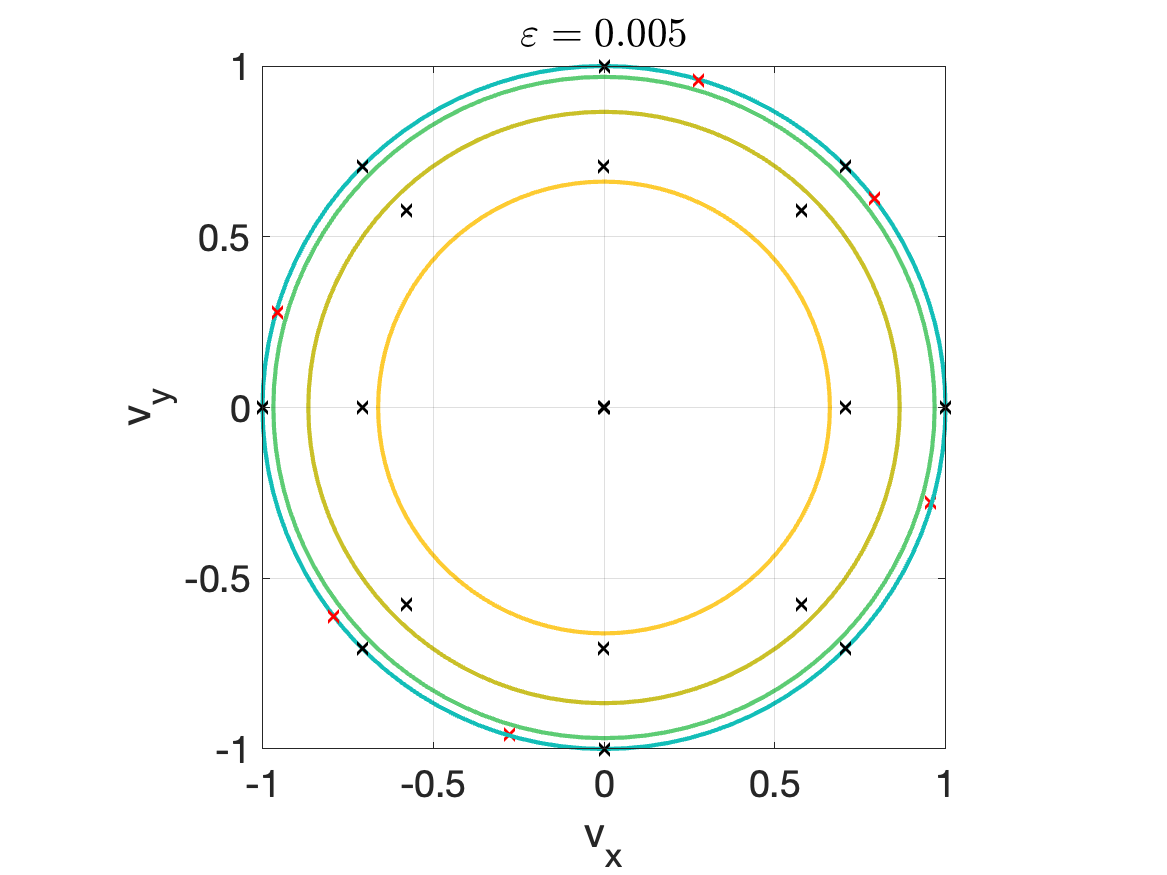}
  \caption{The reduced quadrature nodes on the unit sphere (Black for points in the initial reduced quadrature nodes, and Red for those sampled by the greedy algorithm) and these nodes with a view from the north pole. $\vareps=1.0, 0.1, 0.005$  from left to right. \label{fig:example1-quad}}
  \end{center}
\end{figure}
In Figure \ref{fig:example1-quad}, we present the sampled angular points when the stopping criteria are satisfied. The number of quadrature points in the reduced quadrature rule generated by MMD-RBM are $48$ for $\vareps=1$, $40$ for $\vareps=0.1$ and $32$ for $\vareps=0.005$. {We can see that the sample points are fairly uniform on the sphere for this homogeneous case.}

\medskip
\noindent{\bf Benefit of the equilibrium-respecting strategy:}
We demonstrate the benefit of the equilibrium respecting strategy, that is  the inclusion of $\{\Dt\Theta^{-1}D_x^-\brho^m,\Dt\Theta^{-1}D_y^-\brho^m,t^m\in\mathcal{T}_{\textrm{rb}}^\rho\}$ %the following functions 
when updating the reduced order space $U_{h,r}^g$.
%\begin{align}
%\label{eq:addb}
%$\{\Dt\Theta^{-1}D_x^-\brho^m,\Dt\Theta^{-1}D_y^-\brho^m,t^m\in\mathcal{T}_{\textrm{rb}}^\rho\}$.
%\end{align}
Without these extra functions, 
%If these are not included, 
we report in Table \ref{tab:example1_mm_no_rhox} the dimensions of the reduced order subspaces and the errors 
%are presented in Table \ref{tab:example1_mm_no_rhox} 
when the stopping criteria are the same. Comparing with Table \ref{tab:example1_mm}, we see that
when $\varepsilon=0.1$ and $\varepsilon=0.005$ including derivatives of $\rho$ in $U_{h,r}^g$ leads to smaller values of $r_\rho$, {$N_v^{\textrm{rq}}$} and comparable errors. Having smaller $r_\rho$ values is particularly beneficial since the cost of solving the reduced order problem for one time step scales roughly as $O(r^3 N_v^{\textrm{rq}})$ and the size of the reduced order operator in \eqref{eq:schur:r:MMD} is $r_\rho\times r_\rho$. This advantage is particularly pronounced in the more diffusive regime with $\varepsilon=0.005$.
%When $\varepsilon=1.0$, the two strategies result in similar values of $r_\rho$, while the equilibrium respecting strategy leads to a larger $r_g$ and achieves similar test accuracy.  
%An adaptive strategy to add derivatives terms of $\rho$ in $U_{h,r}^g$ may better balance the online efficiency and accuracy, and we leave that for future investigation. 

%%%%%%%%%%%%%%%%%%%%%%%%%%%%%%%%%%%%%%%%%%%%%%%%%%%%%%%%%%%%%%%
\begin{table}[htbp]
  \centering
 \medskip
    \begin{tabular}{|l|c|c|c|c|c|c|c|c|c|c|c|}
    \hline
 		       &$r_\rho$ & $r_g$ & $N_v^{\textrm{rq}}$ & C-R & $\mcE_\rho$&   $\mcR_\rho$& 
 		        $\mcE_{\lgl \bv f\rgl}$ & $\mcR_{\lgl \bv f\rgl}$ &
 		        $\mcE_f$  & $\mcR_f$  \\ \hline
 $\vareps=1$  &	$14$& $28$ & $52$ & 0.04\% & 1.01e-5& 0.18\% &2.05e-5&  1.33\% & 1.40e-4 & 2.01\% \\ \hline			
 $\vareps=0.1$ &	$16$	& $32$ & $50$ & 0.04\% & 2.15e-5 & 0.72\% &8.19e-6&  1.70\% & 3.96e-5 & 1.20\% \\ \hline	
 $\vareps=0.005$&	$9$	& $18$ & $38$ & 0.02\% & 2.18e-5 & 0.13\% &4.77e-7&  0.53\% & 2.18e-5&  0.13\% \\ \hline
 \end{tabular}
 \caption{Dimensions of the reduced order subspaces, $r_\rho$, $r_g$, the number of reduced quadrature nodes $N_v^{\textrm{rq}}$, the testing error and the compression ratio for the homogeneous media example  with the ROM constructing the reduced space for 
$g$ only with snapshots of $g$.}
      %\caption{Testing errors at the final time and relative online computational time \fliq{/which column??/} for the homogeneous media example,  with the ROM constructing the reduced space for 
      %and the ROM constructs basis for 
      %$g$ only with snapshots of $g$. }
%      Here, $r_\rho$ and $r_g$ are the dimensions of the reduced order subspace for $\rho$ and $g$. C-R is the compression ratio defined in \eqref{eq:compression-ratio}. 
      \label{tab:example1_mm_no_rhox}
\end{table}
%%%%%%%%%%%%%%%%%%%%%%%%%%%%%%%%%%%%%%%%%%%%%%%%%%%%%%%%%%%%%%%%%%%%%%%
\begin{figure}[htbp]
  \begin{center} 
\includegraphics[width=0.45\textwidth]{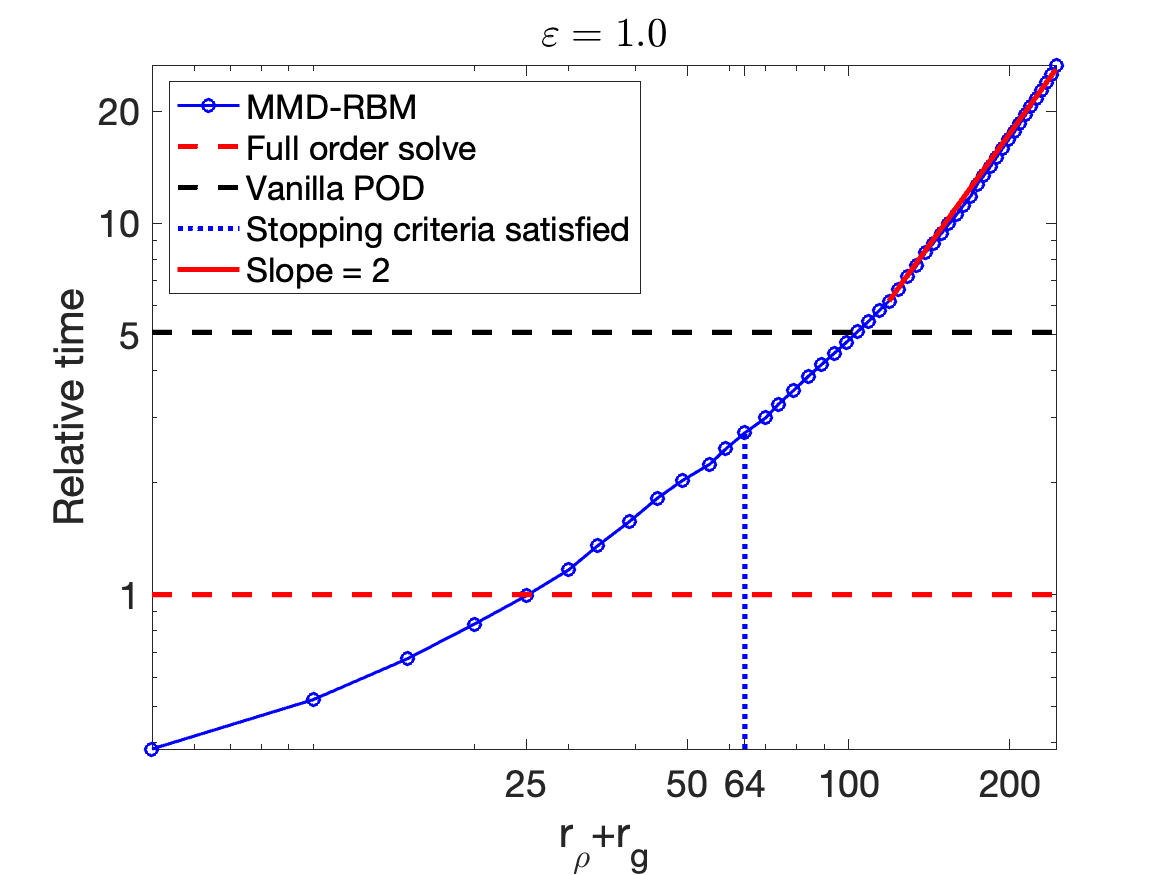}
  \includegraphics[width=0.45\textwidth]{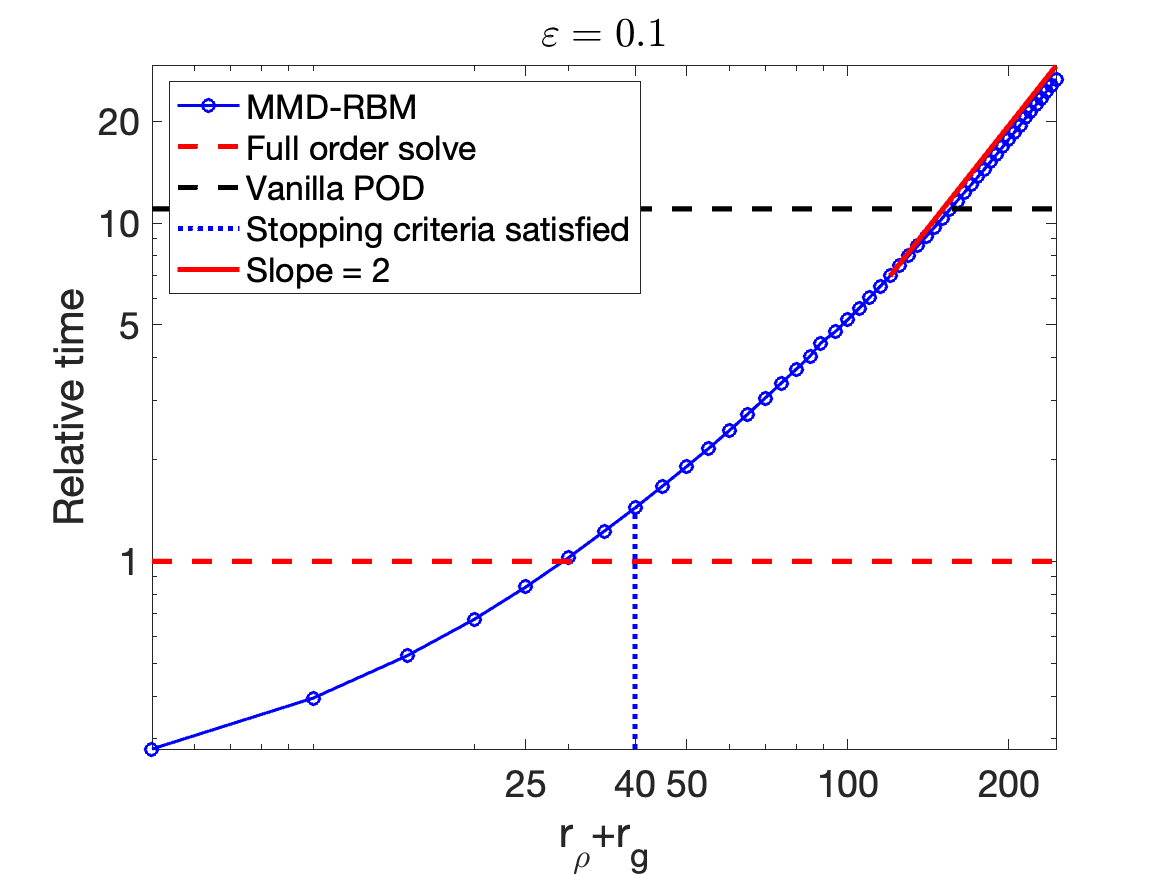}
  \includegraphics[width=0.45\textwidth]{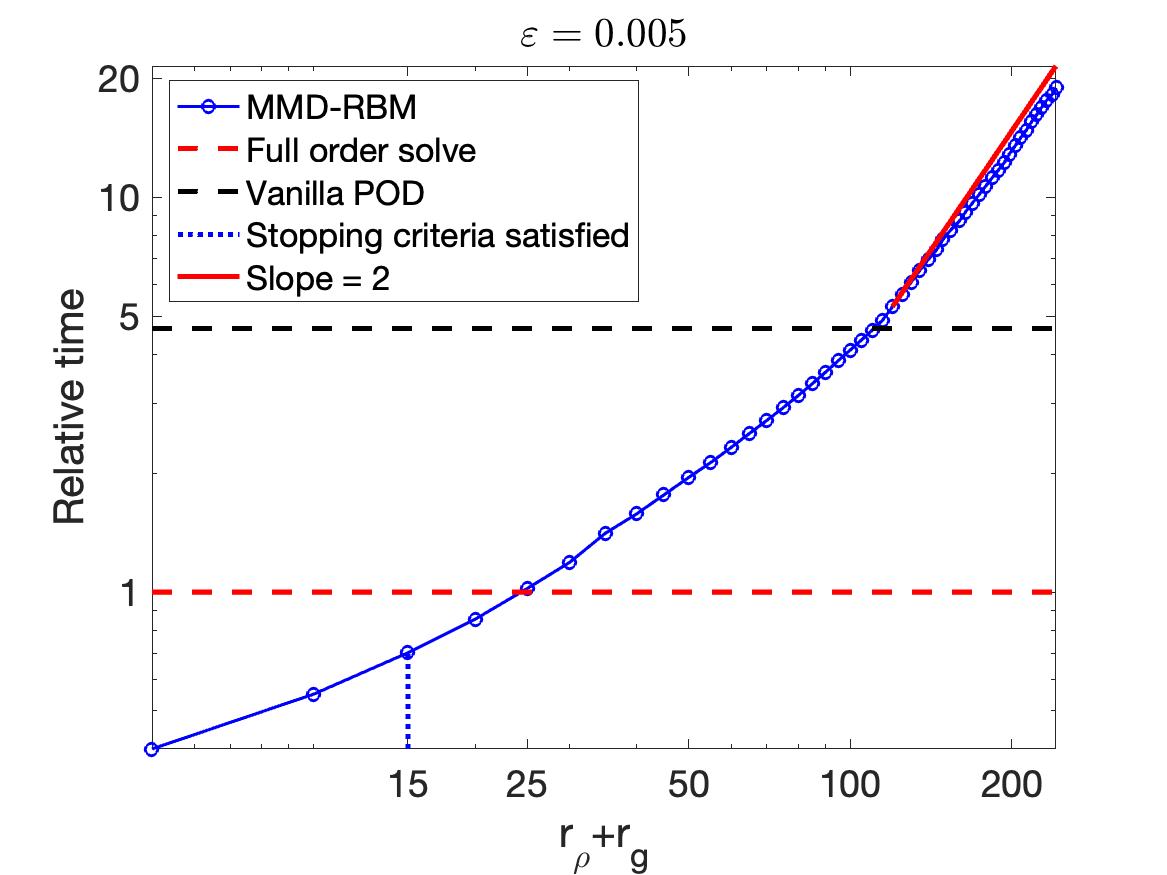}
  \includegraphics[width=0.45\textwidth]{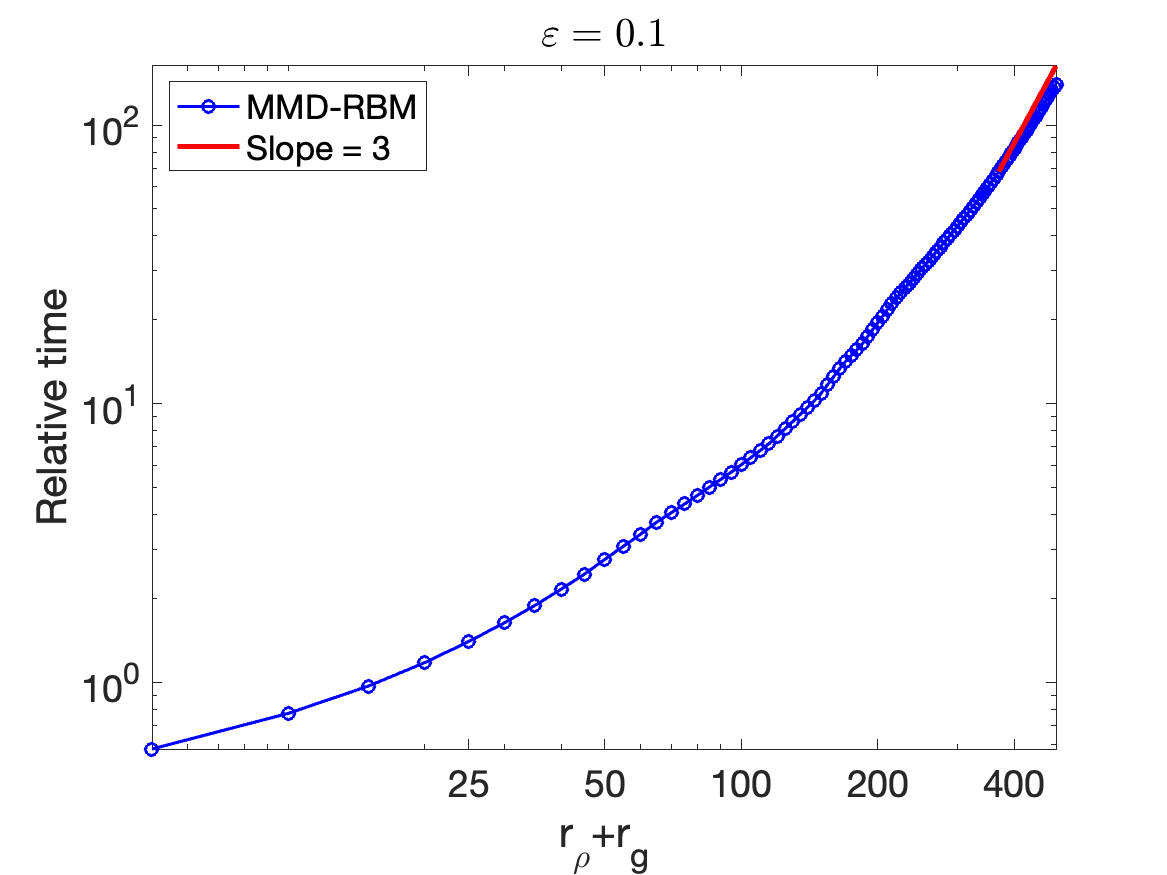}
  \caption{Relative offline computational time with respect to the reduced order $r_\rho+r_g$ for the homogeneous media example. Note the computational time is normalized by the full order solve in each case. Bottom right: 100 greedy iterations; Others: 50 greedy iterations. \label{fig:example1-offline}}
  \end{center}
\end{figure}
\medskip
\noindent{\bf The cost of the Offline stage:} In Figure \ref{fig:example1-offline}, the   offline computational time of our MMD-RBM is reported along with  the computational time of $\textrm{FOM}(\mathcal{V}_\textrm{train})$ and a vanilla POD strategy that computes the SVD of all the snapshots from   $\textrm{FOM}(\mathcal{V}_\textrm{train})$. 
All reported times are normalized by that of the full order solve in each case. Here, for comparison purpose, we implement the offline algorithm with $50$ or $100$ greedy iterations even though 
the stopping criteria are satisfied much sooner. For the first $50$ iterations, we see that the offline computational time of the MMD-RBM scales roughly as $r^2$ (with $r=r_\rho+r_g$) which is faster than the $O(r^3)$ cost suggested by \eqref{eq:offline-cost}. As shown in the bottom right picture of Figure \ref{fig:example1-offline}, the offline cost transitions from $O(r^2)$ to $O(r^3)$ as greedy procedure continues to 100 iterations, and it eventually scales slightly close to $O(r^3)$. We also label the location, via a vertical line, when the stopping criteria 
%$\textrm{tol}_\textrm{ratio}$ as $1\mathrm{e-4}$  and $\textrm{tol}_{\textrm{error},\rho}=1.0\%$, $\textrm{tol}_{\textrm{error},f}=2.0\%$ 
are satisfied. For all $\varepsilon$'s, the offline cost of our method is smaller than the cost of vanilla POD. Moreover, for $\vareps=0.005$, it is even smaller than the time of $\textrm{FOM}(\mathcal{V}_\textrm{train}).$ This shows the effectiveness of the greedy RB procedure in producing a low rank numerical solver. 
%\pzc{In summary, though theoretically offline cost scales as $O(r^3)$ the worst scenario, this cost may not be necessary in practice.}
%for diffusive cases, our method can be viewed as an efficient. %due to the use of reduced quadrature rules offline. 

%%%%%%%%%%%%%%%%%%%%%%%%%%%%%%%%%%%%%%%%%%%%%%%%%%%%%%%%%%%%%%%%%%%%%%%

\subsection{Anisotropic initial condition \label{sec:sample_direction_test}}

\begin{figure}[]
  \begin{center} 
\includegraphics[width=0.32\textwidth]{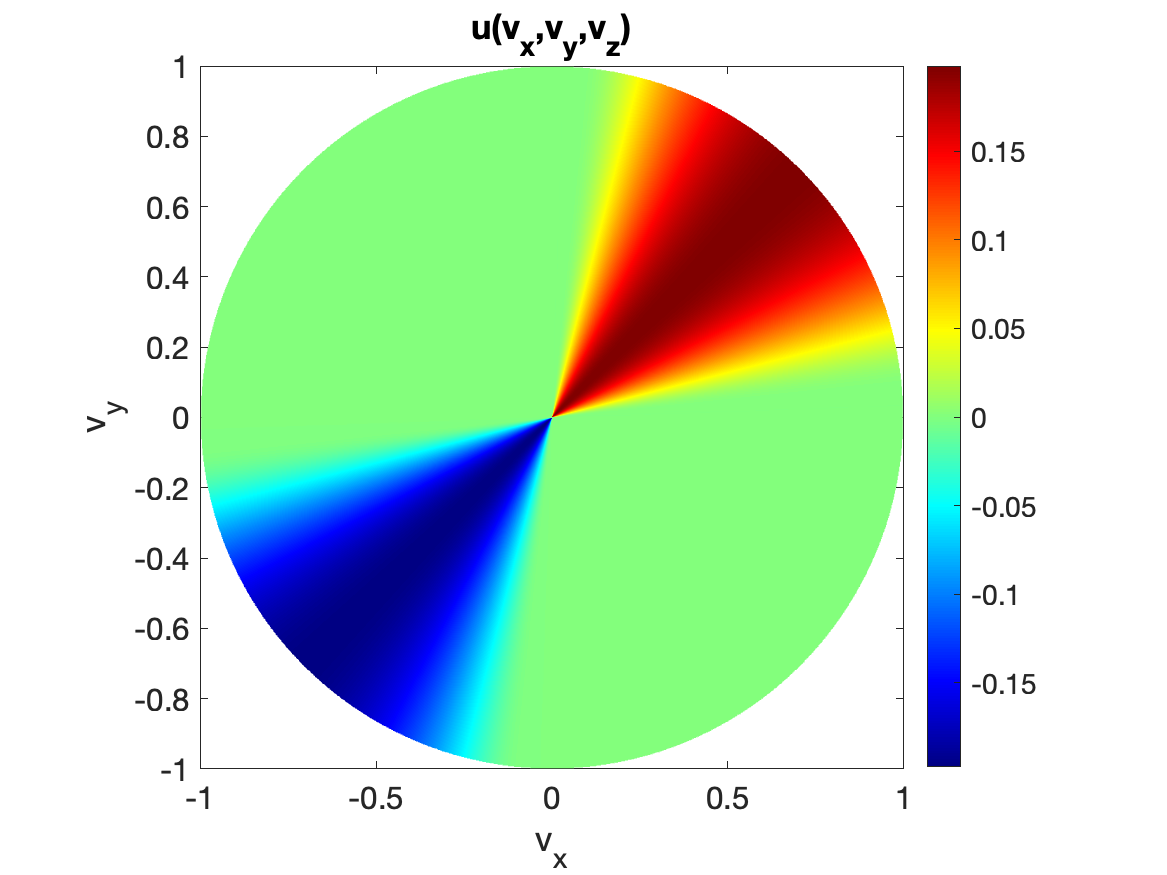}
  \includegraphics[width=0.32\textwidth]{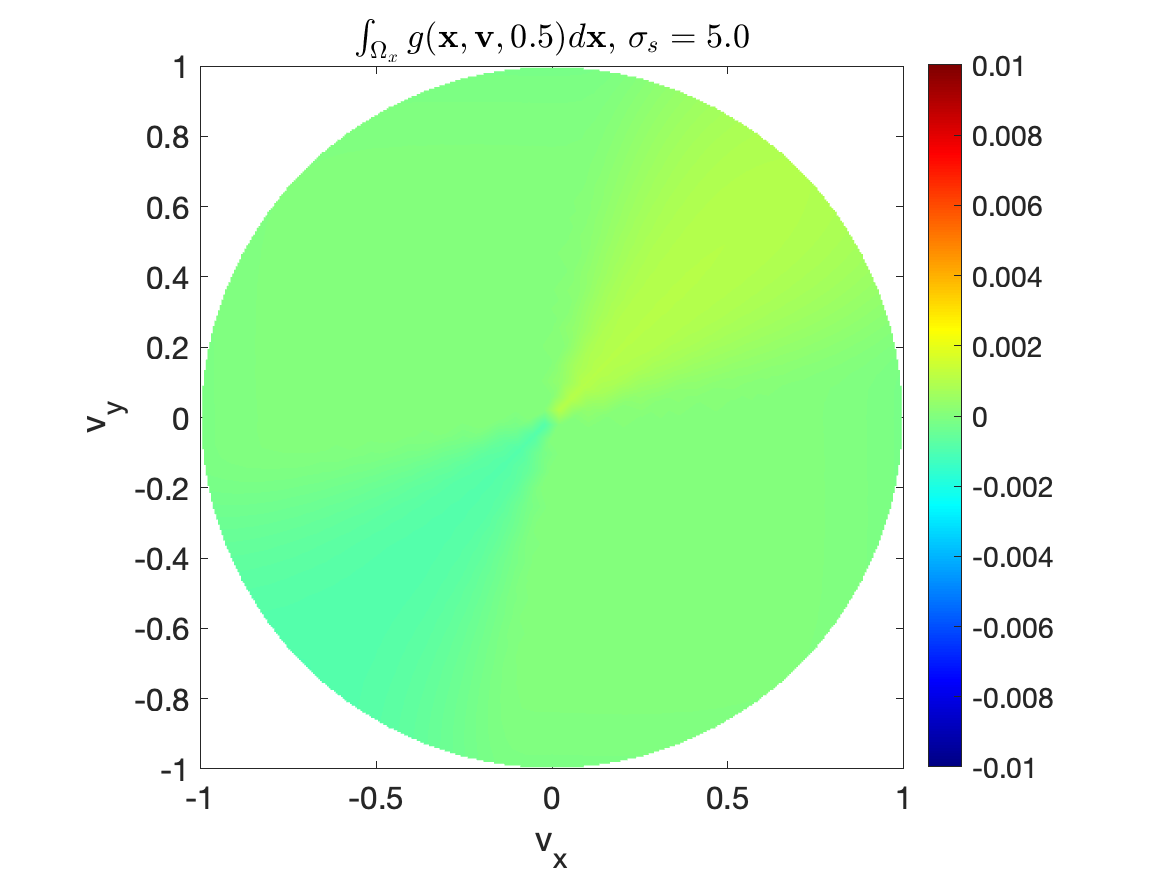}
  \includegraphics[width=0.32\textwidth]{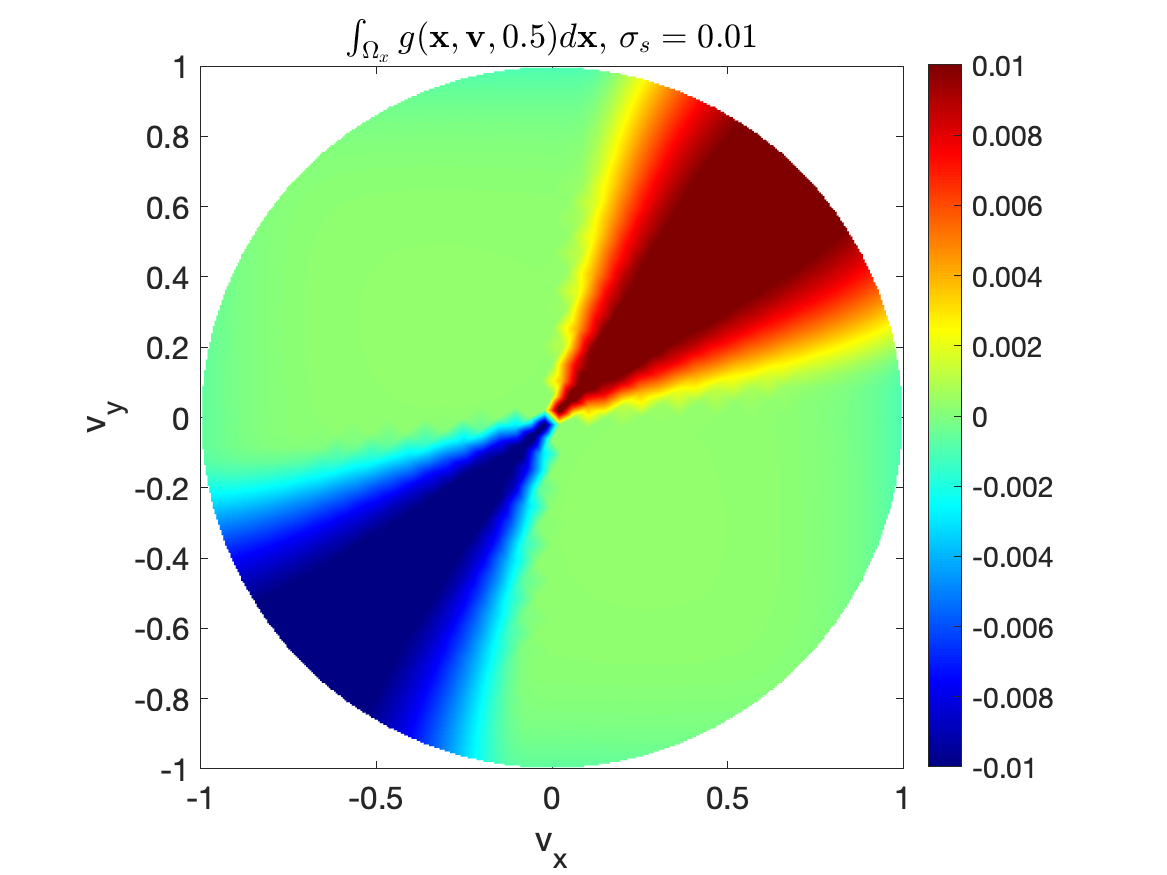}
  \includegraphics[width=0.325\textwidth,trim={0.5cm 0.1cm 0.5cm 0.1cm},clip]{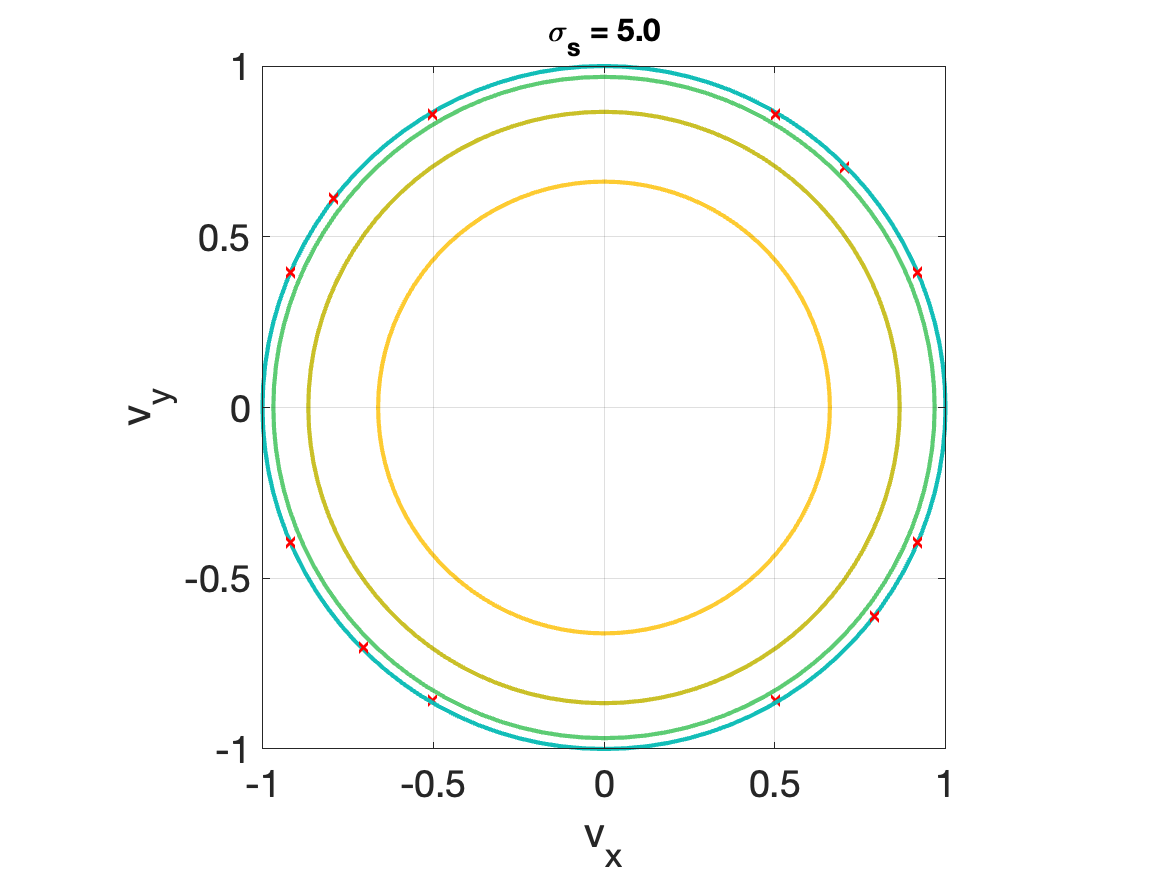}
  \includegraphics[width=0.325\textwidth,trim={0.5cm 0.1cm 0.5cm 0.1cm},clip]{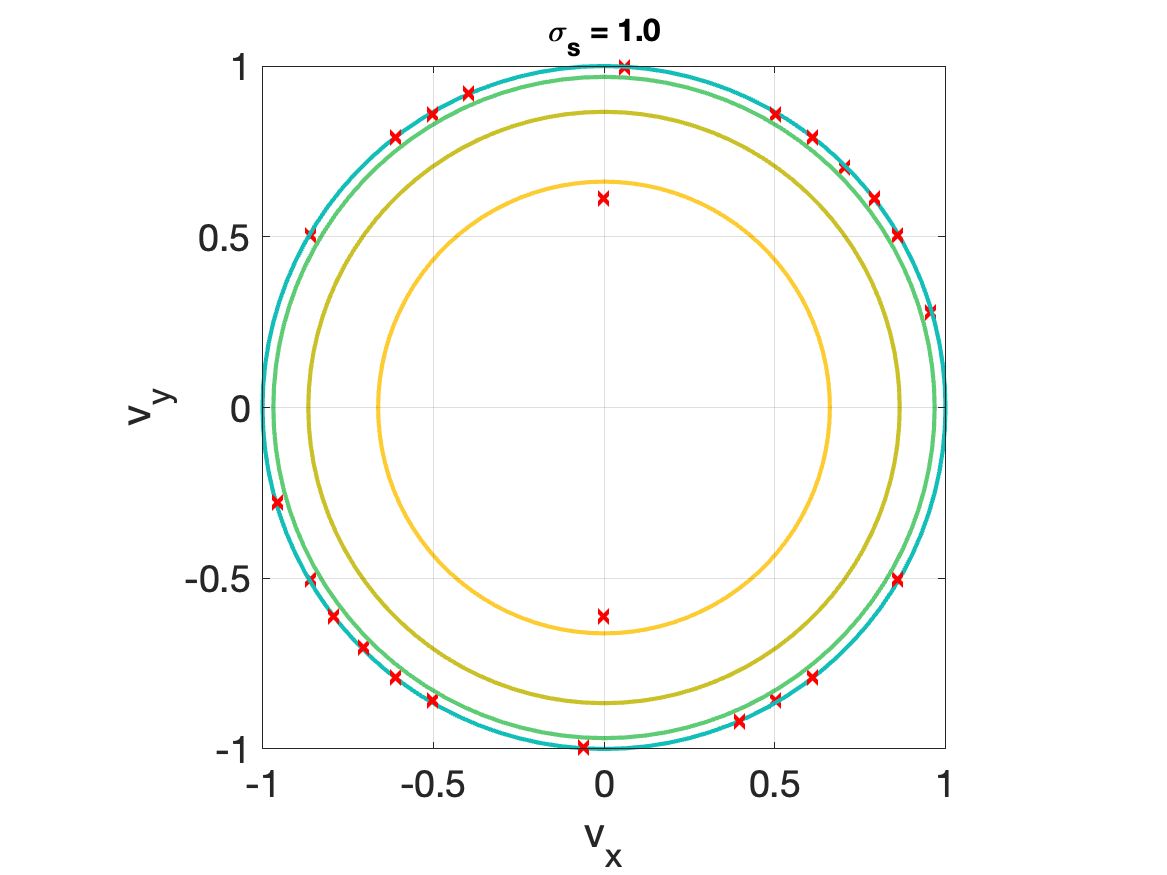}  
  \includegraphics[width=0.325\textwidth,trim={0.5cm 0.1cm 0.5cm 0.1cm},clip]{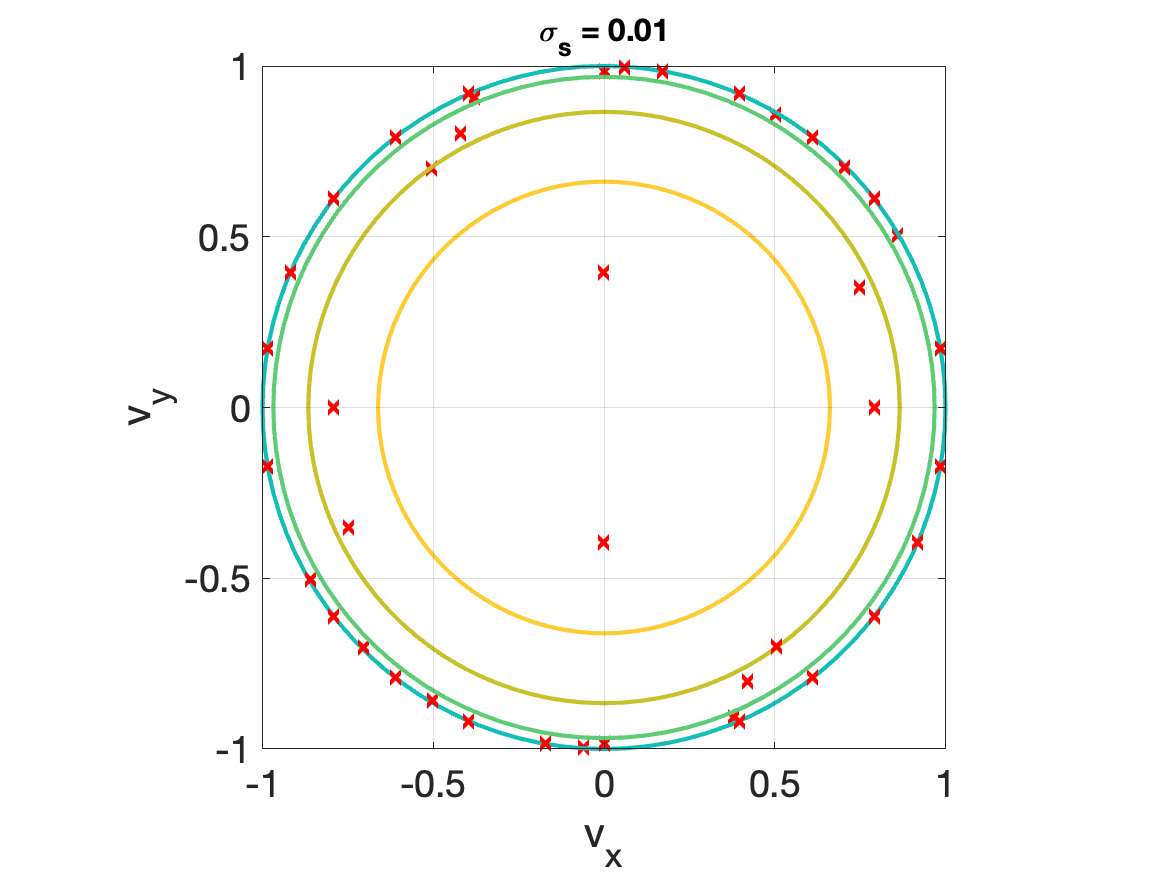}
  \caption{Shown on top are the configuration $u(\bv)$ in the initial condition of $g(\bx,\bv,0)=u(\bv)\rho(\bx,0)$ and $\int_{\Omega_x}g(\bx,\bv,0.5)d\bx$ for $\sigma_s=5.0,\;0.01$ (view from the north pole).  On the bottom are the sampled angular directions (view from the north pole) for the example in Section \ref{sec:sample_direction_test} with various $\sigma_s$ values.}
  \label{fig:example5-quad}
  \end{center}
\end{figure}
To demonstrate the ability of our method in adaptively sampling physically important angular directions, we consider the initial condition with anisotropy in the angular variable for $g$, namely, $g(x,y,\bv,0) = u(\bv(\theta,\phi))\rho(x,y,0)$ with
\[
    \rho(x,y,0) = \begin{cases}
                  \exp(-1.0/(0.5-x^2-y^2)), \quad\text{if}\; x^2+y^2<0.5,\\
                  0.0, \qquad\text{else}
                  \end{cases} \quad \mbox{and}
\]
\[
u(\bv(\theta,\phi)) = \begin{cases}
                                \;\exp\left(\frac{-1}{\frac{\pi^2}{16}-(\phi-\frac{\pi}{4})^2}
                                \right),\quad\text{if}\; v_x>0,v_y>0,\\
                                \;-\exp\left(\frac{-1}{\frac{9\pi^2}{16}-(\phi+\frac{3\pi}{4})^2}\right),\quad\text{if}\; v_x<0, v_y<0,\\
                                0.0, \qquad\text{else}.
                                 \end{cases}
\]
The computational domain is $[-1,1]^2$. The Knudsen number is $\vareps=1.0$ and the final time is $T=0.5$. As shown in the top left picture of Figure \ref{fig:example5-quad}, $u(\bv)$ in the initial condition 
$g(\bx,\bv,0)$ has more features when $v_x$ and $v_y$ are both positive or negative.  We set $\textrm{tol}_{\textrm{ratio}}=1\mathrm{e-4}$, $\textrm{tol}_{\textrm{error},\rho}=1.25\%$ and $\textrm{tol}_{\textrm{error},f}=1.25\%$. The initial reduced quadrature rule is a Lebedev quadrature with $26$ points. We consider different scattering cross sections $\sigma_s=5,\;1,\;0.01$ with zero absorption $\sigma_a=0$. Our MMD-RBM produces less than $1.44\%$ relative error when reconstructing $\rho$ online and less than $2.27\%$ relative error when predicting $f$ for unseen angular directions. In Figure \ref{fig:example5-quad}, we also present $\int_{\Omega_x}g(\bx,\bv,0.5)d\bx$ and the sampled angular directions. When $\sigma_s=5$, $\int_{\Omega_x}g(\bx,\bv,0.5)d\bx$ is almost isotropic w.r.t $\bv$ due to the strong scattering. Indeed, the sampled angular directions are more uniformly distributed. As $\sigma_s$ becomes smaller, the problem becomes more transport dominant and we observe that more angular directions are sampled
%Moreover, for $\sigma_s=1,0.1$ and $0.01$, more angular directions are sampled 
in the first and third quadrants, where $g$ has more features.  %\pzc{\sout{initial condition of}}   

%%%%%%%%%%%%%%%%%%%%%%%%%%%%%%%%%%%%%%%%%%%%%%%%%%%%%%%%%%%%%%%%%%%%%%%
\subsection{A multiscale problem with a spatially dependent scattering\label{sec:multiscale}}
Now, we consider a spatially-dependent scattering cross section \cite{einkemmer2021asymptotic}
\begin{align*}
    \sigma_s(x,y) = \begin{cases}
                0.999r^4(r+\sqrt{2})^2(r-\sqrt{2})^2+0.001,
                \;\text{with}\; r=\sqrt{x^2+y^2}<1,\\
                1,\;\text{otherwise},
                    \end{cases}\label{eq:multiscale_scattering}
\end{align*}
on the computational domain $[-1,1]^2$ with $\vareps=0.01$. The effective Knudsen number for this problem $\vareps/\sigma_s$ smoothly varies from $10$ to $0.01$ indicating a smooth transition from a transport dominant region in the center to a scattering dominant region in the outer part of the computational domain.  The initial value for this problem is 
$f(\bx,\bv,0)=\frac{5}{\pi}\exp(-25(x^2+y^2))$. We use a uniform mesh of $80\times 80$ uniform rectangular elements to partition the computational domain. 
The final time is $T=0.05$. The parameters in the stopping criteria are $\textrm{tol}_{\textrm{ratio}}=1\mathrm{e-4}$, $\textrm{tol}_{\textrm{error},\rho}=1.5\%$ and $\textrm{tol}_{\textrm{error},f}=2.5\%$. The greedy iteration is initialized with the $11$-th order $50$ points Lebedev quadrature rule.
The configuration of $\sigma_s(x,y)$, the FOM and the ROM solutions are presented on the top row of Figure \ref{fig:multiscale}. ROM solution matches the FOM solution well.  In the bottom left of Figure \ref{fig:multiscale}, the $94$  sampled angular points are presented. In the bottom right, we present the relative training error at the final time and the values of error estimators as a function of the number of greedy iterations. Overall, the error estimator provides a reasonable approximation to the relative training error at the final time. 
The errors are shown in Table \ref{tab:multiscale_mm}. It is clear that this example requires a higher rank representation for the reduced solution than the previous examples due to the large effective Knudsen number in the center region.  The MMD-RBM produces numerical solutions with relative error below $0.8\%$ for the scalar flux with only $0.27\%$ degrees of freedom in comparison to the full model.  
%%%%%%%%%%%%%%%%%%%%%%%%%%%%%%%%%%%%%%%%%%%%%%%%%%%%%%%%%%%%%%
\begin{figure}[]
  \begin{center} 
  \includegraphics[width=0.32\textwidth]{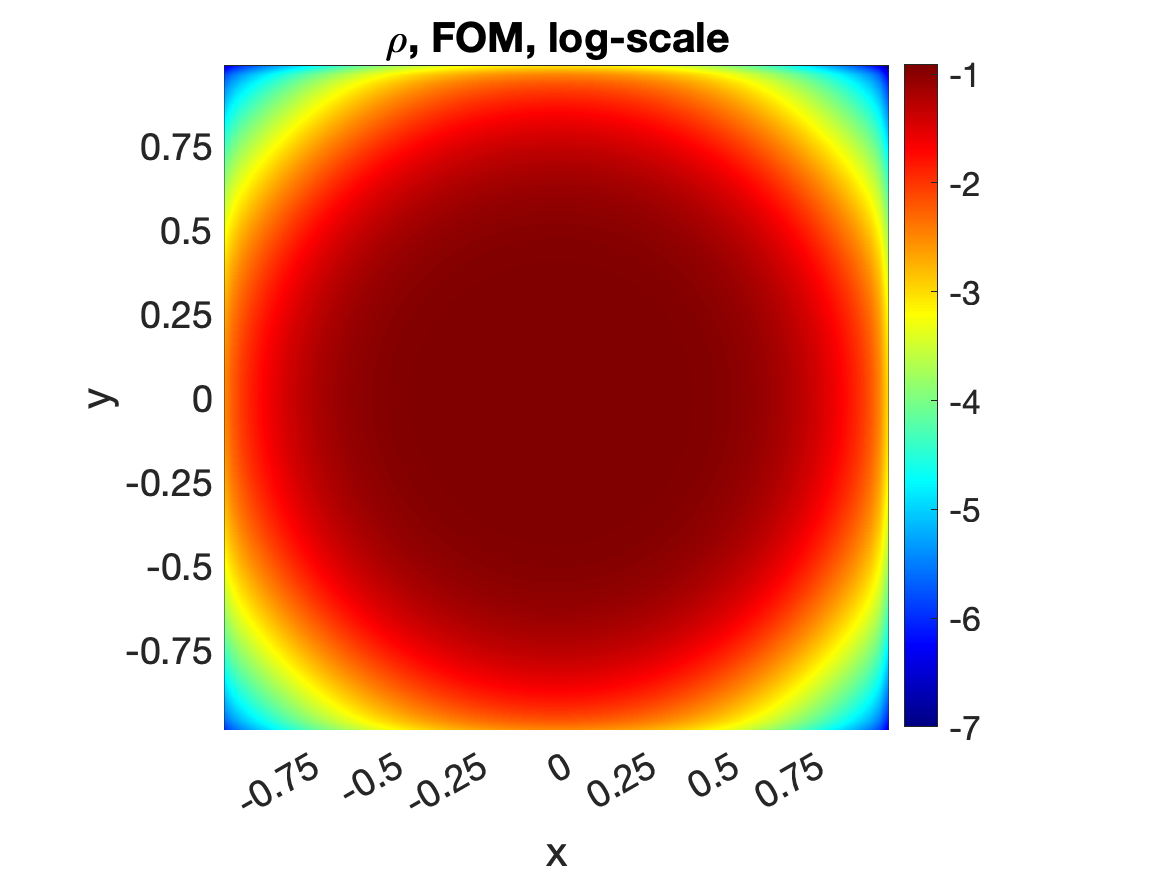}
  \includegraphics[width=0.32\textwidth]{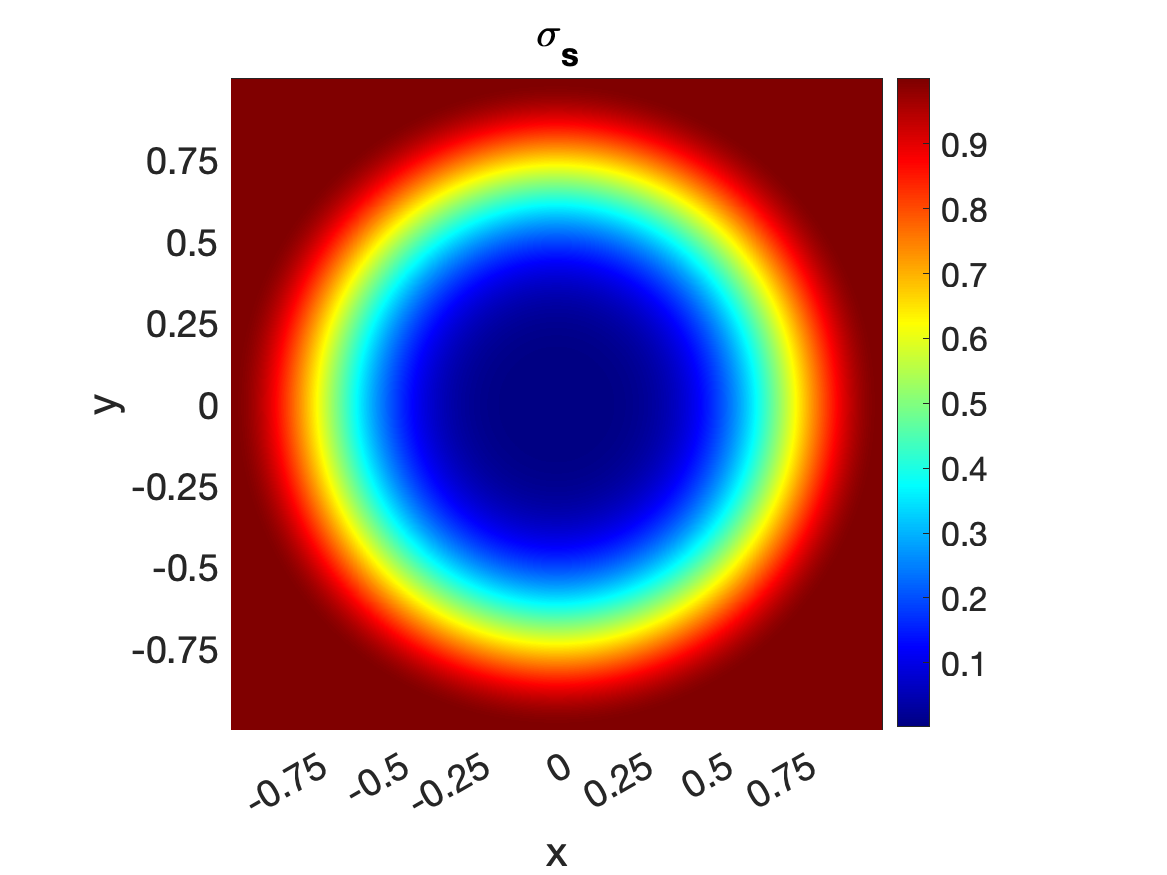}
  \includegraphics[width=0.32\textwidth]{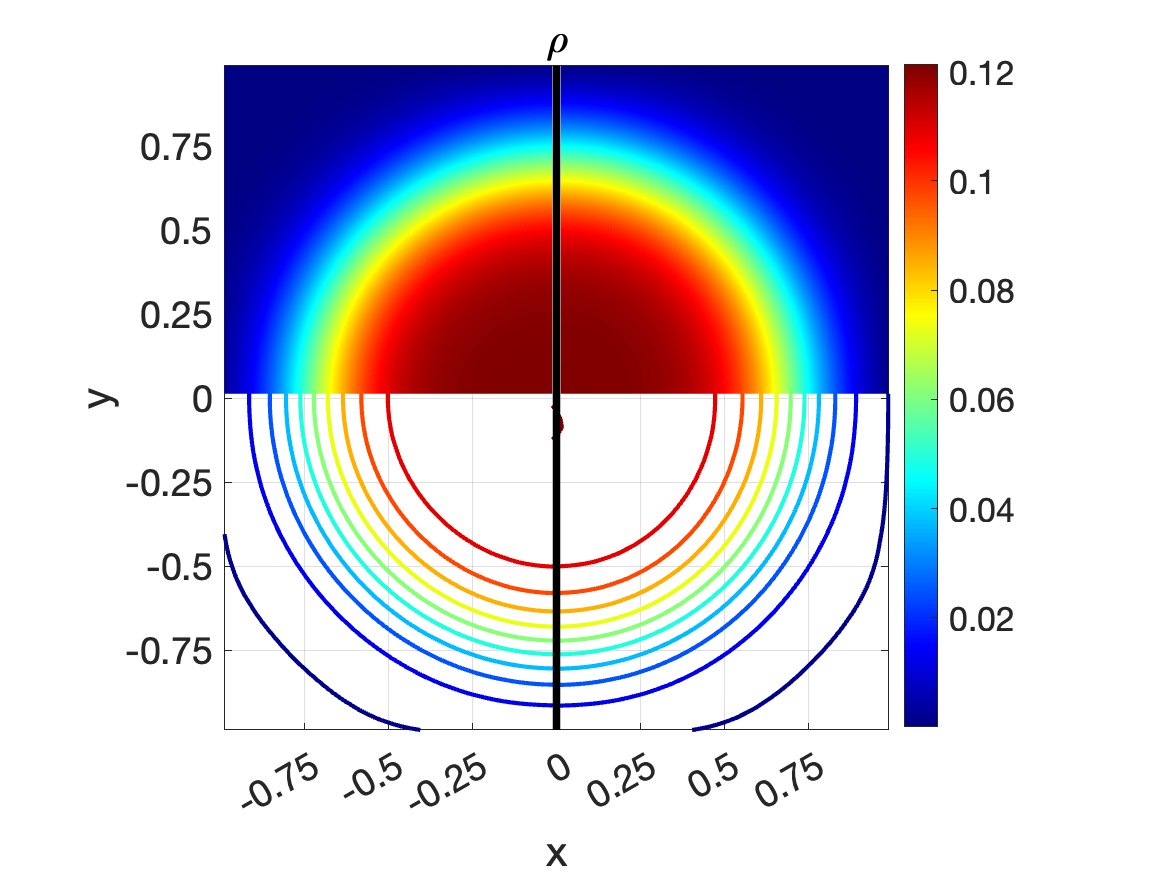}
  \includegraphics[width=0.325\textwidth,trim={2.5cm 0.1cm 2.5cm 0.1cm},clip]{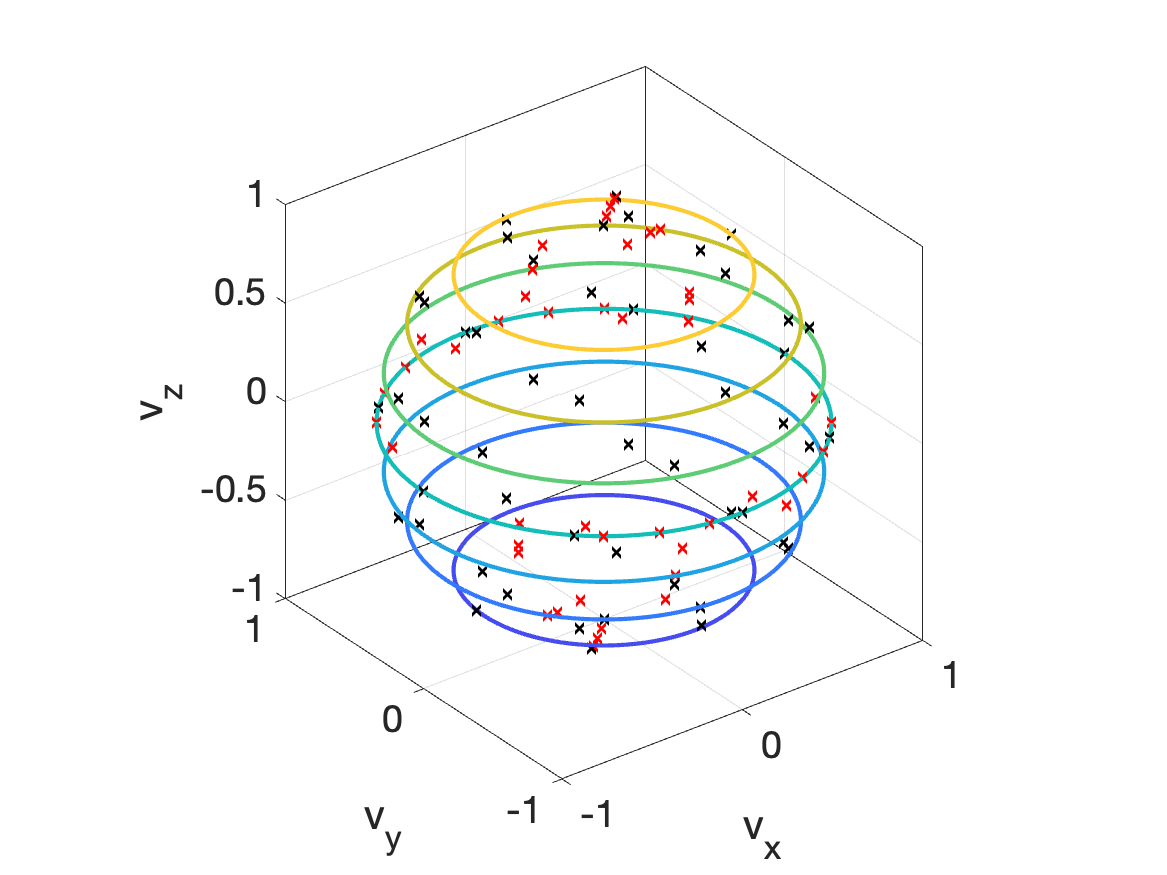}
  \includegraphics[width=0.325\textwidth,trim={2.5cm 0.1cm 2.5cm 0.1cm},clip]{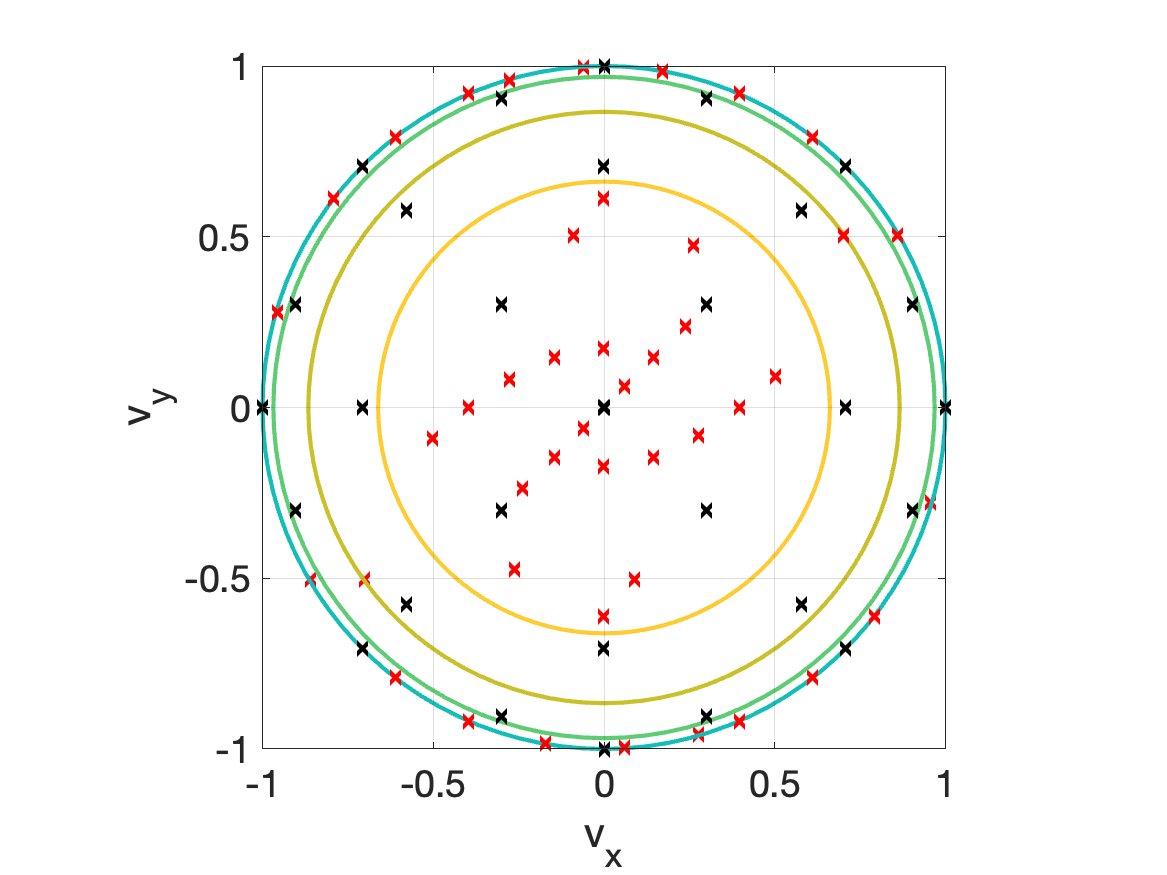}
  \includegraphics[width=0.325\textwidth,trim={2.5cm 0.1cm 1.5cm 0.5cm},clip]{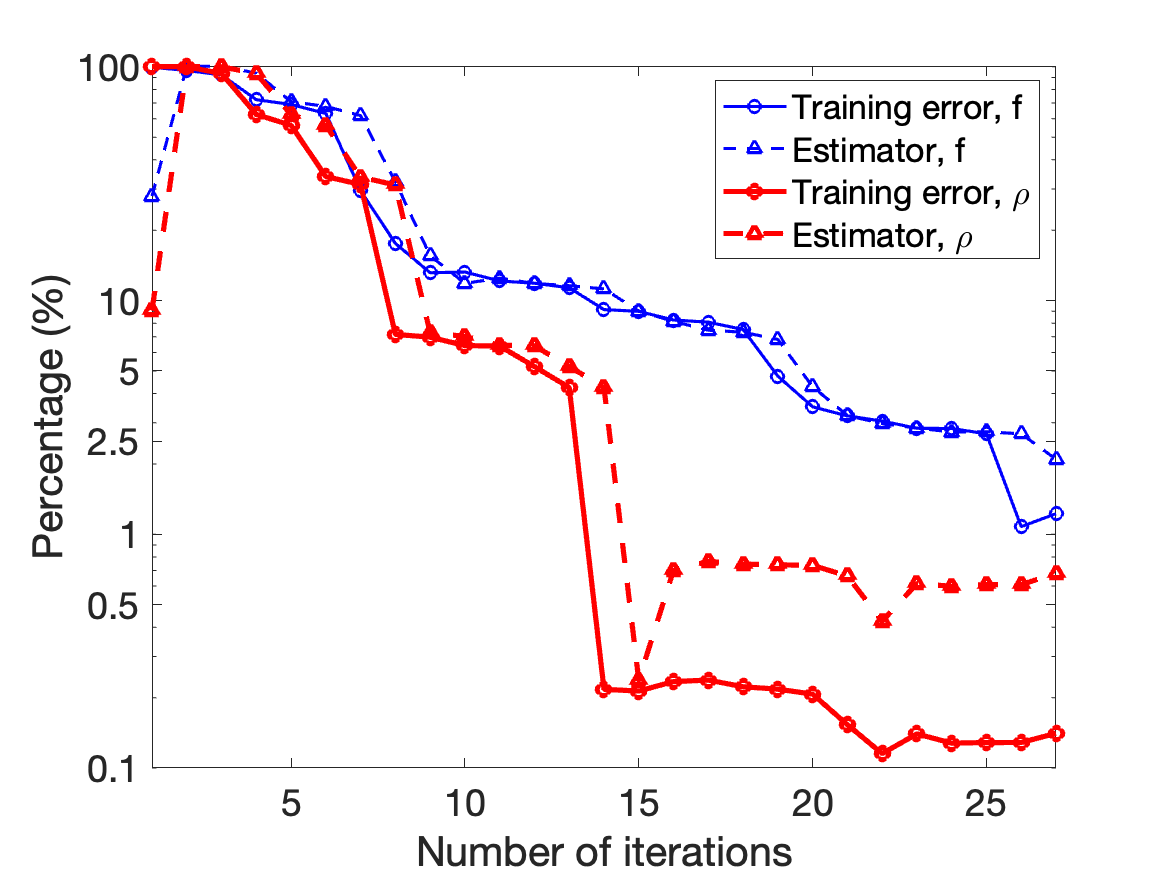}
  \caption{Results for the multiscale example. Shown on top from left to right are the FOM solution in log scale, the function 
  %setup of
  $\sigma_s$, and comparison between the ROM (right) and FOM (left) solutions. Shown on the bottom are reduced quadrature nodes on the unit sphere (Black for points in the initial reduced quadrature nodes, and Red for those sampled by the greedy algorithm), these nodes with a view from the north pole, and the history of the relative training error at the final time and the values of error estimators as a function of number of iterations. \label{fig:multiscale}}
  \end{center}
\end{figure}
%%%%%%%%%%%%%%%%%%%%%%%%%%%%%%%%%%%%%%%%%%%%%%%%%%%%%%%%%%%%%%
\begin{table}[htbp]
  \centering
 \medskip
 \begin{tabular}{|l|c|c|c|c|c|c|c|c|c|c|c|}
    \hline
     $r_\rho$ & $r_g$ & $N_v^{\textrm{rq}}$ & C-R & $\mcE_\rho$&   $\mcR_\rho$& 
 		        $\mcE_{\lgl \bv f\rgl}$ & $\mcR_{\lgl \bv f\rgl}$ &
 		        $\mcE_f$  & $\mcR_f$ \\ \hline
     $27$& $108$ & $94$ & 0.27\% & 3.00e-4 & 0.75\% &8.32e-5&  1.33\% & 1.18e-3 & 1.69\% \\ \hline			
 \end{tabular}
 \caption{Dimensions of the reduced order subspaces, $r_\rho$, $r_g$, the number of reduced quadrature nodes $N_v^{\textrm{rq}}$, the testing error and the compression ratio for the multiscale example  with the MMD-RBM.}
% \caption{Testing errors at the final time and relative online computational time \fliq{/??/} for the multiscale example with the MMD-RBM.} 
% in Section \ref{sec:multiscale} . Here, $r_\rho$ and $r_g$ are the dimensions of the reduced order subspace for $\rho$ and $g$. $N_v^{\textrm{rq}}$ is the number of nodes in the reduced quadrature rule. C-R is the compression ratio defined in \eqref{eq:compression-ratio}.
\label{tab:multiscale_mm}
\end{table}

%%%%%%%%%%%%%%%%%%%%%%%%%%%%%%%%%%%%%%%%%%%%%%%%%%%%%%%%%%%%%%%%%%%%%%%%%%
\subsection{A lattice problem\label{sec:lattice}}
The last example is a two-material lattice problem with $\vareps=1$. The geometry set-up is shown in the middle of the top row of Figure \ref{fig:lattice}. The black region is pure absorption with $\sigma_s=0$ and $\sigma_a=100$, while the rest is pure scattering 
with $\sigma_s=1$ and $\sigma_a=0$. In the orange region, a constant source is imposed:
$$
G(x,y) = \begin{cases}
         1.0,\quad\text{if}\quad |x-2.5|<0.5\;\text{and}\quad|y-2.5|<0.5,\\
         0, \quad\text{otherwise}.
         \end{cases}
$$
A uniform mesh of $100\times100$ rectangular elements is used to partition the computational domain. The final time is $T=1.7$. The tolerances in the stopping criteria are  $\textrm{tol}_\textrm{ratio}=$1e-3, $\textrm{tol}_{\textrm{error},\rho}=1.5\%$ and $\textrm{tol}_{\textrm{error},f}=3.0\%$. When initializing the RBM offline, we use the $11$-th order $50$ point  Lebedev quadrature rule. 

%%%%%%%%%%%%%%%%%%%%%%%%%%%%%%%%%%%%%%%%%%%%%%%%%%%%%%%%%%%%%%
\begin{table}[htbp]
  \centering
 \medskip
 \begin{tabular}{|l|c|c|c|c|c|c|c|c|c|c|c|c|c|}
    \hline
     $r_\rho$ & $r_g$ & $N_v^{\textrm{rq}}$ & C-R & $\mcE_\rho$&   $\mcR_\rho$& 
 		        $\mcE_{\lgl \bv f\rgl}$ & $\mcR_{\lgl \bv f\rgl}$ &
 		        $\mcE_f$  & $\mcR_f$ \\ \hline
     $31$& $124$ & $102$ & 0.21\% & 1.85e-3 & 0.27\% & 4.45e-3&  2.41\% & 2.38e-2 & 2.71\% \\ \hline		
 \end{tabular}
 \caption{Dimensions of the reduced order subspaces, $r_\rho$, $r_g$, the number of reduced quadrature nodes $N_v^{\textrm{rq}}$, the testing error and the compression ratio for the lattice example  with the MMD-RBM.}
 %\caption{Testing errors at the final time and relative online computational time \fliq{/??/} for the Lattice problem with the MMD-RBM.} 
% Here, $r_\rho$ and $r_g$ are the dimensions of the reduced order subspace for $\rho$ and $g$. $N_v^{\textrm{rq}}$ is the number of nodes in the reduced quadrature rule. C-R is the compression ratio defined in \eqref{eq:compression-ratio}.
 \label{tab:example3_mm}
\end{table}
%%%%%%%%%%%%%%%%%%%%%%%%%%%%%%%%%%%%%%%%%%%%%%%%%%%%%%%%%%%%%%%%%%%%%%%
We present the ROM and FOM solutions on the top row of Figure \ref{fig:lattice}. Shown on the bottom are the $102$ nodes of the reduced quadrature rule and the history of the relative training error at the final time and the values of error estimators.  Our error estimators approximate
%s
the relative errors at the final time well and the MMD-RBM solution matches the FOM well.
%Under the log scale, the ROM solution capture the structure of the FOM solution better in the center part of the computational domain compared with the outer part. 
The errors are displayed in Table \ref{tab:example3_mm}. We see that the ROM achieves $0.27\%$ relative error for $\rho$ with $0.21\%$ DOFs w.r.t FOM($\mcV_{\textrm{train}}$), while the relative errors  $\lgl \bv f\rgl$ and $f$ on the test set stay%s 
 about $2\%$ to $3\%$. 
%%%%%%%%%%%%%%%%%%%%%%%%%%%%%%%%%%%%%%%%%%%%%%%%%%%%%%%%%%%%%%%%%%%
\begin{figure}[]
  \begin{center} 
  \includegraphics[width=0.325\textwidth]{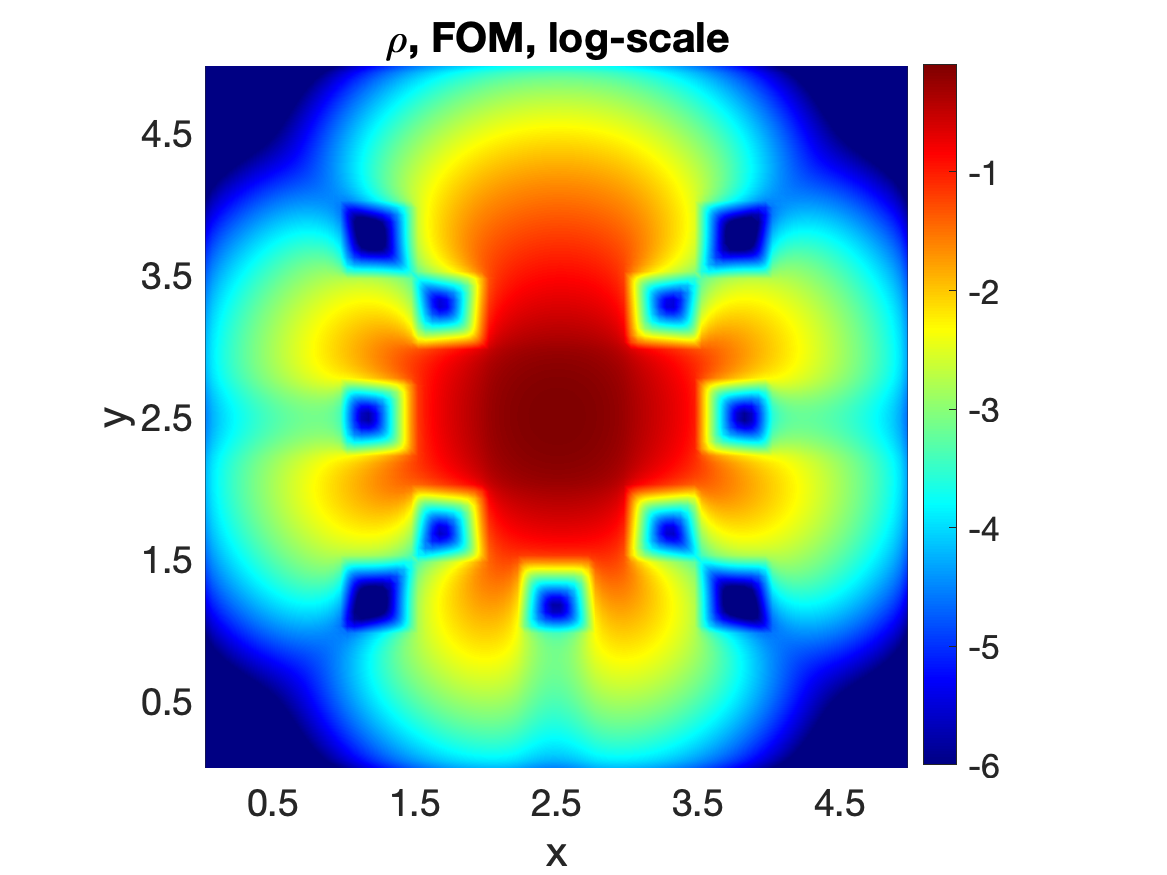}
  \includegraphics[width=0.325\textwidth]{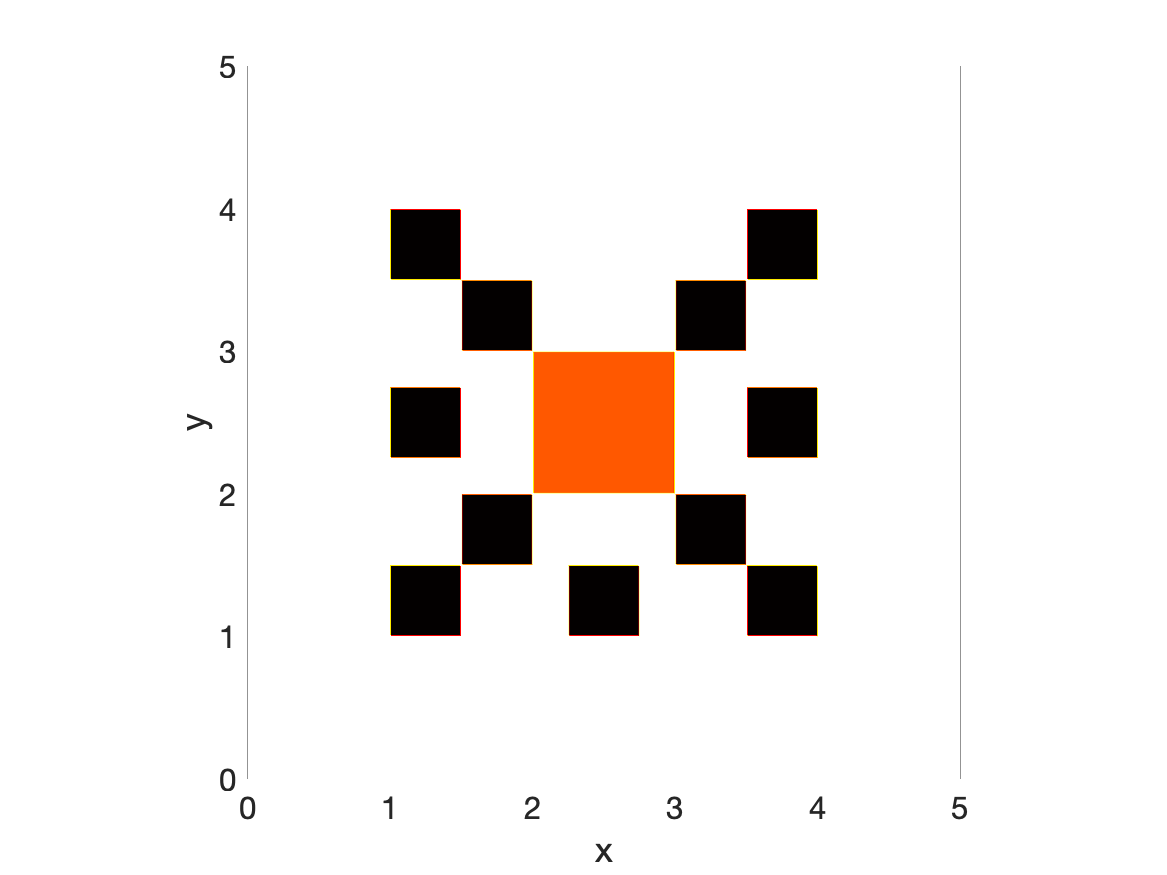}
  \includegraphics[width=0.325\textwidth]{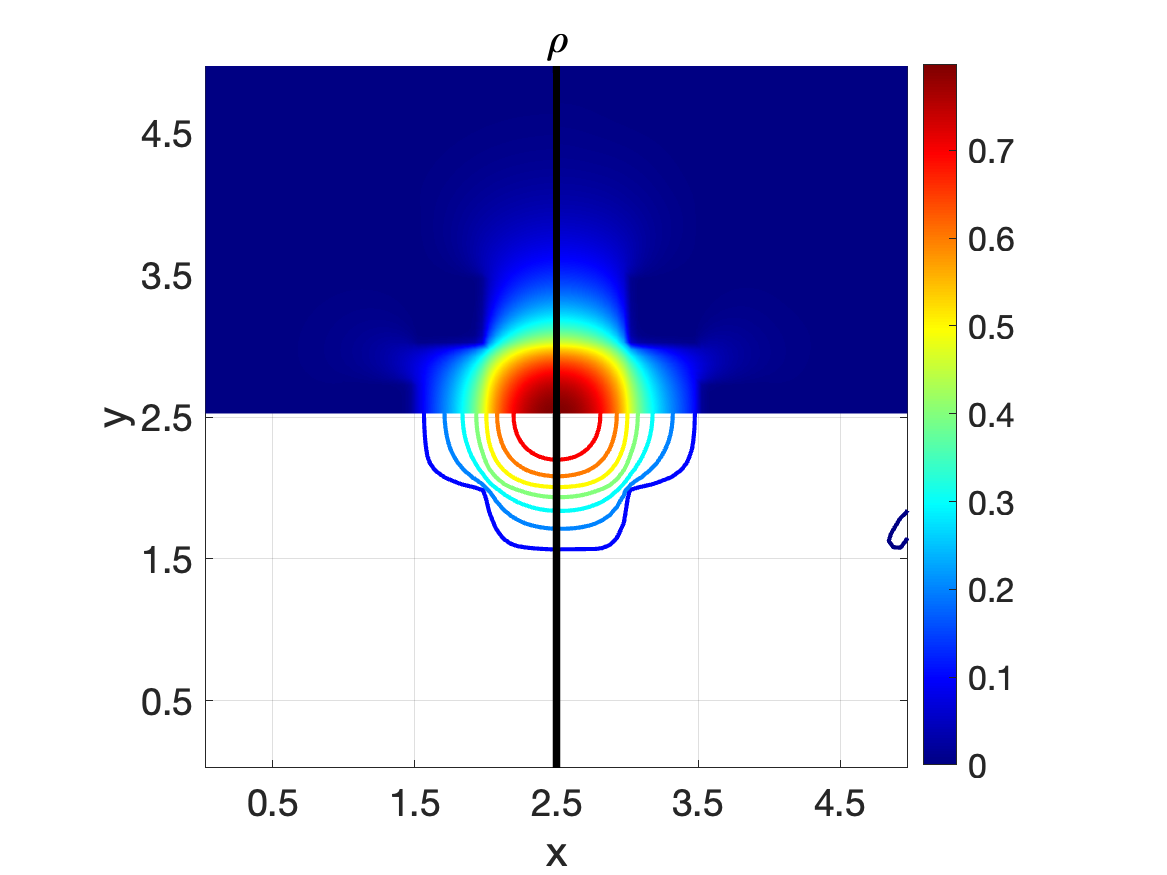}
  \includegraphics[width=0.325\textwidth,trim={2.5cm 0.1cm 2.5cm 0.1cm},clip]{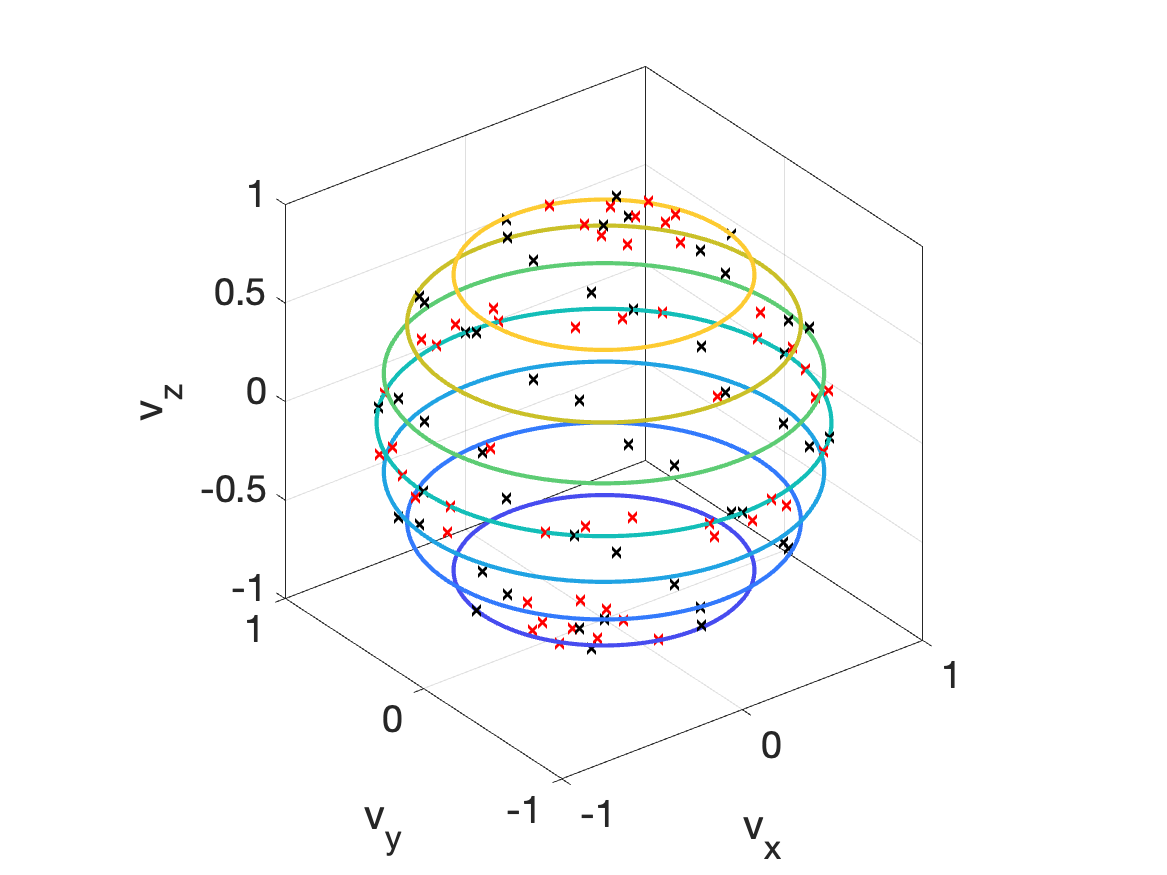}
  \includegraphics[width=0.325\textwidth,trim={2.5cm 0.1cm 2.5cm 0.1cm},clip]{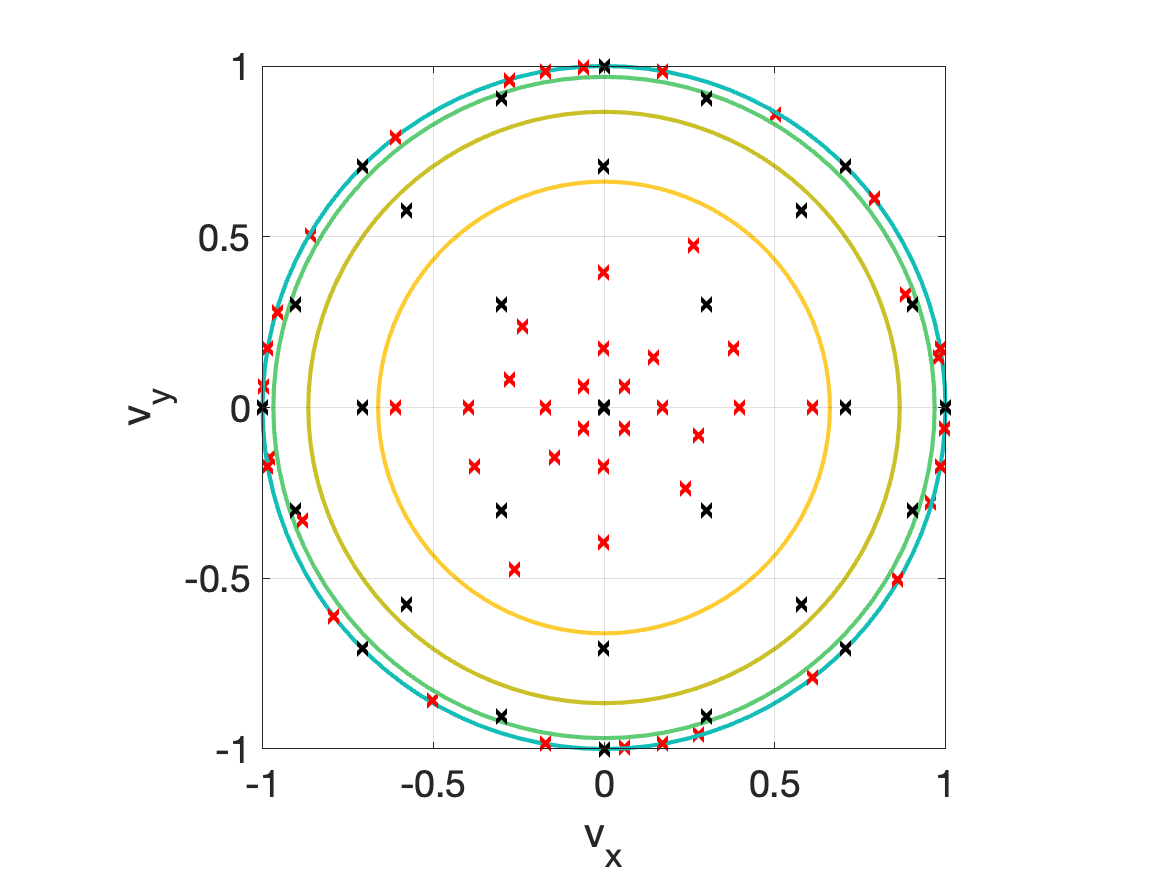}
  \includegraphics[width=0.325\textwidth,trim={2.5cm 0.1cm 1.5cm 0.5cm},clip]{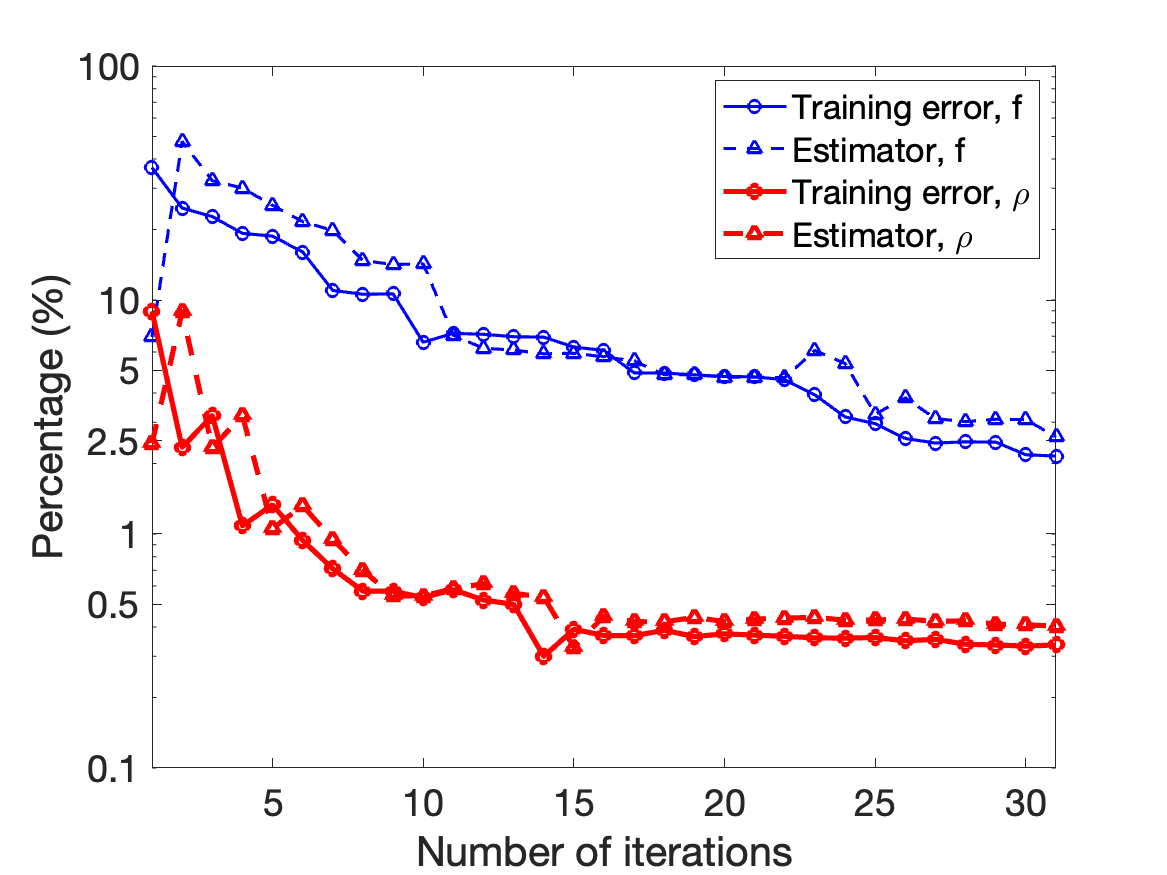}
  \caption{Results for the lattice problem. Shown on top from left to right are FOM solution in the log scale, the domain setup (Black for pure absorption, White for pure scattering, Orange for a constant source and $\sigma_s=1$, $\sigma_a=0$), and comparison between the FOM (left) and ROM (right) solution.  Shown on the bottom are reduced quadrature nodes on the unit sphere (Black for points in the initial reduced quadrature nodes, and Red for those sampled by the greedy algorithm), these nodes with a view from the north pole, and the history of the relative training error at the final time and the values of error estimators as a function of number of iterations.
   \label{fig:lattice}}
  \end{center}
\end{figure}

%%%%%%%%%%%%%%%%%%%%%%%%%%%%%%%%%%%%%%%%%%%%%%%%%%%%%%%%%%%%%%%%%%%
\section{Conclusion \label{sec:conclusions}}
In this paper, utilizing low rank structures with respect to the angular direction $\bv$ and the temporal variable $t$, we developed a novel RBM to construct ROM for the time-dependent RTE based on the micro-macro decomposition. The proposed MMD-RBM is featured by an equilibrium-respecting strategy to construct reduced order subspaces and a reduced quadrature rule with non-negative weights preserving the stability of the underlying numerical solver. As demonstrated by our numerical tests, the Offline stage of the proposed method is more efficient than the vanilla POD method and sometimes even the standard full order solve, and the Online stage is able to efficiently predict angular fluxes for unseen angular directions and reconstruct the moments of the angular flux. The natural next step along this work is to use the proposed method as a building block to design  ROMs for multi-query scenarios (e.g. inverse problems and uncertainty quantification) with essential physical parameters. %and the proposed RBMs can be utilized as a building block to better utilize the low rank structure with respect to the angular and temporal variable. 

\bibliographystyle{plain}
\bibliography{ref}

\begin{thebibliography}{10}

\bibitem{alberti2020reduced}
Anthony~L Alberti and Todd~S Palmer.
\newblock {Reduced-order modeling of nuclear reactor kinetics using proper
  generalized decomposition}.
\newblock {\em Nuclear Science and Engineering}, 194(10):837--858, 2020.

\bibitem{arridge2009optical}
Simon~R Arridge and John~C Schotland.
\newblock {Optical tomography: forward and inverse problems}.
\newblock {\em Inverse problems}, 25(12):123010, 2009.

\bibitem{behne2021model}
Patrick~A Behne, Jean~C Ragusa, and Jim~E Morel.
\newblock {Model order reduction for Sn radiation transport}.
\newblock {\em Nuclear Science and Engineering}, 2021.

\bibitem{brand2002incremental}
Matthew Brand.
\newblock {Incremental singular value decomposition of uncertain data with
  missing values}.
\newblock In {\em European Conference on Computer Vision}, pages 707--720.
  Springer, 2002.

\bibitem{buchan2015pod}
Andrew~G Buchan, AA~Calloo, Mark~G Goffin, Steven Dargaville, Fangxin Fang,
  Christopher~C Pain, and Ionel~Michael Navon.
\newblock {A POD reduced order model for resolving angular direction in
  neutron/photon transport problems}.
\newblock {\em Journal of Computational Physics}, 296:138--157, 2015.

\bibitem{caflisch1997uniformly}
Russel~E Caflisch, Shi Jin, and Giovanni Russo.
\newblock Uniformly accurate schemes for hyperbolic systems with relaxation.
\newblock {\em SIAM Journal on Numerical Analysis}, 34(1):246--281, 1997.

\bibitem{castillo2000priori}
Paul Castillo, Bernardo Cockburn, Ilaria Perugia, and Dominik Sch{\"o}tzau.
\newblock {An a priori error analysis of the local discontinuous Galerkin
  method for elliptic problems}.
\newblock {\em SIAM Journal on Numerical Analysis}, 38(5):1676--1706, 2000.

\bibitem{chen2021eim}
Yanlai Chen, Sigal Gottlieb, Lijie Ji, and Yvon Maday.
\newblock An eim-degradation free reduced basis method via over collocation and
  residual hyper reduction-based error estimation.
\newblock {\em Journal of Computational Physics}, 444:110545, 2021.

\bibitem{chen2021l1}
Yanlai Chen, Lijie Ji, Akil Narayan, and Zhenli Xu.
\newblock {L1-based reduced over collocation and hyper reduction for steady
  state and time-dependent nonlinear equations}.
\newblock {\em Journal of Scientific Computing}, 87(1):1--21, 2021.

\bibitem{choi2021space}
Youngsoo Choi, Peter Brown, William Arrighi, Robert Anderson, and Kevin Huynh.
\newblock {Space--time reduced order model for large-scale linear dynamical
  systems with application to boltzmann transport problems}.
\newblock {\em Journal of Computational Physics}, 424:109845, 2021.

\bibitem{coale2019reduced}
Joseph Coale and Dmitriy~Y Anistratov.
\newblock {A reduced-order model for thermal radiative transfer problems based
  on multilevel quasidiffusion method}.
\newblock In {\em International Conference on Mathematics and Computational
  Methods Applied to Nuclear Science and Engineering, M and C}, volume 2019,
  pages 278--287, 2019.

\bibitem{coale2019reduced2}
Joseph~Michael Coale.
\newblock {Reduced Order Models for Thermal Radiative Transfer Problems Based
  on Low-Order Transport Equations and the Proper Orthogonal Decomposition}.
\newblock 2019.

\bibitem{dominesey2022reduced}
Kurt~A Dominesey and Wei Ji.
\newblock {Reduced-order modeling of neutron transport separated in space and
  angle via proper generalized decomposition}.
\newblock {\em Nuclear Science and Engineering}, 2022.

\bibitem{einkemmer2021asymptotic}
Lukas Einkemmer, Jingwei Hu, and Yubo Wang.
\newblock {An asymptotic-preserving dynamical low-rank method for the
  multi-scale multi-dimensional linear transport equation}.
\newblock {\em Journal of Computational Physics}, 439:110353, 2021.

\bibitem{fornberg2014spherical}
Bengt Fornberg and Jordan~M Martel.
\newblock {On spherical harmonics based numerical quadrature over the surface
  of a sphere}.
\newblock {\em Advances in Computational Mathematics}, 40(5):1169--1184, 2014.

\bibitem{golub2013matrix}
Gene~H Golub and Charles~F Van~Loan.
\newblock {\em Matrix computations}.
\newblock JHU press, 2013.

\bibitem{haasdonk2017reduced}
Bernard Haasdonk.
\newblock {Reduced basis methods for parametrized PDEs--a tutorial introduction
  for stationary and instationary problems}.
\newblock {\em Model reduction and approximation: theory and algorithms},
  15:65, 2017.

\bibitem{hartmann2003adaptive}
Ralf Hartmann and Paul Houston.
\newblock {Adaptive discontinuous Galerkin finite element methods for nonlinear
  hyperbolic conservation laws}.
\newblock {\em SIAM Journal on Scientific Computing}, 24(3):979--1004, 2003.

\bibitem{hughes2022adaptive}
Alexander~C Hughes and Andrew~G Buchan.
\newblock {An adaptive reduced order model for the angular discretization of
  the Boltzmann transport equation using independent basis sets over a
  partitioning of the space-angle domain}.
\newblock {\em International Journal for Numerical Methods in Engineering},
  2022.

\bibitem{jang2015high}
Juhi Jang, Fengyan Li, Jing-Mei Qiu, and Tao Xiong.
\newblock {High order asymptotic preserving DG-IMEX schemes for
  discrete-velocity kinetic equations in a diffusive scaling}.
\newblock {\em Journal of Computational Physics}, 281:199--224, 2015.

\bibitem{jin2010asymptotic}
Shi Jin.
\newblock {Asymptotic preserving (AP) schemes for multiscale kinetic and
  hyperbolic equations: a review}.
\newblock {\em Lecture notes for summer school on methods and models of kinetic
  theory (M\&MKT), Porto Ercole (Grosseto, Italy)}, pages 177--216, 2010.

\bibitem{lebedev1976quadratures}
Vyacheslav~Ivanovich Lebedev.
\newblock {Quadratures on a sphere}.
\newblock {\em USSR Computational Mathematics and Mathematical Physics},
  16(2):10--24, 1976.

\bibitem{lemou2008new}
Mohammed Lemou and Luc Mieussens.
\newblock {A new asymptotic preserving scheme based on micro-macro formulation
  for linear kinetic equations in the diffusion limit}.
\newblock {\em SIAM Journal on Scientific Computing}, 31(1):334--368, 2008.

\bibitem{lewis1984computational}
Elmer~Eugene Lewis and Warren~F Miller.
\newblock {Computational methods of neutron transport}.
\newblock 1984.

\bibitem{liu2004boltzmann}
Tai-Ping Liu and Shih-Hsien Yu.
\newblock {Boltzmann equation: micro-macro decompositions and positivity of
  shock profiles}.
\newblock {\em Communications in Mathematical Physics}, 246(1):133--179, 2004.

\bibitem{mcclarren2018acceleration}
Ryan~G McClarren and Terry~S Haut.
\newblock {Acceleration of source iteration using the dynamic mode
  decomposition}.
\newblock {\em arXiv preprint arXiv:1812.05241}, 2018.

\bibitem{mcclarren2022data}
Ryan~G McClarren and Terry~S Haut.
\newblock {Data-driven acceleration of thermal radiation transfer calculations
  with the dynamic mode decomposition and a sequential singular value
  decomposition}.
\newblock {\em Journal of Computational Physics}, 448:110756, 2022.

\bibitem{naldi1998numerical}
Giovanni Naldi and Lorenzo Pareschi.
\newblock {Numerical schemes for kinetic equations in diffusive regimes}.
\newblock {\em Applied mathematics letters}, 11(2):29--35, 1998.

\bibitem{patera2007reduced}
Anthony~T Patera, Gianluigi Rozza, et~al.
\newblock {Reduced basis approximation and a posteriori error estimation for
  parametrized partial differential equations}, 2007.

\bibitem{peng2022reduced}
Zhichao Peng, Yanlai Chen, Yingda Cheng, and Fengyan Li.
\newblock {A reduced basis method for radiative transfer equation}.
\newblock {\em Journal of Scientific Computing}, 91(1):1--27, 2022.

\bibitem{peng2020stability}
Zhichao Peng, Yingda Cheng, Jing-Mei Qiu, and Fengyan Li.
\newblock {Stability-enhanced AP IMEX-LDG schemes for linear kinetic transport
  equations under a diffusive scaling}.
\newblock {\em Journal of Computational Physics}, 415:109485, 2020.

\bibitem{peng2021stability}
Zhichao Peng, Yingda Cheng, Jing-Mei Qiu, and Fengyan Li.
\newblock {Stability-enhanced AP IMEX1-LDG method: energy-based stability and
  rigorous AP property}.
\newblock {\em SIAM Journal on Numerical Analysis}, 59(2):925--954, 2021.

\bibitem{peng2021asymptotic}
Zhichao Peng and Fengyan Li.
\newblock {Asymptotic preserving IMEX-DG-S schemes for linear kinetic transport
  equations based on Schur complement}.
\newblock {\em SIAM Journal on Scientific Computing}, 43(2):A1194--A1220, 2021.

\bibitem{peng2021high}
Zhuogang Peng and Ryan~G McClarren.
\newblock {A high-order/low-order (HOLO) algorithm for preserving conservation
  in time-dependent low-rank transport calculations}.
\newblock {\em Journal of Computational Physics}, 447:110672, 2021.

\bibitem{peng2020low}
Zhuogang Peng, Ryan~G McClarren, and Martin Frank.
\newblock {A low-rank method for two-dimensional time-dependent radiation
  transport calculations}.
\newblock {\em Journal of Computational Physics}, 421:109735, 2020.

\bibitem{pomraning1973equations}
Gerald~C. Pomraning.
\newblock {The equations of radiation hydrodynamics}.
\newblock {\em International Series of Monographs in Natural Philosophy,
  Oxford: Pergamon Press}, 1973.

\bibitem{prince2020space}
Zachary~M Prince and Jean~C Ragusa.
\newblock {Space-energy separated representations for multigroup neutron
  diffusion using proper generalized decompositions}.
\newblock {\em Annals of Nuclear Energy}, 142:107360, 2020.

\bibitem{rozza2008reduced}
Gianluigi Rozza, Dinh Bao~Phuong Huynh, and Anthony~T Patera.
\newblock {Reduced basis approximation and a posteriori error estimation for
  affinely parametrized elliptic coercive partial differential equations}.
\newblock {\em Archives of Computational Methods in Engineering},
  15(3):229--275, 2008.

\bibitem{spurr2001linearized}
RJD Spurr, TP~Kurosu, and KV~Chance.
\newblock {A linearized discrete ordinate radiative transfer model for
  atmospheric remote-sensing retrieval}.
\newblock {\em Journal of Quantitative Spectroscopy and Radiative Transfer},
  68(6):689--735, 2001.

\bibitem{tano2021affine}
Mauricio Tano, Jean Ragusa, Dominic Caron, and Patrick Behne.
\newblock {Affine reduced-order model for radiation transport problems in
  cylindrical coordinates}.
\newblock {\em Annals of Nuclear Energy}, 158:108214, 2021.

\bibitem{tencer2016reduced}
John Tencer, Kevin Carlberg, Roy Hogan, and Marvin Larsen.
\newblock {Reduced order modeling applied to the discrete ordinates method for
  radiation heat transfer in participating media}.
\newblock In {\em Heat Transfer Summer Conference}, volume 50336, page
  V002T15A011. American Society of Mechanical Engineers, 2016.

\end{thebibliography}

\end{document}